\documentclass[oneside,english,american]{amsart}
\usepackage{mathptmx}

\usepackage[T1]{fontenc}
\usepackage[latin9]{inputenc}
\usepackage{geometry}
\usepackage{xcolor}
\usepackage{float}
\usepackage{amsbsy}
\usepackage{amstext}
\usepackage{amsthm}
\usepackage{amssymb}
\usepackage{graphicx}
\PassOptionsToPackage{normalem}{ulem}
\usepackage{ulem}

\makeatletter

\providecommand{\tabularnewline}{\\}
\floatstyle{ruled}
\newfloat{algorithm}{tbp}{loa}
\providecommand{\algorithmname}{Algorithm}
\floatname{algorithm}{\protect\algorithmname}
\providecolor{lyxadded}{rgb}{0,0,1}
\providecolor{lyxdeleted}{rgb}{1,0,0}

\DeclareRobustCommand{\lyxsout}[1]{\ifx\\#1\else\sout{#1}\fi}

\numberwithin{equation}{section}
\numberwithin{figure}{section}
\theoremstyle{plain}
\newtheorem{thm}{\protect\theoremname}
\theoremstyle{plain}
\newtheorem{lem}[thm]{\protect\lemmaname}
\theoremstyle{plain}
\newtheorem{cor}[thm]{\protect\corollaryname}
\theoremstyle{remark}
\newtheorem{rem}[thm]{\protect\remarkname}

\makeatother

\usepackage{babel}
\addto\captionsamerican{\renewcommand{\corollaryname}{Corollary}}
\addto\captionsamerican{\renewcommand{\lemmaname}{Lemma}}
\addto\captionsamerican{\renewcommand{\remarkname}{Remark}}
\addto\captionsamerican{\renewcommand{\theoremname}{Theorem}}
\addto\captionsenglish{\renewcommand{\algorithmname}{Algorithm}}
\addto\captionsenglish{\renewcommand{\corollaryname}{Corollary}}
\addto\captionsenglish{\renewcommand{\lemmaname}{Lemma}}
\addto\captionsenglish{\renewcommand{\remarkname}{Remark}}
\addto\captionsenglish{\renewcommand{\theoremname}{Theorem}}
\providecommand{\corollaryname}{Corollary}
\providecommand{\lemmaname}{Lemma}
\providecommand{\remarkname}{Remark}
\providecommand{\theoremname}{Theorem}

\begin{document}
\title{Reduction of multivariate mixtures and its applications}
\author{Gregory Beylkin, Lucas Monz\'{o}n and Xinshuo Yang }
\address{Department of Applied Mathematics \\
 University of Colorado at Boulder \\
 UCB 526 \\
 Boulder, CO 80309-0526 }
\begin{abstract}
We consider fast deterministic algorithms to identify the ``best''
linearly independent terms in multivariate mixtures and use them to
compute, up to a user-selected accuracy, an equivalent representation
with fewer terms. One algorithm employs a pivoted Cholesky decomposition
of the Gram matrix constructed from the terms of the mixture to select
what we call skeleton terms and the other uses orthogonalization for
the same purpose. Importantly, the multivariate mixtures do not have
to be a separated representation of a function. Both algorithms require
$\mathcal{O}\left(r^{2}N+p\left(d\right)r\thinspace N\right)$ operations,
where $N$ is the initial number of terms in the multivariate mixture,
$r$ is the number of selected linearly independent terms, and $p\left(d\right)$
is the cost of computing the inner product between two terms of a
mixture in $d$ variables. For general Gaussian mixtures $p\left(d\right)\sim d^{3}$
since we need to diagonalize a $d\times d$ matrix, whereas for separated
representations $p\left(d\right)\sim d$ (there is no need for diagonalization).
Due to conditioning issues, the resulting accuracy is limited to about
one half of the available significant digits for both algorithms.
We also describe an alternative algorithm that is capable of achieving
higher accuracy but is only applicable in low dimensions or to multivariate
mixtures in separated form. 

We describe a number of initial applications of these algorithms to
solve partial differential and integral equations and to address several
problems in data science. For data science applications in high dimensions,
we consider the kernel density estimation (KDE) approach for constructing
a probability density function (PDF) of a cloud of points, a far-field
kernel summation method and the construction of equivalent sources
for non-oscillatory kernels (used in both, computational physics and
data science) and, finally, show how to use the new algorithm to produce
seeds for subdividing a cloud of points into groups.
\end{abstract}

\maketitle

\section{Introduction}

We present (what we call) reduction algorithms for computing with
multivariate mixtures that allow us to obtain solutions of PDEs in
high dimensions as well as to address several problems in data science.
We describe a new approach for solving partial differential and integral
equations in a functional form, consider a far-field kernel summation
method and the construction of equivalent sources for non-oscillatory
kernels. As an illustration of data science applications, we present
examples of using these reduction algorithms for kernel density estimation
(KDE) to construct a probability density function (PDF) of a cloud
of points and to generate seeds for subdividing a cloud of points
into groups.

We use the well-known pivoted Cholesky factorization\footnote{Pivoted Cholesky decomposition of the Gram matrix was used by Martin
Mohlenkamp (Ohio University) and G.B. for the reduction of the number
of terms in separated representations (see comments in \cite{BI-BE-BE:2015}).} as well as a version of modified Gram-Schmidt orthogonalization to
identify the ``best'' linearly independent terms from a collection
of functions. The renewed interest in this problem is due to two observations:
(i) the approach can be used for more general multivariate mixtures
than the separated representations in \cite{BEY-MOH:2002,BEY-MOH:2005}
and (ii) multivariate Gaussian mixtures can achieve any target accuracy
when approximating functions since a multiresolution analysis can
employ a Gaussian as an approximate scaling function \cite{BE-MO-SA:2017}.
The first observation makes our approach so far the only choice for
reduction of general multivariate mixtures while the second assures
that a multivariate Gaussian mixture (and its modifications that e.g.\ include
polynomial factors) is sufficient to represent an arbitrary function
while allowing us to exploit the fact that integrals involving Gaussian
mixtures can be evaluated explicitly. We expand further on these observations
below. 

We consider multivariate functions that can be approximated via a
linear combination of multivariate atoms, 
\begin{equation}
u\left(\mathbf{x}\right)=\sum_{l=1}^{r}c_{l}g_{l}\left(\mathbf{x}\right),\,\,\,\mathbf{x}\in\mathbb{R}^{d},\label{eq:Representation via multivariate atoms}
\end{equation}
such that the inner product between the atoms,
\begin{equation}
\langle g_{l},g_{l'}\rangle=\int_{\mathbb{R}^{d}}g_{l}\left(\mathbf{x}\right)g_{l'}\left(\mathbf{x}\right)d\mathbf{x},\label{eq:inner product}
\end{equation}
can be computed efficiently. We normalize the atoms so that they have
unit $L^{2}$-norm
\[
\left\Vert g_{l}\right\Vert _{2}=\sqrt{\langle g_{l},g_{l}\rangle}=1.
\]
A particularly important example are multivariate Gaussian atoms yielding
a multivariate Gaussian mixture. In this case the atoms are
\begin{eqnarray}
g_{l}\left(\mathbf{x},\boldsymbol{\mu}_{l},\boldsymbol{\Sigma}_{l}\right) & = & \left(\det\left(4\pi\boldsymbol{\Sigma}_{l}\right)\right)^{\frac{1}{4}}N\left(\mathbf{x},\boldsymbol{\mu}_{l},\boldsymbol{\Sigma}_{l}\right)\label{eq:Gaussian_atoms}\\
 & = & \frac{1}{\left(\det\left(\pi\boldsymbol{\Sigma}_{l}\right)\right)^{1/4}}\exp\left(-\frac{1}{2}\left(\mathbf{x}-\boldsymbol{\mu}_{l}\right)^{T}\boldsymbol{\Sigma}_{l}^{-1}\left(\mathbf{x}-\boldsymbol{\mu}_{l}\right)\right),\nonumber 
\end{eqnarray}
where $N\left(\mathbf{x},\boldsymbol{\mu}_{l},\boldsymbol{\Sigma}_{l}\right)$
is the standard multivariate Gaussian distribution with mean $\boldsymbol{\mu}_{l}$,
symmetric positive definite covariance matrix $\boldsymbol{\Sigma}_{l}\in\mathbb{R}^{d\times d}$,
and $\left\Vert g_{l}\right\Vert _{2}=1$. Already in early quantum
chemistry computations, Gaussian mixtures were used because integrals
involving them can be computed explicitly (see e.g.\ \cite{BOYS:1960,LON-SIN:1960,SINGER:1960}).
Indeed, integrals with Gaussians (and (\ref{eq:inner product}) in
particular), can be evaluated explicitly in any dimension. Multivariate
Gaussian mixtures (as well as other multivariate atoms) are used within
the Radial Basis Functions (RBF) approach (see e.g.\ \cite{FOR-FLY:2015a}
and references therein). A method for approximating smooth functions
via Gaussians and a number of its applications have been developed
in \cite{MAZ-SCH:2007}.

As was demonstrated previously, a number of key operators of mathematical
physics can be efficiently represented via Gaussians leading to their
separated representations (see e.g.\ \cite{BEY-MOH:2002,BEY-MOH:2005,BEY-MON:2005,BEY-MON:2010,B-F-H-K-M:2012})
and, as a consequence, to practical algorithms (see \cite{H-F-Y-G-B:2004,Y-F-G-H-B:2004,Y-F-G-H-B:2004a,H-B-B-C-F-F-G-etc:2016}).
Importantly, as it was demonstrated recently, for any finite but arbitrary
accuracy a Gaussian can serve as a scaling function of an approximate
multiresolution analysis \cite{BE-MO-SA:2017}.

These considerations combined with algorithms of this paper to reduce
the number of terms in (\ref{eq:Representation via multivariate atoms})
suggest a new type of numerical algorithms. The basic idea of such
algorithms is simple: in the process of iteratively solving equations,
we represent both \textit{operators and} \textit{functions} via Gaussians
and, at each iteration step, compute the required integrals explicitly.
The difficulty of this approach is then how to deal with a rapid proliferation
of terms in the resulting Gaussian mixtures. For example, if an integral
involves three Gaussian mixtures with $100$ terms each, the resulting
Gaussian mixture has $10^{6}$ terms. However, in most practical applications
most of these terms are close to be linearly dependent and, thus,
to represent the result, we only need a fast algorithm to find the
``best'' linearly independent subset of the terms. We describe (what
we call) reduction algorithms to maintain a reasonable number of terms
in intermediate computations when using these representations. 

The multivariate mixtures that we consider (and construct algorithms
for) can be significantly more general than the separated representations
introduced in \cite{BEY-MOH:2002,BEY-MOH:2005} for the purpose of
computing in higher dimensions by avoiding ``the curse of dimensionality''.
Recall that a separated representation is a natural extension of the
usual separation of variables as we seek an approximation 
\begin{equation}
f\left(x_{1},\ldots,x_{d}\right)=\sum_{l=1}^{r}s_{l}\phi_{1}^{(l)}\left(x_{1}\right)\cdots\phi_{d}^{(l)}\left(x_{d}\right)+\mathcal{O}\left(\epsilon\right),\label{eqn:sumsep}
\end{equation}
where the functions $\phi_{j}^{(l)}(x_{j})$ are normalized by the
standard $L^{2}$-norm, $\Vert\phi_{j}^{(l)}\Vert_{2}=1$ and $s_{l}>0$
are referred to as $s$-values. In this approximation the functions
$\phi_{j}^{(l)}(x_{j})$ are not fixed in advance but are optimized
as to achieve the accuracy goal with (ideally) a minimal separation
rank $r$. Importantly, a separated representation is not a projection
onto a subspace, but rather a nonlinear method to track a function
in a high-dimensional space while using a small number of parameters.
The key to obtaining useful separated representations is to use the
Alternating Least Squares (ALS) algorithm to reduce the separation
rank while maintaining an acceptable error. ALS is one of the key
tools in numerical multilinear algebra and was originally introduced
for data fitting as PARAFAC model (PARAllel FACtor analysis) \cite{HARSHM:1970}
and CANDECOMP (Canonical Tensor Decomposition) \cite{CAR-CHA:1970}.
ALS has been used extensively in data analysis of (mostly) three-way
arrays (see e.g.\ the reviews \cite{TOM-BRO:2006,BRO:1997}, \cite{KOL-BAD:2009}
and references therein). We note that any discretization of $f$ in
(\ref{eqn:sumsep}) leads to a $d$-dimensional tensor $\mathcal{U}\in\mathbb{R}^{M_{1}\times\cdots\times M_{d}}$
yielding a canonical tensor decomposition of separation rank $r$,
\begin{equation}
\mathcal{U}_{i_{1},\dots,i_{d}}=\sum_{l=1}^{r}\sigma_{l}\prod_{j=1}^{d}u_{i_{j}}^{(l)},\label{eq:sep-rep-algebraic}
\end{equation}
where the $s$-values $\sigma_{l}$ are chosen so that each vector
$\mathbf{u}_{j}^{(l)}=\left\{ u_{i_{j}}^{(l)}\right\} _{i_{j}=1}^{M_{j}}$
has unit Frobenius norm $\Vert\mathbf{u}_{j}^{(l)}\Vert_{F}=1$ for
all $j,l$. However, the ALS algorithm relies heavily on the separated
form (\ref{eqn:sumsep}-\ref{eq:sep-rep-algebraic}) and is not available
for general multivariate mixtures. 

In this paper we detail and use algorithms to reduce the number of
terms in multivariate mixtures that do not necessarily admit a separated
representation. In spite their accuracy limitations, these algorithms
have several advantages as their complexity depends mildly on the
number of variables, the dimension $d$. The ``mild'' dependence
should be understood in the context of the ``curse of dimensionality''
arising when fast algorithms in low dimensions are extended to high
dimensions in a straightforward manner. Fast reduction algorithms
of this paper require $\mathcal{O}\left(r^{2}N+p\left(d\right)r\,N\right)$
operations, where $N$ is the initial number of terms in a multivariate
mixture, $r$ is the number of selected terms and $p\left(d\right)$
is the cost of computing the inner product between two terms of the
mixture. We also describe an algorithm capable of achieving higher
accuracy by avoiding the loss of precision due to conditioning issues
but which currently is only applicable in low dimensions or to multivariate
mixtures in separated form.

We start by describing reduction algorithms in Section~\ref{sec:Reduction-Algorithms}.
We then present several examples of using a reduction algorithm for
solving equations in Section~\ref{sec:Applications-of-reduction}.
In Section~\ref{sec:Kernel-Density-Estimation} we show how to apply
our approach to construct the PDF for a cloud of points in high dimensions
using the KDE approach. We then turn to far-field summation in high
dimensions in Section~\ref{sec:Far-field-summation-in} and demonstrate
how to use a reduction algorithm in this problem and for the problem
of finding equivalent sources. We also describe how to use a reduction
algorithm for splitting a cloud of points into groups (potentially
in a hierarchical manner) and briefly mention properties of such subdivisions.
Conclusions are presented in Section~\ref{sec:Conclusions-and-further}
and some key identities for multivariate Gaussians in the Appendix.

\section{\label{sec:Reduction-Algorithms}Reduction Algorithms}

In this section we describe fast deterministic algorithms to reduce
the number of terms of a linear combination of multivariate functions
of $d$ variables by selecting the ``best'' subset of these functions
that can, within a target accuracy, approximate the rest of them.
The first algorithm is based on a pivoted Cholesky decomposition of
the Gram matrix of the terms of the multivariate mixture; we assume
that the entries of this matrix, i.e.\ the inner product of these
functions, can be computed efficiently. This algorithm was mentioned
in the discussion of tensor interpolative decomposition (tensor ID)
of the canonical tensor representation in \cite{BI-BE-BE:2015}. Due
to the use of a Gram matrix, the accuracy of this approach is limited
to about one half of the available significant digits (e.g.\ $10^{-7}\sim10^{-8}$
when using double precision arithmetic). The second algorithm is based
on a pivoted modified Gram-Schmidt orthogonalization and achieves
the same accuracy (due to conditioning issues) as the first algorithm.
Nevertheless, these algorithms appear advantageous in high dimensions
since their complexities depend mildly on dimension and, if desired,
full accuracy can be restored by performing some evaluations in higher
precision. We also describe an alternative approach that achieves
full precision, but so far is limited to low dimensions or mixtures
in separated form. For this alternative reduction algorithm, we need
access to the Fourier transform of the functions in the mixture. Since
the Fourier transform is readily available for Gaussian atoms, we
present this algorithm for the case of Gaussian mixtures and note
that it can be used for any functional form that allows a rapid computation
of the integrals involved.

\subsection{Cholesky reduction}

We start with a linear combination of atoms of the form
\begin{equation}
u\left(\mathbf{x}\right)=\sum_{l=1}^{N}c_{l}g_{l}\left(\mathbf{x}\right),\ \ \ \mathbf{x}\in\mathbb{R}^{d}\label{eq:multivariate mixture}
\end{equation}
and, within a user-selected accuracy $\epsilon$, seek a representation
of the same form but with fewer terms. More precisely, we look for
a partition of indices $I=\left[\widehat{I},\widetilde{I}\right],$
where $\widehat{I}=\left[i_{1},i_{2},\cdots i_{r}\right]$ and $\widetilde{I}=\left[i_{r+1},i_{r+2},\cdots,i_{N}\right]$,
and new coefficients $\tilde{c}_{i_{m}},m=1,\cdots r$, such that
the function
\begin{equation}
\widetilde{u}\left(\mathbf{x}\right)=\sum_{m=1}^{r}\tilde{c}_{i_{m}}g_{i_{m}}\left(\mathbf{x}\right),\ \ \ r\ll N,\label{eq:function via skeletons terms-1}
\end{equation}
approximates $u$,
\begin{equation}
u\left(\mathbf{x}\right)\approx\widetilde{u}\left(\mathbf{x}\right).\label{eq:reduced rep}
\end{equation}
We present an algorithm based on a partial, pivoted Cholesky decomposition
of the Gram matrix constructed using the atoms $g_{l}$ in (\ref{eq:multivariate mixture})
and provide an estimate for the error in (\ref{eq:reduced rep}).

By analogy with the matrix Interpolative Decomposition (matrix-ID)
(see e.g.\ \cite{HA-MA-TR:2011}), we call the subset $\left\{ g_{i_{m}}\right\} _{m=1}^{r}$
the skeleton terms and $\left\{ g_{i_{m}}\right\} _{m=r+1}^{N}$ the
residual terms. In order to identify the ``best'' subset of linear
independent terms, we compute a pivoted Cholesky decomposition of
the Gram matrix of the atoms of the multivariate mixture in (\ref{eq:multivariate mixture}).
If the number of terms, $N$, is large then the cost of the full Cholesky
decomposition is prohibitive. However, we show that we can terminate
the Cholesky decomposition once the pivots are below a selected threshold.
As a result, the complexity of the algorithm is $\mathcal{O}\left(r^{2}N+p\left(d\right)r\,N\right)$,
where $N$ is the initial number of terms, $r$ is the number of selected
(skeleton) terms and $p\left(d\right)$ is the cost of computing the
inner product between two terms of a mixture in $d$ variables. In
fact, the final result will be the same as if we were to perform the
full decomposition and then keep only the significant terms. This
property is a consequence of the following lemma that can be found
in e.g.\ \cite[p.434, problem 7.1.P1]{HOR-JOH:2013}.
\begin{lem}
\label{lem:Let--B be a selfadjoint non-negative} Let $\mathbf{B}\in\mathbb{C}^{n\times n}$
be positive semi-definite and self-adjoint, i.e.\ $\mathbf{x}^{*}\mathbf{B}\mathbf{x}\ge0$
and $\mathbf{B}=\mathbf{B}^{*}$ for any $\mathbf{x}\in\mathbb{C}^{n}$.
Then its diagonal entries $b_{ii}$ are non-negative and the entries
$b_{ij}$ of~ $\mathbf{B}$ satisfy
\begin{equation}
\left|b_{ij}\right|\le\sqrt{b_{ii}b_{jj}}.\label{eq:Lemma_ineq}
\end{equation}
In particular, assuming that the first $i$ diagonal entries are in
descending order and are greater or equal than the remaining diagonal
entries,
\[
b_{11}\ge b_{22}\ge\dots\ge b_{ii}\ge b_{i+1,i+1},b_{i+2,i+2},\dots b_{nn},
\]
we have 
\begin{equation}
\left|b_{ij}\right|\le b_{ii},\,\,\,\mbox{for all }\,\,\,j\ge i.\label{eq:key estimate}
\end{equation}
\end{lem}

\begin{proof}
Let $\left\{ \mathbf{e_{i}}\right\} _{1\leq i\leq n}$ be the standard
basis vectors, that is, $\left(\mathbf{e_{i}}\right)_{j}=\delta_{ij}$.
The diagonal entries are non-negative since, for any index $i$, $b_{ii}=\mathbf{\mathbf{e_{i}}}^{*}\mathbf{B}\mathbf{\mathbf{e_{i}}}\ge0$.
We now use the same approach to estimate the size of an off-diagonal
entry $b_{ij}=\left|b_{ij}\right|e^{i\theta_{ij}}$. For the vector
$\mathbf{x}=x_{i}\mathbf{e_{i}}+x_{j}\mathbf{e_{j}}$ we have 
\begin{equation}
0\le\mathbf{\mathbf{x}^{*}\mathbf{B}\mathbf{x}}=b_{ii}x_{i}\overline{x_{i}}+b_{ij}x_{i}\overline{x_{j}}+\overline{b_{ij}}\overline{x_{i}}x_{j}+b_{jj}x_{j}\overline{x_{j}}.\label{eq:two-term inequality}
\end{equation}
Setting
\[
x_{i}=\left(\frac{b_{jj}}{b_{ii}}\right)^{1/4}\,\,\,\mbox{and}\,\,\,x_{j}=-e^{i\theta_{ij}}\left(\frac{b_{ii}}{b_{jj}}\right)^{1/4}
\]
in (\ref{eq:two-term inequality}), we obtain
\[
0\le b_{ii}\left(\frac{b_{jj}}{b_{ii}}\right)^{1/2}-\left|b_{ij}\right|-\left|b_{ij}\right|+b_{jj}\left(\frac{b_{ii}}{b_{jj}}\right)^{1/2}=2\sqrt{b_{ii}b_{jj}}-2\left|b_{ij}\right|.
\]
Thus, we arrive at 
\begin{equation}
\left|b_{ij}\right|\le\sqrt{b_{ii}b_{jj}}\le\frac{b_{ii}+b_{jj}}{2}.\label{eq:key inequality}
\end{equation}
For the second part of the lemma, selecting $i\le j$ implies that
$b_{ii}\ge b_{jj}$ and, thus, (\ref{eq:key estimate}) follows from
the last inequality in (\ref{eq:key inequality}).
\end{proof}
Lemma~\ref{lem:Let--B be a selfadjoint non-negative} implies
\begin{cor}
\label{cor: eps estimate}Let $G$ be a self-adjoint positive semi-definite
matrix such that its Cholesky decomposition has monotonically decaying
diagonal entries. If we write its Cholesky decomposition as
\[
G=\left(\begin{array}{cc}
L_{r} & 0\\
W & Q
\end{array}\right)\left(\begin{array}{cc}
L_{r}^{*} & W^{*}\\
0 & Q^{*}
\end{array}\right),
\]
where $L_{r}$ is an $r\times r$ lower triangular matrix with the
smallest diagonal entry $\epsilon>0$, then a partial Cholesky decomposition
is of the form 
\[
G=\left(\begin{array}{cc}
L_{r} & 0\\
W & 0
\end{array}\right)\left(\begin{array}{cc}
L_{r}^{*} & W^{*}\\
0 & 0
\end{array}\right)+\left(\begin{array}{cc}
0 & 0\\
0 & QQ^{*}
\end{array}\right),
\]
where all entries of the matrix $QQ^{*}$ are less than $\epsilon$.
\end{cor}

\begin{proof}
After applying $r$ steps of Cholesky decomposition, the matrix $W$
does not change in the consecutive steps. The remaining matrix $QQ^{*}$
is self-adjoint positive semi-definite and, due to the decay of the
pivots, all of its diagonal entries are less than $\epsilon$. Using
Lemma~\ref{lem:Let--B be a selfadjoint non-negative}, we conclude
that all entries of $QQ^{*}$ are less than $\epsilon$.
\end{proof}
Let us organize the collection of atoms in (\ref{eq:multivariate mixture})
as
\[
A=\left[g_{1}\left(\mathbf{x}\right),g_{2}\left(\mathbf{x}\right),\dots g_{N}\left(\mathbf{x}\right)\right].
\]
We can view $A$ as a matrix with a gigantic number of rows resulting
from a discretization of the argument $\mathbf{x}\in\mathbb{R}^{d}$.
If we replace operations that require row-wise summation by the inner
product between the atoms, then we can (and do) use matrix notation
in the sequel. Without loss of generality, to simplify notation, we
assume that the first $r$ atoms in $A$ form the skeleton, that is,
$A=\left(A_{s}\mid A_{ns}\right)$, where $A_{s}$ denotes the $r$
skeleton atoms and $A_{ns}$ the $N-r$ non-skeleton atoms. 

Given the vector of coefficients $c=\left[c_{1},c_{2},\dots c_{N}\right]^{T}$in
(\ref{eq:multivariate mixture}), we want to find new coefficients
$\tilde{c}=\left[\tilde{c}_{1},\tilde{c}_{2},\dots,\tilde{c}_{r}\right]^{T}$
to approximate $u\left(\mathbf{x}\right)$ by 
\begin{equation}
\widetilde{u}\left(\mathbf{x}\right)=\sum_{l=1}^{r}\tilde{c}_{l}g_{l}\left(\mathbf{x}\right),\ \ \ r\ll N,\label{eq:function via skeletons terms}
\end{equation}
and estimate the error $\left\Vert u-\widetilde{u}\right\Vert _{2}$
of the approximation. Note that we identify $u=Ac$ and $\widetilde{u}=A_{s}\tilde{c}$.

We first seek an approximation of all atoms via the skeleton atoms,
\[
g_{k}\left(\mathbf{x}\right)\approx\sum_{i=1}^{r}p_{ik}g_{i}\left(\mathbf{x}\right),\,\,\,k=1,\dots,N.
\]
Selecting the coefficients $p_{ik}$ as the solutions of the least
squares problem, $p_{ik}$ satisfy the normal equations, 
\begin{equation}
\sum_{i=1}^{r}p_{ik}\langle g_{i},g_{i^{\prime}}\rangle=\langle g_{k},g_{i^{\prime}}\rangle,\,\,\,k=1,\dots,N,\,\,i^{\prime}=1,2,\dots,r.\label{eq: for p_ik}
\end{equation}
Introducing the matrix $P=\left\{ p_{ik}\right\} _{{i=1,\dots r\atop k=1,\dots N}}$,
we write (\ref{eq: for p_ik}) as
\begin{equation}
A_{s}^{*}A_{s}P=A_{s}^{*}A\label{eq: normal equations for P}
\end{equation}
and observe that $P=\left(I_{r}\mid S\right)$, where $I_{r}$ is
the $r\times r$ identity matrix and $S$, an $\left(N-r\right)\times r$
matrix, which satisfies 
\begin{equation}
A_{s}^{*}A_{s}S=A_{s}^{*}A_{ns}.\label{eq: normal equations for S}
\end{equation}
Setting 
\[
\tilde{c}_{i}=\sum_{k=1}^{N}p_{ik}c_{k},
\]
or 
\[
\tilde{c}=Pc,
\]
we obtain from (\ref{eq: normal equations for P}) that the coefficients
$\tilde{c}$ solve the system of normal equations 
\begin{equation}
A_{s}^{*}A_{s}\tilde{c}=A_{s}^{*}Ac.\label{eq:normal equation for coeffs}
\end{equation}

\begin{thm}
\label{thm: main estimate}Let the Gram matrix $G$, 
\[
G=\left(\begin{array}{cc}
A_{s}^{*}A_{s} & A_{s}^{*}A_{ns}\\
A_{ns}^{*}A_{s} & A_{ns}^{*}A_{ns}
\end{array}\right),
\]
be such that its partial Cholesky decomposition has monotonically
decaying pivots and is of the form
\[
G=\left(\begin{array}{cc}
L_{r} & 0\\
W & 0
\end{array}\right)\left(\begin{array}{cc}
L_{r}^{*} & W^{*}\\
0 & 0
\end{array}\right)+\left(\begin{array}{cc}
0 & 0\\
0 & QQ^{*}
\end{array}\right)
\]
where $L_{r}$ is an $r\times r$ lower triangular matrix with the
smallest diagonal entry $\epsilon>0$. If the coefficients of the
skeleton terms are computed via (\ref{eq:normal equation for coeffs}),
then the difference between $u$ in (\ref{eq:multivariate mixture})
and its approximation (\ref{eq:function via skeletons terms}) can
be estimated as 

\begin{equation}
\left\Vert u-\widetilde{u}\right\Vert _{2}=\left\Vert A_{s}\tilde{c}-Ac\right\Vert _{2}\le\left\Vert c\right\Vert _{2}\sqrt{N-r}\,\epsilon^{1/2}.\label{eq: key estimate}
\end{equation}
\end{thm}

\begin{proof}
We have
\begin{equation}
\left\Vert A_{s}\tilde{c}-Ac\right\Vert _{2}^{2}=\left\Vert A_{s}Pc-Ac\right\Vert _{2}^{2}=\langle c,\left(P^{*}A_{s}^{*}-A^{*}\right)\left(A_{s}P-A\right)c\rangle,\label{eq:intermediate 0}
\end{equation}
where the coefficient matrix $P$ solves the normal equations (\ref{eq: normal equations for P})
and, therefore, 
\begin{equation}
A_{s}^{*}\left(A_{s}P-A\right)=0,\label{eq:null space}
\end{equation}
as well as 
\begin{equation}
A_{s}^{*}\left(A_{s}S-A_{ns}\right)=0.\label{eq:in the null space}
\end{equation}
Using (\ref{eq:null space}), we obtain

\begin{equation}
\left(P^{*}A_{s}^{*}-A^{*}\right)\left(A_{s}P-A\right)=A^{*}A-A^{*}A_{s}P\label{eq: intermediate 1}
\end{equation}
and proceed to compute $A^{*}A_{s}P$. We have $A^{*}=\left(\begin{array}{c}
A_{s}^{*}\\
A_{ns}^{*}
\end{array}\right)$ and $A_{s}P=\left(A_{s}\mid A_{s}S\right)$, so that
\[
A^{*}A_{s}P=\left(\begin{array}{cc}
A_{s}^{*}A_{s} & A_{s}^{*}A_{s}S\\
A_{ns}^{*}A_{s} & A_{ns}^{*}A_{s}S
\end{array}\right).
\]
Thus, we have 
\[
A^{*}A-A^{*}A_{s}P=\left(\begin{array}{cc}
0 & A_{s}^{*}\left(A_{ns}-A_{s}S\right)\\
0 & A_{ns}^{*}\left(A_{ns}-A_{s}S\right)
\end{array}\right).
\]
and, using (\ref{eq:in the null space}), arrive at 
\begin{equation}
A^{*}A-A^{*}A_{s}P=\left(\begin{array}{cc}
0 & 0\\
0 & A_{ns}^{*}A_{ns}-A_{ns}^{*}A_{s}S
\end{array}\right).\label{eq:intermediate 2}
\end{equation}
Equating the two expressions of the Gram matrix in the statement of
the Theorem, 
\[
G=\left(\begin{array}{cc}
A_{s}^{*}A_{s} & A_{s}^{*}A_{ns}\\
A_{ns}^{*}A_{s} & A_{ns}^{*}A_{ns}
\end{array}\right)=\left(\begin{array}{cc}
L_{r} & 0\\
W & 0
\end{array}\right)\left(\begin{array}{cc}
L_{r}^{*} & W^{*}\\
0 & 0
\end{array}\right)+\left(\begin{array}{cc}
0 & 0\\
0 & QQ^{*}
\end{array}\right),
\]
we obtain that
\begin{equation}
A_{s}^{*}A_{s}=L_{r}L_{r}^{*},\label{eq:AsL}
\end{equation}
\begin{equation}
A_{ns}^{*}A_{s}=WL_{r}^{*},\,\,\,\,\,A_{s}A_{ns}^{*}=L_{r}W^{*},\label{eq:AnsA}
\end{equation}
and 
\begin{equation}
A_{ns}^{*}A_{ns}=WW^{*}+QQ^{*}.\label{eq:AnsStarAns}
\end{equation}
We observe that from (\ref{eq: normal equations for S}) using (\ref{eq:AsL})
and (\ref{eq:AnsA}), we arrive at

\begin{equation}
L_{r}L_{r}^{*}S=L_{r}W^{*}.\label{eq:L_rW}
\end{equation}
Next we show that the non-zero block of the matrix in the right hand
side of (\ref{eq:intermediate 2}) coincides with $QQ^{*}$. Using
(\ref{eq:AnsA}) and (\ref{eq:L_rW}), we have
\begin{eqnarray*}
A_{ns}^{*}A_{s}S & = & WL_{r}^{*}S\\
 & = & W\left(L_{r}^{-1}L_{r}\right)L_{r}^{*}S\\
 & = & WL_{r}^{-1}\left(L_{r}L_{r}^{*}S\right)\\
 & = & WL_{r}^{-1}\left(L_{r}W^{*}\right)=WW^{*},
\end{eqnarray*}
where we used that $L_{r}$ is non-singular. Hence, combining the
last identity with (\ref{eq:AnsStarAns}), we obtain
\begin{equation}
A_{ns}^{*}A_{ns}-A_{ns}^{*}A_{s}S=WW^{*}+QQ^{*}-WW^{*}=QQ^{*}.\label{eq:intermediate_3}
\end{equation}
By (\ref{eq:intermediate 0}), (\ref{eq: intermediate 1}) , (\ref{eq:intermediate 2}),
and (\ref{eq:intermediate_3}) we obtain
\begin{eqnarray*}
\left\Vert A_{s}\tilde{c}-Ac\right\Vert _{2}^{2} & \le & \left\Vert c\right\Vert _{2}^{2}\left\Vert \left(P^{*}A_{s}^{*}-A^{*}\right)\left(A_{s}P-A\right)\right\Vert _{2}\\
 & = & \left\Vert c\right\Vert _{2}^{2}\left\Vert A_{ns}^{*}A_{ns}-A_{ns}^{*}A_{s}S\right\Vert _{2}\\
 & = & \left\Vert c\right\Vert _{2}^{2}\left\Vert QQ^{*}\right\Vert _{2}.
\end{eqnarray*}
Using Corollary~\ref{cor: eps estimate}, we estimate $\left\Vert QQ^{*}\right\Vert _{2}$
by its Frobenius norm, 
\[
\left\Vert QQ^{*}\right\Vert _{2}\le\left\Vert QQ^{*}\right\Vert _{F}\le\left(N-r\right)\epsilon
\]
which yields the desired estimate (\ref{eq: key estimate}).
\end{proof}
\begin{rem}
The estimate (\ref{eq: key estimate}) is tighter than the one obtained
in \cite[Theorem 3.1]{BI-BE-BE:2015} since the dependence on the
number of terms is $\mathcal{O}\left(N^{1/2}\right)$ instead of $\mathcal{O}\left(N^{3/4}\right)$.
Yet, the estimate is still pessimistic since $\left\Vert QQ^{*}\right\Vert _{2}$
is usually significantly smaller than the Frobenius norm $\left\Vert QQ^{*}\right\Vert _{F}$.
In practice, for $N$ in the range $10^{5}-10^{6}$, we did not observe
the reduction of accuracy suggested by the factor $\mathcal{O}\left(N^{1/2}\right)$.
\end{rem}

In Table~\ref{alg:Reduction-algorithm-using-Gram} we present pseudo-code
for the reduction Algorithm~\ref{alg:Reduction-algorithm-using-Gram}.
Note that this algorithm is dimension independent except for the cost
of computing the inner product which we always assume to be reasonable
by a judicious choice of the functions in the mixture. As a consequence
of Theorem~\ref{thm: main estimate} and Lemma~\ref{lem:Let--B be a selfadjoint non-negative},
it is sufficient to generate only $N\times r$ entries of the Cholesky
decomposition of the Gram matrix $G$ which requires $p\left(d\right)r\thinspace N$
operations, where $p\left(d\right)$ is the cost of computing the
inner product between two terms of a mixture. Therefore, the overall
computational cost of Algorithm~\ref{alg:Reduction-algorithm-using-Gram}
is $\mathcal{O}\left(r^{2}N+p\left(d\right)r\thinspace N\right)$
operations (see Table~\ref{tab:Timing-of-Algorithm-gram complexity}
for some examples). For a general Gaussian mixture $p\left(d\right)$
is proportional to $d^{3}$ since we need to diagonalize a $d\times d$
matrix in order to evaluate the explicit expression for the inner
product in Appendix~\ref{subsec:Inner-product-of} whereas, for separated
representations, no diagonalization is needed so that $p\left(d\right)\sim d$.
\begin{algorithm}
\noindent \begin{raggedright}
\caption{\label{alg:Reduction-algorithm-using-Gram}Reduction algorithm using
a Gram matrix}
\textbf{Inputs}: Atoms $g_{l}$ and coefficients $c_{l}$ in the representation
of $u\left(\mathbf{x}\right)=\sum_{l=1}^{N}c_{l}g_{l}\left(\mathbf{x}\right)$
and error tolerance, $\epsilon$, $10^{-14}\le\epsilon<1$. We assume
that a subroutine to compute the inner product $\langle g_{l},g_{m}\rangle$
is available.
\par\end{raggedright}
\noindent \begin{raggedright}
\textbf{Outputs}: A pivot vector $I=\left[\widehat{I},\widetilde{I}\right]$,
where $\widehat{I}=\left[i_{1},\dots,i_{r}\right]$ contains indices
of $r$ skeleton terms and $\widetilde{I}=\left[i_{r+1},\dots i_{N}\right]$
indices of terms being removed from the final representation and the
coefficients $\widetilde{c}_{i_{m}},m=1,\cdots r$, such that $\left\Vert u-\sum_{m=1}^{r}\widetilde{c}_{i_{m}}g_{i_{m}}\right\Vert _{2}\leq\left\Vert c\right\Vert _{2}\sqrt{N-r}\,\epsilon^{1/2}$
due to Theorem~\ref{thm: main estimate}.
\par\end{raggedright}
\noindent \begin{raggedright}
~
\par\end{raggedright}
\noindent \begin{raggedright}
\textbf{Stage 1: Pivoted Cholesky decomposition of the Gram matrix.}
\par\end{raggedright}
\noindent \begin{raggedright}
~
\par\end{raggedright}
\noindent \begin{raggedright}
\textbf{Initialization.}
\par\end{raggedright}
\noindent \begin{raggedright}
~
\par\end{raggedright}
\noindent \begin{raggedright}
We maintain the diagonal $\left[d_{1},d_{2},\dots,d_{N}\right]$ of
the Cholesky factor $L$ separately and initialize it as $\begin{bmatrix}1, & 1, & \cdots, & 1\end{bmatrix}$
(since all atoms have unit $L^{2}$-norm).
\par\end{raggedright}
\noindent \begin{raggedright}
Set $r=0$ and initialize a pivot vector as $I=\left[1,2,\dots,N\right]$.
\par\end{raggedright}
\noindent \begin{raggedright}
~
\par\end{raggedright}
\noindent \begin{raggedright}
\textbf{for} $l=1,N$
\par\end{raggedright}
\begin{enumerate}
\item \begin{raggedright}
Find the largest element of the diagonal and its index $i_{j}=\left\{ i_{j}:\,\,\,d_{i_{j}}\ge d_{i_{k}},\,\,\,\,k=l,\dots,N\right\} $\\
\textbf{if $d_{i_{j}}<\epsilon$ goto} \textbf{Stage 2}
\par\end{raggedright}
\item \begin{raggedright}
Swap indices $i_{j}$ and $i_{l}$ in the pivot vector $I$.
\par\end{raggedright}
\item \begin{raggedright}
Set the diagonal element of the matrix $L_{i_{l},l}=\left(d_{i_{l}}\right)^{1/2}$
\par\end{raggedright}
\end{enumerate}
\noindent \begin{raggedright}
\textbf{~~~~}~~\textbf{for }$j=l+1,N$
\par\end{raggedright}
\noindent \begin{raggedright}
~~~~~~~~~~~~$L_{i_{j},l}=\left(\left\langle g_{i_{j}},g_{i_{l}}\right\rangle -\sum_{k=1}^{l-1}L_{i_{l},k}L_{i_{j},k}\right)/L_{i_{l},l}$
\par\end{raggedright}
\noindent \begin{raggedright}
~~~~~~~~~~~~$d_{i_{j}}=d_{i_{j}}-L_{i_{j},l}^{2}$
\par\end{raggedright}
\noindent \begin{raggedright}
\textbf{~~}~~~~\textbf{end}
\par\end{raggedright}
\noindent \begin{raggedright}
~~~~~~update: $r=r+1$
\par\end{raggedright}
\noindent \begin{raggedright}
\textbf{end}
\par\end{raggedright}
\noindent \begin{raggedright}
~
\par\end{raggedright}
\noindent \begin{raggedright}
\textbf{Stage 2: Find new coefficients $\widetilde{c}_{i_{m}},m=1,\cdots,r$}
\par\end{raggedright}
\begin{enumerate}
\item \begin{raggedright}
Form a vector $b$ such that its $j$-th element is the inner product
of $\sum_{i_{m}\in\widetilde{I}}c_{i_{m}}g_{i_{m}}=\sum_{m=r+1}^{N}c_{i_{m}}g_{i_{m}}$
and $g_{i_{j}}.$\\
\textbf{for} $j=1,r$\\
~~~~~~$b\left(j\right)=\sum_{m=r+1}^{N}c_{i_{m}}\left(\sum_{k=1}^{j}L_{i_{j},k}L_{i_{m},k}\right)$\\
\textbf{end}
\par\end{raggedright}
\item \begin{raggedright}
Solve the linear system $\widehat{G}\widetilde{c}=b$, where $\widehat{G}_{jl}=\left\langle g_{i_{j}},g_{i_{l}}\right\rangle =\sum_{k=1}^{N}L_{i_{j}k}L_{i_{l}k}$
and $\widetilde{c}=\left[\widetilde{c}_{i_{1}},\widetilde{c}_{i_{1}},\cdots,\widetilde{c}_{i_{r}}\right]^{T}$
using forward and backward substitution.
\par\end{raggedright}
\item \raggedright{}Add the original coefficients of the skeleton terms
$c_{i_{m}}$ to $\widetilde{c}_{i_{m}}$ to get the new coefficients\\
\textbf{for} $m=1,r$~\\
~~~~~~$\widetilde{c}_{i_{m}}=\widetilde{c}_{i_{m}}+c_{i_{m}}$\textbf{}\\
\textbf{end}
\end{enumerate}
\end{algorithm}

\subsection{Reduction via a rank-revealing modified Gram-Schmidt algorithm}

For ordinary matrices, Rank Revealing QR (RRQR) algorithms (see e.g.\ \cite{CHA-IPS:1994,GU-EIS:1996}
and references therein) select a set of linear independent columns
that one can use to represent the other columns (within a certain
accuracy). All matrix RRQR algorithms routinely use vector addition
whereas, in our problem, the sum of two or more terms of multivariate
mixture does not simplify and, therefore, a sum must to be maintained
as a linear combination. Moreover, as far as we know, there is no
analogue of the Householder reflection or Givens rotation (as tools
for orthogonalization) unless we discretize the argument $\mathbf{x}\in\mathbb{R}^{d}$
of the terms $\left\{ g_{i}\left(\mathbf{x}\right)\right\} _{i=1}^{N}$
(note that for reduction in low dimensions we use a discretization
of the Fourier transform of the terms of the mixture as described
in Section~\ref{subsec:Alternative-reduction-algorithms}). As long
as we rely on the rapid evaluation of inner products, the reduction
algorithm mimics the modified Gram-Schmidt (MGS) orthogonalization
with the caveat that the sum of terms is maintained as a linear combination.
Unfortunately, this algorithm (even for ordinary matrices) loses one
half of significant digits (see \cite{BJORCK:1967} and discussion
in \cite[Section 5.2.9]{GOL-LOA:1996}) as does Algorithm~\ref{alg:Reduction-algorithm-using-Gram}
(it remains an open question if there exists a modification of MGS
for our problem that maintains full accuracy). As we show below, the
reduction algorithm via orthogonalization for multivariate mixtures
has formal complexity $\mathcal{O}\left(r^{3}+r^{2}N+p\left(d\right)r\thinspace N\right)$,
which is somewhat worse than the complexity of the reduction via Cholesky
decomposition. However, since in a typical reduction $N\gg r$, such
algorithm can be considered to be of the same complexity as Algorithm~\ref{alg:Reduction-algorithm-using-Gram}. 

To introduce notation, we consider a set of multivariate atoms $\left\{ g_{i}\left(\mathbf{x}\right)\right\} _{i=1}^{N}$
and write
\begin{align}
\widetilde{\psi}_{1}\left(\mathbf{x}\right) & =g_{1}\left(\mathbf{x}\right), & \psi_{1}\left(\mathbf{x}\right) & =\frac{\widetilde{\psi}_{1}\left(\mathbf{x}\right)}{\left\Vert \widetilde{\psi}_{1}\right\Vert }\nonumber \\
\widetilde{\psi}_{2}\left(\mathbf{x}\right) & =g_{2}\left(\mathbf{x}\right)-t_{21}\widetilde{\psi}_{1}\left(\mathbf{x}\right), & \psi_{2}\left(\mathbf{x}\right) & =\frac{\widetilde{\psi}_{2}\left(\mathbf{x}\right)}{\left\Vert \widetilde{\psi}_{2}\right\Vert }\nonumber \\
\widetilde{\psi}_{3}\left(\mathbf{x}\right) & =g_{3}\left(\mathbf{x}\right)-t_{31}\widetilde{\psi}_{1}\left(\mathbf{x}\right)-t_{32}\widetilde{\psi}_{2}\left(\mathbf{x}\right), & \psi_{3}\left(\mathbf{x}\right) & =\frac{\widetilde{\psi}_{3}\left(\mathbf{x}\right)}{\left\Vert \widetilde{\psi}_{3}\right\Vert }\label{eq:gram schmidt process}\\
 & \vdots &  & \text{\ensuremath{\vdots}}\nonumber \\
\widetilde{\psi}_{N}\left(\mathbf{x}\right) & =g_{N}\left(\mathbf{x}\right)-\left(\sum_{j=1}^{N}t_{Nj}\widetilde{\psi}_{j}\left(\mathbf{x}\right)\right), & \psi_{N}\left(\mathbf{x}\right) & =\frac{\widetilde{\psi}_{N}\left(\mathbf{x}\right)}{\left\Vert \widetilde{\psi}_{N}\right\Vert }\nonumber 
\end{align}
where the functions $\widetilde{\psi}_{i}$ are orthogonal, $\psi_{i}$
are orthonormal, and
\[
t_{ij}=\frac{\left\langle \widetilde{\psi}_{j},g_{i}\right\rangle }{\left\langle \widetilde{\psi}_{j},\widetilde{\psi}_{j}\right\rangle },\,\,\,1\leq j<i,\thinspace i=1,\dots,N.
\]
We can also rewrite (\ref{eq:gram schmidt process}) as
\begin{align}
g_{1}\left(\mathbf{x}\right) & =r_{11}\psi_{1}\left(\mathbf{x}\right),\nonumber \\
g_{2}\left(\mathbf{x}\right) & =r_{21}\psi_{1}\left(\mathbf{x}\right)+r_{22}\psi_{2}\left(\mathbf{x}\right),\nonumber \\
g_{3}\left(\mathbf{x}\right) & =r_{31}\psi_{1}\left(\mathbf{x}\right)+r_{32}\psi_{2}\left(\mathbf{x}\right)+r_{33}\psi_{3}\left(\mathbf{x}\right),\label{eq:g via u-tilde}\\
 & \vdots\nonumber \\
g_{N}\left(\mathbf{x}\right) & =\sum_{j=1}^{N}r_{Nj}\psi_{j}\left(\mathbf{x}\right),\nonumber 
\end{align}
where $r_{ij}=t_{ij}\left\Vert \widetilde{\psi}_{j}\right\Vert $
and $r_{ii}=\left\Vert \widetilde{\psi}_{i}\right\Vert $, $i=1,\dots,N$.
We can also consider a set of coefficients $s_{ij}$, $1\leq j\leq i\leq N$
such that
\begin{eqnarray}
\psi_{1}\left(\mathbf{x}\right) & = & s_{11}g_{1}\left(\mathbf{x}\right),\nonumber \\
\psi_{2}\left(\mathbf{x}\right) & = & s_{21}g_{1}\left(\mathbf{x}\right)+s_{22}g_{2}\left(\mathbf{x}\right),\nonumber \\
\psi_{3}\left(\mathbf{x}\right) & = & s_{31}g_{1}\left(\mathbf{x}\right)+s_{32}g_{2}\left(\mathbf{x}\right)+s_{33}g_{3}\left(\mathbf{x}\right),\label{eq:u-tilde via g}\\
 & \vdots\nonumber \\
\psi_{N}\left(\mathbf{x}\right) & = & \sum_{j=1}^{N}s_{Nj}g_{j}\left(\mathbf{x}\right).\nonumber 
\end{eqnarray}

To compute the coefficients $r_{ij}$ in (\ref{eq:g via u-tilde})
and $s_{ij}$ in (\ref{eq:u-tilde via g}), we need to evaluate the
inner product between the multivariate atoms $g_{i}$ and $g_{j}$
which we denote as $G_{ij}=\left\langle g_{i},g_{j}\right\rangle $.

In order to estimate the computational cost, let us assume that the
first $k-1$ steps have been accomplished so that the coefficients
$r_{ij}$, $1\leq i\leq N$, $1\leq j\leq k-1$ and $j\le i$ in (\ref{eq:g via u-tilde}),
the coefficients $s_{ij}$, $1\leq j\leq i\leq k$ in (\ref{eq:u-tilde via g}),
the norms of $\widetilde{\psi}_{i}=g_{i}-\sum_{j=1}^{k-1}r_{ij}\psi_{j}$,
$i=1,\dots,N$ as well as the partial Gram matrix $G_{ij}$, $1\leq i\leq N$,
$1\leq j\leq k-1$ and $j\le i$ are already available. We then select
the term, $\widetilde{\psi}_{i}=g_{i}-\sum_{j=1}^{k-1}r_{ij}\psi_{j}$,
$i=k,\dots,N$ with the largest norm which we then swap to become
the term with index $k$. For simplicity of indexing we assume that
such term was already in the position $k$. Thus, our pivoting strategy
uses a greedy algorithm by selecting terms with the largest norm (we
note that alternative pivoting strategies may be possible but we do
not explore them here).

Note that the norms $\left\Vert \widetilde{\psi}_{i}\right\Vert =\left\Vert g_{i}-\sum_{j=1}^{i-1}r_{ij}\psi_{j}\right\Vert =r_{ii}$,
$i=1,\dots,k$ do not change in the steps that follow. At the $k$-th
step, we first need to compute the coefficients in $s_{kj}$, $j=1,\dots,k$
such that
\[
\psi_{k}\left(\mathbf{x}\right)=\sum_{j=1}^{k}s_{kj}g_{j}\left(\mathbf{x}\right)
\]
as in (\ref{eq:u-tilde via g}). Since the norm $r_{kk}=\left\Vert \widetilde{\psi}_{k}\right\Vert $
is available we use (\ref{eq:gram schmidt process}) to evaluate
\begin{eqnarray*}
r_{kk}\psi_{k}\left(\mathbf{x}\right) & = & g_{k}\left(\mathbf{x}\right)-\sum_{i=1}^{k-1}r_{ki}\psi_{i}\left(\mathbf{x}\right)=g_{k}\left(\mathbf{x}\right)-\sum_{i=1}^{k-1}r_{ki}\left(\sum_{j=1}^{i}s_{ij}g_{j}\left(\mathbf{x}\right)\right)\\
 & = & g_{k}\left(\mathbf{x}\right)-\sum_{j=1}^{k-1}\left(\sum_{i=j}^{k-1}r_{ki}s_{ij}\right)g_{j}\left(\mathbf{x}\right)=-\sum_{j=1}^{k-1}\left(\sum_{i=j}^{k-1}r_{ki}s_{ij}\right)g_{j}\left(\mathbf{x}\right)+g_{k}\left(\mathbf{x}\right)
\end{eqnarray*}
and, therefore, obtain
\begin{equation}
s_{kj}=-\frac{1}{r_{kk}}\sum_{i=j}^{k-1}r_{ki}s_{ij},\ \ \ 1\leq j\leq k-1,\ \ \ s_{kk}=\frac{1}{r_{k,k}}.\label{eq:compute Q matrix}
\end{equation}
Computing the coefficients $s_{kj}$ take $\mathcal{O}\left(k^{2}\right)$
operations at this step. Next, we compute the coefficients in (\ref{eq:g via u-tilde})
and update the term $\widetilde{\psi}_{i}=g_{i}-\sum_{j=1}^{k-1}r_{ij}\psi_{j}$
by subtracting $r_{ik}\psi_{k}$ and evaluate the resulting norms
for $i=k+1,\dots,N$. Specifically, for $i=k+1,\dots,N$, we compute
$G_{ik}=\left\langle g_{i},g_{k}\right\rangle ,$and
\begin{equation}
r_{ik}=\left\langle g_{i},\psi_{k}\right\rangle =\left\langle g_{i},\sum_{j=1}^{k}s_{kj}g_{j}\right\rangle =\sum_{j=1}^{k}s_{kj}\left\langle g_{i},g_{j}\right\rangle =\sum_{j=1}^{k}s_{kj}G_{ij}.\label{eq:compute coef r}
\end{equation}
We update the norm of $\widetilde{\psi}_{i}$ as follows,
\begin{eqnarray}
\left\Vert \widetilde{\psi}_{i}\right\Vert ^{2} & = & \left\Vert g_{i}-\sum_{j=1}^{k-1}r_{ij}\psi_{j}-r_{ik}\psi_{k}\right\Vert ^{2}=\left\Vert g_{i}-\sum_{j=1}^{k-1}r_{ij}\psi_{j}\right\Vert ^{2}+r_{ik}^{2}-2\left\langle g_{i}-\sum_{j=1}^{k-1}r_{ij}\psi_{j},r_{ik}\psi_{k}\right\rangle \label{eq:update norm}\\
 & = & \left\Vert g_{i}-\sum_{j=1}^{k-1}r_{ij}\psi_{j}\right\Vert ^{2}+r_{ik}^{2}-2\left\langle g_{i},r_{ik}\psi_{k}\right\rangle =\left\Vert g_{i}-\sum_{j=1}^{k-1}r_{ij}\psi_{j}\right\Vert ^{2}+r_{ik}^{2}-2r_{ik}\left\langle g_{i},\psi_{k}\right\rangle \nonumber \\
 & = & \left\Vert g_{i}-\sum_{j=1}^{k-1}r_{ij}\psi_{j}\right\Vert ^{2}-r_{ik}^{2},\nonumber 
\end{eqnarray}
where we used (\ref{eq:compute coef r}). At this step, computing
$G_{ik}$, (\ref{eq:compute coef r}) and (\ref{eq:update norm})
are performed for each index $i$ so that it requires $\mathcal{O}\left(p\left(d\right)N+k\thinspace N\right)$
operations, where $p\left(d\right)$ is the cost of computing the
inner product between two terms of the mixture.

Once the skeleton terms $\left\{ g_{i}\left(\mathbf{x}\right)\right\} _{i=1}^{r}$
are identified in the process of orthogonalization, we compute the
coefficients $\left\{ \widetilde{c}_{i}\right\} _{i=1}^{r}$ of the
new representation of the multivariate mixture via these terms, i.e.\ approximate
$\sum_{i=1}^{N}c_{i}g_{i}\left(\mathbf{x}\right)$ as $\sum_{i=1}^{r}\widetilde{c}_{i}g_{i}\left(\mathbf{x}\right)$.
Using the fact that for $j>r$
\[
g_{j}\left(\mathbf{x}\right)\approx\sum_{k=1}^{r}r_{jk}\psi_{k}\left(\mathbf{x}\right)=\sum_{k=1}^{r}r_{jk}\left(\sum_{i=1}^{k}s_{ki}g_{i}\left(\mathbf{x}\right)\right)=\sum_{i=1}^{r}\left(\sum_{k=i}^{r}r_{jk}s_{ki}\right)g_{i}\left(\mathbf{x}\right),
\]
we compute
\begin{eqnarray*}
\sum_{i=1}^{N}c_{i}g_{i}\left(\mathbf{x}\right) & = & \sum_{i=1}^{r}c_{i}g_{i}\left(\mathbf{x}\right)+\sum_{j=r+1}^{N}c_{j}g_{j}\left(\mathbf{x}\right)\\
 & \approx & \sum_{i=1}^{r}c_{i}g_{i}\left(\mathbf{x}\right)+\sum_{j=r+1}^{N}c_{j}\left(\sum_{i=1}^{r}\left(\sum_{k=i}^{r}r_{jk}s_{ki}\right)g_{i}\left(\mathbf{x}\right)\right)\\
 & = & \sum_{i=1}^{r}c_{i}g_{i}\left(\mathbf{x}\right)+\sum_{i=1}^{r}\left(\sum_{j=r+1}^{N}c_{j}\left(\sum_{k=i}^{r}r_{jk}s_{ki}\right)\right)g_{i}\left(\mathbf{x}\right)\\
 & = & \sum_{i=1}^{r}\widetilde{c}_{i}g_{i}\left(\mathbf{x}\right)
\end{eqnarray*}
so that 
\begin{equation}
\widetilde{c}_{i}=c_{i}+\sum_{j=r+1}^{N}c_{j}\left(\sum_{k=i}^{r}r_{jk}s_{ki}\right).\label{eq:new coefficients}
\end{equation}
If the number of skeleton terms is $r$ then, combining complexity
estimates for all steps, the resulting algorithm has a complexity
$\mathcal{O}\left(r^{3}+r^{2}N+p\left(d\right)r\thinspace N\right)$.
Since this algorithm is designed to be used when $N\gg r$, we conclude
that the overall cost is $\mathcal{O}\left(r^{2}N+p\left(d\right)r\thinspace N\right)$.
Pseudo-code for this reduction algorithm is presented as Algorithm~\ref{alg:reduction MGS}. 

\begin{algorithm}
\noindent \begin{raggedright}
\caption{\label{alg:reduction MGS}Reduction algorithm using modified Gram-Schmidt
orthogonalization}
\textbf{Inputs}: Atoms $g_{l}$ and coefficients $c_{l}$ in the representation
of $u\left(\mathbf{x}\right)=\sum_{l=1}^{N}c_{l}g_{l}\left(\mathbf{x}\right)$
and error tolerance, $\epsilon$, $10^{-14}\le\epsilon<1$. We assume
that a subroutine to compute the inner product $\langle g_{l},g_{l^{'}}\rangle$
is available.
\par\end{raggedright}
\noindent \begin{raggedright}
\textbf{Outputs}: The pivot vector $I=\left[\widehat{I},\widetilde{I}\right]$,
where $\widehat{I}=\left[i_{1},\dots,i_{r}\right]$ contains indices
of $r$ skeleton terms and $\widetilde{I}=\left[i_{r+1},\dots i_{N}\right]$
indices of terms being removed from the final representation and the
coefficients $\widetilde{c}_{i_{m}},m=1,\cdots r$, such that $\left|u\left(\mathbf{x}\right)-\sum_{m=1}^{r}\widetilde{c}_{i_{m}}g_{i_{m}}\left(\mathbf{x}\right)\right|=\mathcal{O}\left(\epsilon^{1/2}\right).$
\par\end{raggedright}
\noindent \begin{raggedright}
~
\par\end{raggedright}
\noindent \begin{raggedright}
\textbf{Stage 1: Pivoted MGS applied to the atoms $g_{l},$ $l=1,\cdots,N.$}
\par\end{raggedright}
\noindent \begin{raggedright}
Set $\widetilde{\psi}_{l}\left(\mathbf{x}\right)=g_{l}\left(\mathbf{x}\right)$
for $l=1,\cdots,N$. Note that, initially, all norms $\left\Vert \widetilde{\psi}_{l}\right\Vert =1$
since all atoms have unit $L^{2}$-norm. 
\par\end{raggedright}
\noindent \begin{raggedright}
Set rank $r=0$ and initialize a pivot vector as $I=\left[1,2,\dots,N\right]$.
\par\end{raggedright}
\noindent \begin{raggedright}
\textbf{for} $k=1,N$
\par\end{raggedright}
\begin{enumerate}
\item \begin{raggedright}
Find the largest norm $\left\Vert \widetilde{\psi}_{i_{j}}\right\Vert $
and its index $i_{j}=\left\{ i_{j}:\,\,\,\left\Vert \widetilde{\psi}_{i_{j}}\right\Vert \ge\left\Vert \widetilde{\psi}_{i_{l}}\right\Vert ,\,\,\,\,l=k,\dots,N\right\} $\textbf{}\\
\textbf{if $\left\Vert \widetilde{\psi}_{i_{j}}\right\Vert <\epsilon$
goto} \textbf{Stage 2}
\par\end{raggedright}
\item \begin{raggedright}
Swap indices $i_{j}$ and $i_{k}$ in the pivot vector $I$.
\par\end{raggedright}
\item \begin{raggedright}
\textbf{set} $r_{i_{k},k}=\left\Vert \widetilde{\psi}_{i_{k}}\right\Vert $
\par\end{raggedright}
\end{enumerate}
\noindent \begin{raggedright}
\textbf{~}
\par\end{raggedright}
\noindent \begin{raggedright}
\textbf{~~~~}~~\textbf{for $j=1,\dots k-1$}
\par\end{raggedright}
\noindent \begin{raggedright}
\textbf{~~~~~~~~}~~~~$s_{i_{k},j}=-\frac{1}{r_{i_{k},k}}\sum_{l=j}^{k-1}r_{i_{k},l}s_{i_{l},j}$
\par\end{raggedright}
\noindent \begin{raggedright}
\textbf{~~~~}~~\textbf{end}
\par\end{raggedright}
\noindent \begin{raggedright}
\textbf{~~~~}~~$s_{i_{k},k}=\frac{1}{r_{i_{k},k}}$
\par\end{raggedright}
\noindent \begin{raggedright}
\textbf{~}
\par\end{raggedright}
\noindent \begin{raggedright}
\textbf{~~~~}~~\textbf{for $j=k+1,N$}
\par\end{raggedright}
\noindent \begin{raggedright}
\textbf{~~~~~~~~}~~~~$G_{i_{j},k}=\left\langle g_{i_{j}},g_{i_{k}}\right\rangle $
\par\end{raggedright}
\noindent \begin{raggedright}
\textbf{~~~~~~~~}~~~~$r_{i_{j}k}=\sum_{l=1}^{k}s_{i_{k},l}G_{i_{j},l}$
\par\end{raggedright}
\noindent \begin{raggedright}
\textbf{~~~~~~~~}~~~~update $\left\Vert \widetilde{\psi}_{i_{j}}\right\Vert =\left(\left\Vert \widetilde{\psi}_{i_{j}}\right\Vert ^{2}-r_{i_{j},k}^{2}\right)^{1/2}$
\par\end{raggedright}
\noindent \begin{raggedright}
\textbf{~~~~}~~\textbf{end}
\par\end{raggedright}
\noindent \begin{raggedright}
~~~~~~update: $r=r+1$
\par\end{raggedright}
\noindent \begin{raggedright}
\textbf{end}
\par\end{raggedright}
\noindent \begin{raggedright}
~
\par\end{raggedright}
\noindent \begin{raggedright}
\textbf{Stage 2: Find new coefficients $\widetilde{c}_{i_{m}}$, $m=1,\cdots,r$.}
\par\end{raggedright}
\noindent \raggedright{}~Compute new coefficients $\widetilde{c}_{i_{m}}=c_{i_{m}}+\sum_{j=r+1}^{N}c_{i_{j}}\left(\sum_{k=m}^{r}r_{i_{j},k}s_{i_{k},m}\right)$,
\textbf{$m=1,\cdots,r$}.
\end{algorithm}

\subsection{Alternative reduction algorithms\label{subsec:Alternative-reduction-algorithms}}

Using Algorithm~\ref{alg:Reduction-algorithm-using-Gram} or \ref{alg:reduction MGS},
half of the significant digits are lost due to poor conditioning (see
examples in \cite{BI-BE-BE:2015} and \cite{RE-DO-BE:2016}). In order
to identify ``best'' linear independent terms we can design a matrix
with a better condition number if instead of the functions of the
mixture we use a ``dual'' family for computing inner products. In
the case of Gaussians (which are well localized), a natural set of
such ``dual'' functions are exponentials with purely imaginary exponents
(which are global functions); computing the inner product with them
reduces to computing their Fourier transform. Therefore, as representatives
of Gaussian atoms we can then use frequency vectors, i.e.\ samples
of their Fourier transforms. Such sampling strategy should be sufficient
to differentiate between all Gaussian atoms; it is fairly straightforward
to achieve this in low dimensions or if the functions admit a separated
representation. Currently, we do not know how to do it efficiently
in high dimensions. Naively it appears to require the construction
of a sample matrix with $\mathcal{O}\left(N\times N\thinspace d\thinspace r\right)$
entries and additional work is required to understand how to lower
this complexity. Alternatively, in dimensions $d=1,2,3$ it is sufficient
to use $\mathcal{O}\left(N\times r^{d}\right)$ samples if we were
to use the straightforward generalization of the algorithm in dimension
$d=1$ described below. In all cases, the last step in this approach
is to compute the matrix ID of the sample matrix.

We present a deterministic algorithm in dimension $d=1$ and note
that its extension to functions in separated form in high dimensions
can follow the approach in \cite{BI-BE-BE:2015}. We consider a univariate
Gaussian mixture
\begin{equation}
u\left(x\right)=\sum_{l=1}^{N}c_{l}g_{l}\left(x\right),\ \ \ x\in\mathbb{R},\label{eq:univariate Gaussian mixture}
\end{equation}
where
\[
g_{l}\left(x\right)=\frac{1}{\pi^{\frac{1}{4}}\sigma_{l}^{\frac{1}{2}}}e^{-\frac{\left(x-\mu_{l}\right)^{2}}{2\sigma_{l}^{2}}},\,\,\,\,\,\left\Vert g_{l}\right\Vert _{2}=1,
\]
and seek the best linear independent subset as in (\ref{eq:function via skeletons terms-1}).
Defining the Fourier transform of $f$ as
\[
\hat{f}\left(\xi\right)=\frac{1}{\sqrt{2\pi}}\int_{-\infty}^{\infty}f\left(x\right)e^{-ix\xi}dx,
\]
we obtain
\[
\hat{g_{l}}\left(\xi\right)=\frac{\sigma_{l}^{\frac{1}{2}}}{\pi^{\frac{1}{4}}}e^{-\frac{\sigma_{l}^{2}\xi^{2}}{2}}e^{-i\mu_{l}\xi}.
\]
We set the highest frequency $\mathbf{\xi}_{high}$ of $\left\{ \hat{g}_{l}\left(\xi\right)\right\} _{l=1}^{N}$
as
\[
\mathbf{\xi}_{high}=\sqrt{\frac{-2\log\frac{\pi^{\frac{1}{4}}10^{-16}}{\sigma^{\frac{1}{2}}}}{\sigma^{2}}},\ \ \ \sigma=\min_{l=1,\cdots,N}\sigma_{l},
\]
such that $\left|\hat{g}_{l}\left(\xi\right)\right|<10^{-16}$ for
$\xi>\mathbf{\xi}_{high}$ and $l=1,\cdots,N$. We also set the lowest
frequency $\mathbf{\xi}_{low}=10^{-2}\sim10^{-3}$, a positive value
obtained experimentally. We then sample the interval $\left[\mathbf{\xi}_{low},\mathbf{\xi}_{high}\right]$
using frequencies $\xi_{k},k=1,\cdots r_{p},$ equally spaced on a
logarithmic scale,
\[
\xi_{k}=e^{\left(\log\mathbf{\xi}_{low}+k\frac{\mathbf{\xi}_{high}-\mathbf{\xi}_{low}}{r_{p}}\right)},
\]
where $r_{p}$ is the number of samples. We choose $r_{p}>r$, where
$r$ is the expected final number of terms. In our setup, the column
\[
\begin{bmatrix}\hat{g}_{l}\left(\xi_{1}\right)\\
\hat{g}_{l}\left(\xi_{2}\right)\\
\vdots\\
\hat{g}_{l}\left(\xi_{r_{p}}\right)
\end{bmatrix}
\]
serves as a representative of the Gaussian $g_{l}$. In this way,
we reduce the problem to that of using the matrix ID. Specifically,
given frequencies $\xi_{k}$, $k=1,\cdots,r_{p},$ we construct a
$r_{p}\times N$ sample matrix $Y$,
\[
Y=\begin{bmatrix}\hat{g_{1}}\left(\xi_{1}\right) & \hat{g_{2}}\left(\xi_{1}\right) & \cdots & \hat{g_{N}}\left(\xi_{1}\right)\\
\hat{g_{1}}\left(\xi_{2}\right) & \hat{g_{2}}\left(\xi_{2}\right) & \cdots & \hat{g_{N}}\left(\xi_{2}\right)\\
\vdots & \vdots & \ddots & \vdots\\
\hat{g_{1}}\left(\xi_{r_{p}}\right) & \hat{g_{2}}\left(\xi_{r_{p}}\right) & \cdots & \hat{g_{N}}\left(\xi_{r_{p}}\right)
\end{bmatrix},
\]
and compute its matrix ID (see e.g.\ \cite{HA-MA-TR:2011}). We obtain
a partition of indices $I=\left[\widehat{I},\widetilde{I}\right],$
where $\widehat{I}=\left[i_{1},\dots,i_{r}\right]$ and $\widetilde{I}=\left[i_{r+1},\dots i_{N}\right]$
denote the skeleton and residual terms respectively. We also obtain
a matrix $X$ such that
\[
Y=Y_{\left[:,\widehat{I}\right]}X,
\]
where $X$ is a $r\times N$ matrix that satisfies $X_{\left[:,\widehat{I}\right]}=I_{r}$.
We then compute the new coefficients as
\[
\tilde{c}_{i_{m}}=c_{i_{m}}+\sum_{n=r+1}^{N}c_{i_{n}}X_{mn},\,\,\,m=1,2,\dots,r,
\]
and use them to approximate 
\[
u\left(x\right)\approx\sum_{m=1}^{r}\tilde{c}_{i_{m}}g_{i_{m}}\left(x\right).
\]
The accuracy of this approximation appears to be the same as the accuracy
of matrix ID. Unfortunately, for multivariate Gaussian atoms, the
size of the matrix $Y$ appears to grow too fast with the dimension
$d$ (except in the case of separated representations where such dependence
is linear). While our approach via frequency vectors can be extended
in a straightforward manner to dimensions $d=2,3$, it is of interest
to construct an algorithm yielding high accuracy approximations in
higher dimensions.

\begin{algorithm}
\caption{Reduction using frequencies (dimension $d=1$)\label{alg:Reduction-using-frequencies}}
\textbf{Inputs}: Atoms $g_{l}$ and coefficients $c_{l}$ in the representation
of $u\left(x\right)=\sum_{l=1}^{N}c_{l}g_{l}\left(x\right)$, number
of frequencies samples $r_{p}$, and error tolerance $\epsilon>0$.
We assume that a subroutine to compute the Fourier transform $\hat{g}_{l}\left(\mathbf{\xi}\right)$
is available.
\noindent \begin{raggedright}
~
\par\end{raggedright}
\textbf{Outputs}: an index set $I=\left[\widehat{I},\widetilde{I}\right]$,
where $\widehat{I}=\left[i_{1},\dots,i_{r}\right]$ contains indices
of $r$ skeleton terms and $\widetilde{I}=\left[i_{r+1},\dots i_{N}\right]$
contains indices of terms being removed from the final representation,
and coefficients $\widetilde{c}_{i_{m}},m=1,\cdots r$, such that
$u\left(x\right)\approx\sum_{m=1}^{r}\widetilde{c}_{i_{m}}g_{i_{m}}\left(x\right)$
with accuracy of the matrix ID.
\begin{enumerate}
\item set $\xi_{low}\in\left[10^{-2},10^{-3}\right]$ and compute $\xi_{high}=\sqrt{\frac{-2\log\frac{\pi^{\frac{1}{4}}10^{-16}}{\sigma^{\frac{1}{2}}}}{\sigma^{2}}}$
where $\sigma=\min_{l=1,\cdots,N}\left\{ \sigma_{l}\right\} $.
\item initialize $\xi_{k}=e^{\left(\log\mathbf{\xi}_{low}+k\frac{\mathbf{\xi}_{high}-\mathbf{\xi}_{low}}{r_{p}}\right)},k=1,\cdots,r_{p}$.
\item construct matrix $Y$ such that $Y_{kl}=\hat{g}_{l}\left(\xi_{k}\right)$.
\item compute a matrix ID of $Y$ to obtain a index partition $\left[\widehat{I},\widetilde{I}\right]$
and a matrix $X$ such that $Y=Y_{\left[:,\widehat{I}\right]}X$
\item compute new coefficients via $\tilde{c}_{i_{m}}=c_{i_{m}}+\sum_{n=r+1}^{N}c_{i_{n}}X_{mn}$,
for $m=1,\dots r$ .
\end{enumerate}
\end{algorithm}

\subsection{Timings and comparisons}

We compare the performance of Algorithms~\ref{alg:Reduction-algorithm-using-Gram},
\ref{alg:reduction MGS} and \ref{alg:Reduction-using-frequencies}
in dimension $d=1$ by considering a univariate Gaussian mixture of
the form (\ref{eq:univariate Gaussian mixture}). We choose $N=10000$,
and sample $c_{l}$, $\sigma_{l}$ and $\mu_{l}$ from uniform distributions
$\mathcal{U}\left(-1,1\right)$, $\mathcal{U}\left(0,0.5\right)$,
and $\mathcal{U}\left(-5,5\right)$, respectively. We apply Algorithms~\ref{alg:Reduction-algorithm-using-Gram},
\ref{alg:reduction MGS} and \ref{alg:Reduction-using-frequencies}
to reduce the number of terms in the Gaussian mixture and display
the original function and the errors of the resulting approximations
in Figure~\ref{fig: reduction_1d}. 

\begin{figure}
\centering{}\includegraphics[scale=0.8]{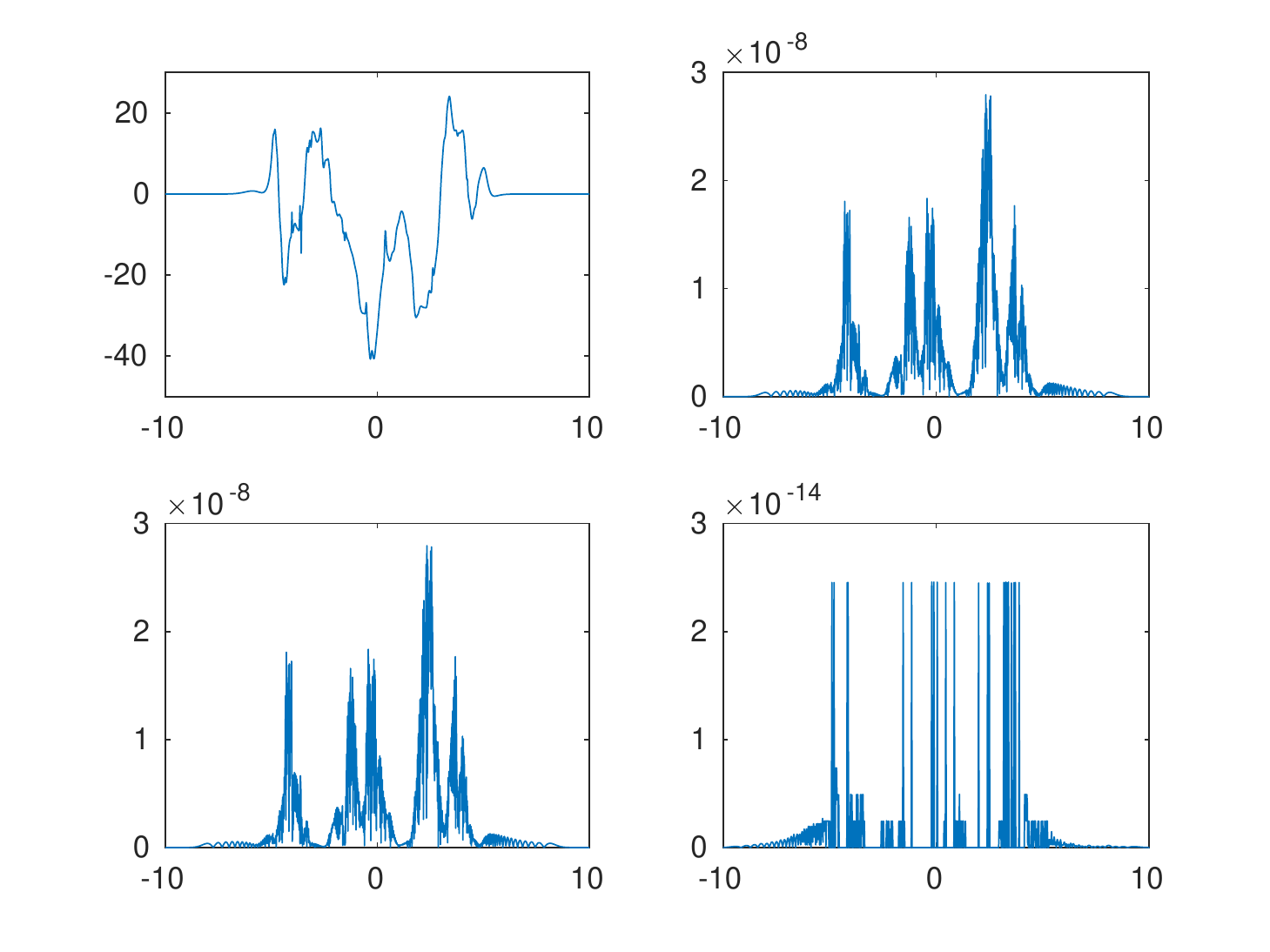}\caption{\label{fig: reduction_1d}A function represented via a Gaussian mixture
with $10^{4}$ terms (top left). Relative errors obtained by using
Algorithm~\ref{alg:Reduction-algorithm-using-Gram} (requested accuracy
$10^{-7}$) yielding $300$ terms (top right), by using Algorithm~\ref{alg:reduction MGS}
(requested accuracy $10^{-7}$) yielding $300$ terms (bottom left)
and by using Algorithm~\ref{alg:Reduction-using-frequencies} (requested
accuracy $10^{-13}$) yielding $398$ terms (bottom right).}
\end{figure}
We run this experiment with different accuracy thresholds and display
the resulting number of skeleton terms obtained by these algorithms
in Table~\ref{tab:Number-of-skeleton terms } (we use Algorithms~\ref{alg:Reduction-algorithm-using-Gram}
and \ref{alg:reduction MGS} implemented in quadruple precision to
make comparison possible for higher accuracies). We observe that the
number of skeleton terms obtained by these algorithms differs only
slightly (this difference is irrelevant for intended applications). 

\begin{table}
\centering{}%
\begin{tabular}{|l|c|l|c|l|c|l|}
\hline 
Requested  & Alg.~\ref{alg:Reduction-algorithm-using-Gram} & Actual & Alg.~\ref{alg:reduction MGS} & Actual  & Alg.~\ref{alg:Reduction-using-frequencies} & Actual\tabularnewline
accuracy &  & accuracy &  & accuracy &  & accuracy\tabularnewline
\hline 
\hline 
$10^{-3}$ & $216$ & $0.197e-3$ & $216$ & $0.197e-3$ & $217$ & $0.220e-3$\tabularnewline
\hline 
$10^{-5}$ & $261$ & $0.244e-5$ & $261$ & $0.244e-5$ & $259$ & $0.368e-5$\tabularnewline
\hline 
$10^{-7}$ & $300$ & $0.279e-7$ & $300$ & $0.279e-7$ & $297$ & $0.383e-7$\tabularnewline
\hline 
$10^{-9}$ & $330$ & $0.140e-9$ & $330$ & $0.140e-9$ & $330$ & $0.173e-9$\tabularnewline
\hline 
$10^{-11}$ & $361$ & $0.143e-11$ & $361$ & $0.143e-11$ & $361$ & $0.204e-11$\tabularnewline
\hline 
$10^{-13}$ & $390$ & $0.179e-13$ & $390$ & $0.179e-13$ & $398$ & $0.924e-14$\tabularnewline
\hline 
\end{tabular}\caption{\label{tab:Number-of-skeleton terms }Number of skeleton terms retained
by\foreignlanguage{english}{ }Algorithms~\ref{alg:Reduction-algorithm-using-Gram}~and~\ref{alg:reduction MGS}
(implemented in quadruple precision) and Algorithm~\ref{alg:Reduction-using-frequencies}
for different approximation accuracies. The requested accuracy is
shown in the first column and (slightly different) resulting accuracies
and the number of skeleton terms of the mixture are shown for each
algorithm separately. In this example\foreignlanguage{english}{ }Algorithms~\ref{alg:Reduction-algorithm-using-Gram}~and~\ref{alg:reduction MGS}
selected the same skeleton terms.}
\end{table}
We also time Algorithm~\ref{alg:Reduction-algorithm-using-Gram}
(implemented in double precision) and examine its scaling as a function
of the number of initial terms and the number of skeleton terms $r$
in dimensions $d=1$ and $d=5$. In this test we fix the initial number
of terms $N$ and vary the number of skeleton terms $r$ and vice
versa. The resulting representation may not achieve a particular accuracy,
but our the goal here is to see how Algorithm~\ref{alg:Reduction-algorithm-using-Gram}
scales in $N$ and $r$. We consider a Gaussian mixture with parameters
$c_{l}$ and $\boldsymbol{\mu}_{l}$ sampled from uniform distributions
$\mathcal{U}\left(-1,1\right)$ and $\mathcal{U}\left(-25,25\right)^{d}$,
respectively, and set matrices $\boldsymbol{\Sigma}_{l}=U_{l}D_{l}U_{l}^{T}$,
where $U_{l}$ is a $d\times d$ random unitary matrix and $D_{l}$
is a $d\times d$ diagonal matrix with positive entries sampled from
the uniform distribution $\mathcal{U}\left(0,0.01\right)^{d}$. We
report dimensions $d$, the number of initial terms $N$, the number
of skeleton terms $r$ and the running time $T$ in seconds in Table~\ref{tab:Timing-of-Algorithm-gram complexity}.
We implemented all algorithms in Fortran90 and compile them with Intel
Fortran Compiler version 18.0.3. The computations are performed on
a single core (without parallelization) using Intel i7-6700 CPU @3.4
GHz on a 64-bit Linux workstation with 64 GB of RAM.
\begin{table}
\centering{}%
\begin{tabular}{|c|c|c|c|c|c|c|c|}
\hline 
\multicolumn{4}{|c|}{$d=1$} & \multicolumn{4}{c|}{$d=5$}\tabularnewline
\hline 
\hline 
\multicolumn{2}{|c|}{Fixed $r=10^{2}$} & \multicolumn{2}{c|}{Fixed $N=10^{4}$} & \multicolumn{2}{c|}{Fixed $r=10^{2}$} & \multicolumn{2}{c|}{Fixed $N=10^{4}$}\tabularnewline
\hline 
$N\left(\times10^{4}\right)$ & $T$ & $r\left(\times10^{2}\right)$ & $T$ & $N\left(\times10^{4}\right)$ & $T$ & $r\left(\times10^{2}\right)$ & $T$\tabularnewline
\hline 
$1$ & $0.0960$ & $1$ & $0.0960$ & $1$ & $5.21$ & $1$ & $5.21$\tabularnewline
\hline 
$2$ & $0.264$ & $2$ & $0.356$ & $2$ & $10.5$ & $2$ & $10.6$\tabularnewline
\hline 
$4$ & $0.612$ & $4$ & $1.17$ & $4$ & $20.7$ & $4$ & $21.2$\tabularnewline
\hline 
$8$ & $1.29$ & $8$ & $3.80$ & $8$ & $42.2$ & $8$ & $43.0$\tabularnewline
\hline 
$16$ & $2.06$ & $16$ & $11.9$ & $16$ & $84.3$ & $16$ & $86.9$\tabularnewline
\hline 
\end{tabular}\caption{\label{tab:Timing-of-Algorithm-gram complexity}Algorithm~\ref{alg:Reduction-algorithm-using-Gram}
timings $T$ (in seconds) to be compared with the theoretical complexity
estimate $\mathcal{O}\left(r^{2}N+p\left(d\right)r\thinspace N\right)$.
We observe the linear dependence on the number of initial terms $N$.
Also, timings reveal that in dimension $d=5$ the cost of computing
inner products dominates and, as a result, timing is practically linear
as well. We note that the algorithm was implemented on a single core
without any parallelization.}
\end{table}

\subsection{Applications of reduction algorithms}

In the following sections, we present several examples of application
of reduction algorithms in both low and high dimensions. In Section~\ref{sec:Applications-of-reduction}
we represent solutions of differential and integral equations in a
functional form and adaptively solve these equations. We start with
the free space Poisson's equation with non-separable right hand side
and then present an example of solving an elliptic problem with variable
coefficients; we consider both examples in dimensions $d=3$ through
$d=7$. In Section~\ref{sec:Kernel-Density-Estimation} we first
use our algorithm in dimension $d=1$ to construct an efficient representation
of the PDF of a cloud of points via kernel density estimation and
compare it with results obtained via the usual approach. We then present
an example of constructing PDFs in high dimensions. Finally, we turn
to kernel summation methods in high dimensions in Section~\ref{sec:Far-field-summation-in},
consider far-field evaluation in such computations and explore the
problem of constructing equivalent sources in a similar setup. We
also illustrate how a reduction algorithm can be used to partition
points into groups (in a hierarchical fashion if desired).

\section{\label{sec:Applications-of-reduction}Reduction algorithms for solving
differential and integral equations}

\subsection{Poisson equation in free space in high dimensions\label{subsec:Poisson-equation-in}}

The Poisson's equation 
\begin{equation}
-\Delta u\mathbf{\left(x\right)}=f\left(\mathbf{x}\right),\ \ \ \mathbf{x}\in\mathbb{R}^{d},\label{eq:poissoneqn}
\end{equation}
arises in numerous applications in nearly all field of physics and
computational chemistry (see e.g.\ \cite{G-D-N-G-B:2006}). The Reduction
Algorithm~\ref{alg:Reduction-algorithm-using-Gram} allows us to
solve this equation in dimensions $d\geq3$ assuming that the charge
distribution $f\left(\mathbf{x}\right)$ is given by, e.g.\ , a linear
combination of multivariate Gaussian atoms. We obtain the solution
via 
\begin{equation}
u\mathbf{\left(x\right)}=\int_{\mathbb{R}^{d}}G\left(\mathbf{x}-\mathbf{y}\right)f\mathbf{\left(y\right)}d\mathbf{y},\label{eq:Solution Poisson's eq via integral}
\end{equation}
where the free-space Green's function for (\ref{eq:poissoneqn}) is
given by the radial function
\begin{equation}
G\mathbf{\left(x\right)}=C_{d}\mathbf{\left\Vert x\right\Vert }^{2-d},\,\,\,\,C_{d}=\frac{\Gamma\left(\frac{d}{2}+1\right)}{d\left(d-2\right)\pi^{\frac{d}{2}}},\label{eq:poissongreenfun}
\end{equation}
where $\mathbf{\left\Vert \cdot\right\Vert }=\mathbf{\left\Vert \cdot\right\Vert }_{2}$
is the standard $l^{2}$-norm. In order to evaluate the integral (\ref{eq:Solution Poisson's eq via integral}),
we approximate the Green's function $G$ via a linear combination
of Gaussians (see e.g.\ \cite{H-F-Y-G-B:2004,BEY-MON:2005,BEY-MON:2010}). 

The error estimates in \cite[Theorem 3]{BEY-MON:2010} are based on
discretizing the integral 
\[
\frac{1}{r^{d-2}}=\frac{1}{\Gamma\left(\frac{d-2}{2}\right)}\int_{-\infty}^{\infty}e^{-r^{2}e^{t}+\frac{d-2}{2}t}dt
\]
as
\begin{equation}
G_{\infty}\left(r,h\right)=\frac{C_{d}h}{\Gamma\left(\frac{d-2}{2}\right)}\sum_{l\in\mathbb{Z}}e^{hl\left(d-2\right)/2}e^{-e^{hl}r^{2}},\label{eq:ApproxViaGaussians}
\end{equation}
where the step size $h$ satisfies
\begin{equation}
h\leq\frac{2\pi}{\log3+\frac{d-2}{2}\log(\cos1)^{-1}+\log\epsilon^{-1}}\label{hEstimate}
\end{equation}
and $\epsilon$ is any user-selected accuracy. Then \cite[Theorem 3]{BEY-MON:2010}
implies
\begin{equation}
\left|G\left(r\right)-G_{\infty}\left(r,h\right)\right|\leq\epsilon\thinspace G\left(r\right),\,\,\,\,\,\textrm{\mbox{ for all }\ensuremath{r>0}}.\label{ErrorSeriesApproximationOfG}
\end{equation}
To estimate the error of approximating the solution $u$ in (\ref{eq:Solution Poisson's eq via integral})
using the series (\ref{eq:ApproxViaGaussians}) instead of the Green's
function (\ref{eq:poissongreenfun}), we first prove the following
lemma.
\begin{lem}
\label{Lemma:ErrorApproxPoissonSolutionPositiveRHS} For any $d\geq3$,
$e^{-1}\ge\epsilon>0$, and $f$ nonnegative in (\ref{eq:poissoneqn}),
there exist a step size $h$ such that
\begin{equation}
\left|u\mathbf{\left(x\right)}-\int_{\mathbb{R}^{d}}G_{\infty}\left(\left\Vert \mathbf{x}-\mathbf{y}\right\Vert ,h\right)f\mathbf{\left(y\right)}d\mathbf{y}\right|\leq\epsilon\left|u\mathbf{\left(x\right)}\right|,\,\,\,\textrm{\mbox{ for all }\,\,\,\ensuremath{\mathbf{x}}}\neq0.\label{eq:thmPowAsGaussians}
\end{equation}
\end{lem}

\begin{proof}
From (\ref{eq:Solution Poisson's eq via integral}) and (\ref{ErrorSeriesApproximationOfG}),
we have
\begin{align*}
\left|u\mathbf{\left(x\right)}-\int_{\mathbb{R}^{d}}G_{\infty}\left(\left\Vert \mathbf{x}-\mathbf{y}\right\Vert ,h\right)f\mathbf{\left(y\right)}d\mathbf{y}\right| & =\left|\int_{\mathbb{R}^{d}}\left[G\left(\mathbf{x}-\mathbf{y}\right)-G_{\infty}\left(\left\Vert \mathbf{x}-\mathbf{y}\right\Vert ,h\right)\right]f\mathbf{\left(y\right)}d\mathbf{y}\right|\\
 & \leq\int_{\mathbb{R}^{d}}\left|G\left(\mathbf{x}-\mathbf{y}\right)-G_{\infty}\left(\left\Vert \mathbf{x}-\mathbf{y}\right\Vert ,h\right)\right|\left|f\mathbf{\left(y\right)}\right|d\mathbf{y}\\
 & \leq\epsilon\int_{\mathbb{R}^{d}}G\left(\mathbf{x}-\mathbf{y}\right)\left|f\mathbf{\left(y\right)}\right|d\mathbf{y}=\epsilon\int_{\mathbb{R}^{d}}G\left(\mathbf{x}-\mathbf{y}\right)f\mathbf{\left(y\right)}d\mathbf{y}\\
 & =\epsilon\thinspace u\mathbf{\left(x\right)}=\epsilon\left|u\mathbf{\left(x\right)}\right|.
\end{align*}
\end{proof}
In our examples, we always consider functions $f$ represented in
the form
\begin{equation}
f\mathbf{\left(x\right)}=\sum_{l=1}^{N}c_{l}g_{l}\left(\mathbf{x},\boldsymbol{\mu}_{l},\boldsymbol{\Sigma}_{l}\right),\label{eq:f_as_gaussian_mixture}
\end{equation}
for some Gaussian atoms as in (\ref{eq:Gaussian_atoms}). In particular,
\[
u\mathbf{\left(x\right)}=\sum_{l=1}^{N}c_{l}\int_{\mathbb{R}^{d}}G\left(\mathbf{x}-\mathbf{y}\right)g_{l}\left(\mathbf{x},\boldsymbol{\mu}_{l},\boldsymbol{\Sigma}_{l}\right)d\mathbf{y}
\]
 and 
\[
u^{+}\mathbf{\left(x\right)}=\sum_{l=1}^{N}\left|c_{l}\right|\int_{\mathbb{R}^{d}}G\left(\mathbf{x}-\mathbf{y}\right)g_{l}\left(\mathbf{x},\boldsymbol{\mu}_{l},\boldsymbol{\Sigma}_{l}\right)d\mathbf{y}
\]
 are both bounded. We assume that 
\begin{equation}
\left\Vert u^{+}\right\Vert _{L^{\infty}}\leq c\left\Vert u\right\Vert _{L^{\infty}}\label{eq:assumption_inf_norm}
\end{equation}
 for a moderate size constant $c$. This assumption prevents representations
of $u$ that involve large coefficients $c_{l}$ of opposite signs.
We then have
\begin{lem}
\label{Lemma:ErrorApproxPoissonSolution} Let $d\geq3$, $e^{-1}\ge\epsilon>0$,
and $f$ as in (\ref{eq:f_as_gaussian_mixture}). If (\ref{eq:assumption_inf_norm})
holds, then there exist a step size $h$ such that
\begin{equation}
\left|\int_{\mathbb{R}^{d}}G\left(\mathbf{x}-\mathbf{y},h\right)f\mathbf{\left(y\right)}d\mathbf{y}-\int_{\mathbb{R}^{d}}G_{\infty}\left(\left\Vert \mathbf{x}-\mathbf{y}\right\Vert ,h\right)f\mathbf{\left(y\right)}d\mathbf{y}\right|\leq\epsilon\thinspace c\left|u\mathbf{\left(x\right)}\right|,\,\,\,\textrm{\mbox{ for all }\,\,\,\ensuremath{\mathbf{x}}}\neq0.\label{eq:thmPowAsGaussians-2}
\end{equation}
\end{lem}

\begin{proof}
It follows from Lemma~\ref{Lemma:ErrorApproxPoissonSolutionPositiveRHS}
that 
\[
\left|u\mathbf{\left(x\right)}-\tilde{u}\mathbf{\left(x\right)}\right|\leq\epsilon u^{+}\mathbf{\left(x\right)}
\]
where 
\[
\tilde{u}\mathbf{\left(x\right)}=\int_{\mathbb{R}^{d}}G_{\infty}\left(\left\Vert \mathbf{x}-\mathbf{y}\right\Vert ,h\right)f\mathbf{\left(y\right)}d\mathbf{y}.
\]
Therefore, using (\ref{eq:assumption_inf_norm}) , we have
\[
\left\Vert u\mathbf{\left(x\right)}-\tilde{u}\right\Vert _{L^{\infty}}\leq\epsilon u^{+}\mathbf{\left(x\right)}\leq\epsilon\thinspace c\left\Vert u\right\Vert _{L^{\infty}}.
\]
\end{proof}
In practice, when using (\ref{eq:ApproxViaGaussians}), we truncate
the sum 
\begin{equation}
G_{M,N}\left(r\right)=\frac{C_{d}h}{\Gamma\left(\frac{d-2}{2}\right)}\sum_{l=M}^{N}e^{hl\left(d-2\right)/2}e^{-e^{hl}r^{2}}=\frac{C_{d}h}{\Gamma\left(\frac{d-2}{2}\right)}\sum_{l=1}^{N_{terms}}e^{h\left(M+l-1\right)\left(d-2\right)/2}e^{-e^{h\left(M+l-1\right)}r^{2}}\label{eq:ApproxViaGaussians-finite}
\end{equation}
so that the removed terms contribute less than $\epsilon$ and we
limit the range of $r$ to some interval of the form $\left[\delta,R\right]$;
the resulting approximation has $N_{terms}=N-M+1$. In our computations
we set the range of $r$ to be $\left[10^{-10},10^{10}\right]$ and
the accuracy to be $\epsilon=10^{-14}$. As shown in Table~\ref{table_poisson-single-rhs},
the number of terms in (\ref{eq:ApproxViaGaussians-finite}) depends
on the dimension only weakly. 

As an illustration, we first demonstrate our approach for a single
Gaussian,
\begin{equation}
f\left(\mathbf{x}\right)=e^{-\frac{1}{2}\left(\mathbf{x}-\boldsymbol{\mu}\right)^{T}\boldsymbol{\Sigma}^{-1}\left(\mathbf{x}-\boldsymbol{\mu}\right)}.\label{eq:RHSsingleGaussian}
\end{equation}
Using (\ref{eq:ApproxViaGaussians-finite}), we approximate the solution
$u$ by
\begin{eqnarray}
u_{\epsilon}\mathbf{\left(x\right)} & = & \frac{C_{d}h}{\Gamma\left(\frac{d-2}{2}\right)}\sum_{l=1}^{N_{terms}}e^{h\left(M+l-1\right)\left(d-2\right)/2}\int_{\mathbb{R}^{d}}e^{-e^{h\left(M+l-1\right)}\left\Vert \mathbf{x}-\mathbf{y}\right\Vert ^{2}}f\mathbf{\left(y\right)}d\mathbf{y}\nonumber \\
 & = & \frac{C_{d}h}{\Gamma\left(\frac{d-2}{2}\right)}\sum_{l=1}^{N_{terms}}e^{h\left(M+l-1\right)\left(d-2\right)/2}\int_{\mathbb{R}^{d}}e^{-e^{h\left(M+l-1\right)}\left\Vert \mathbf{x}-\mathbf{y}\right\Vert ^{2}}e^{-\frac{1}{2}\left(\mathbf{y}-\boldsymbol{u}\right)^{T}\boldsymbol{\Sigma}^{-1}\left(\mathbf{y}-\boldsymbol{\mu}\right)}d\mathbf{y}.
\end{eqnarray}
Evaluating the integral explicitly (see Appendix~A for details),
we obtain
\begin{equation}
u_{\epsilon}\mathbf{\left(x\right)}=\sum_{l=1}^{N_{terms}}c_{l}g_{l}\left(\mathbf{x},\boldsymbol{\mu}_{l},\boldsymbol{\Sigma}_{l}\right)\label{eq:u_epsilon}
\end{equation}
where
\[
c_{l}=\frac{C_{d}h\pi^{\frac{3d}{4}}}{\Gamma\left(\frac{d-2}{2}\right)}e^{-h\left(M+l-1\right)\left(d+2\right)/2}\frac{\left(\det\Sigma\right)^{\frac{1}{2}}}{\left(\det\Sigma_{l}\right)^{\frac{1}{4}}},
\]
\[
\boldsymbol{\Sigma}_{l}=\Sigma+\begin{bmatrix}\frac{1}{2e^{h\left(M+l-1\right)}} &  &  & 0\\
 & \frac{1}{2e^{h\left(M+l-1\right)}}\\
 &  & \ddots\\
0 &  &  & \frac{1}{2e^{h\left(M+l-1\right)}}
\end{bmatrix},
\]
and
\[
\boldsymbol{\mu}_{l}=\boldsymbol{\mu}.
\]

The number of terms in the representation of $u_{\epsilon}$ is excessive
and we reduce it using Algorithm~\ref{alg:Reduction-algorithm-using-Gram}
to obtain our final approximation as 
\begin{equation}
\tilde{u}\left(\mathbf{x}\right)=\sum_{m=1}^{\widetilde{N}}\tilde{c}_{i_{m}}g_{i_{m}}\left(\mathbf{x},\boldsymbol{\mu}_{i_{m}},\boldsymbol{\Sigma}_{i_{m}}\right),\label{eq:U_tilde}
\end{equation}
where $\tilde{c}_{i_{m}}$ are the new coefficients and $i_{m}\in\widehat{I}$
(see Algorithm~\ref{alg:Reduction-algorithm-using-Gram} for details).
\begin{rem}
The representation of the kernel in (\ref{eq:ApproxViaGaussians-finite})
can be obtained for a large spatial range since the number of terms
$N_{terms}$ is proportional to the logarithm of the range. For this
reason our approach is viable in high dimensions while employing the
Fast Fourier Transform is not an option due to the size of the Fourier
domain, c.f. \cite{VI-GR-FE:2016}. 
\end{rem}

In order to demonstrate the performance of our approach, we choose
the right hand side $f$ to be a Gaussian mixture with $100$ terms,
\[
f\left(\mathbf{x}\right)=\sum_{i=1}^{100}c_{f_{i}}e^{-\frac{1}{2}\left(\mathbf{x}-\boldsymbol{\mu}_{f_{i}}\right)^{T}\boldsymbol{\Sigma}_{f_{i}}^{-1}\left(\mathbf{x}-\boldsymbol{\mu}_{f_{i}}\right)}.
\]
In the Gaussian mixture $f$, the coefficients $c_{f_{i}}$ and means
$\boldsymbol{\mu}_{f_{i}}$ are sampled from a one and a $d$-dimensional
standard normal distributions respectively. The symmetric positive
definite matrices $\boldsymbol{\Sigma}_{f_{i}}$ are constructed as
\begin{equation}
\boldsymbol{\Sigma}_{f_{i}}=U_{i}^{T}U_{i}+\frac{1}{10}I_{d},\label{eq:construction of sigma}
\end{equation}
where $U_{i}$ is a $d\times d$ matrix of standard normally distributed
numbers and $I_{d}$ is the $d\times d$ identity matrix. We obtain
$u_{\epsilon}$ in (\ref{eq:u_epsilon}) and apply Algorithm~\ref{alg:Reduction-algorithm-using-Gram}
to reduce the number of terms to obtain $\tilde{u}$ in (\ref{eq:U_tilde}).
The results are displayed in Table~\ref{table_poisson-100rhs}, where
we show the dimension of the problem, $d$, the number of terms, $N_{terms}$,
in the approximation of the Green's function, the number of terms,
$N_{tot}$, in the solution $u_{\epsilon}$ before reduction and its
accuracy, the number of terms, $\widetilde{N}$, in the solution $\widetilde{u}$
after reduction and its accuracy, and, finally, the relative error
between $u_{\epsilon}$ and $\widetilde{u}$. 

In order to estimate the accuracy of the solution $\tilde{u}$ of
(\ref{eq:poissoneqn}), we define the errors 
\[
h_{\epsilon}\left(\mathbf{x}\right)=-\Delta u_{\epsilon}\left(\mathbf{x}\right)-f\left(\mathbf{x}\right),\ \ \ \widetilde{h}\left(\mathbf{x}\right)=-\Delta\tilde{u}\left(\mathbf{x}\right)-f\left(\mathbf{x}\right),\ \ \ h\left(\mathbf{x}\right)=u_{\epsilon}\left(\mathbf{x}\right)-\tilde{u}\left(\mathbf{x}\right).
\]
The usual approach to ascertain the size of $h_{\epsilon}$, $\widetilde{h}$,
and $h$ by evaluating them on a lattice of grid points is impractical
in high dimensions. Instead, we compute values of $h_{\epsilon}$,
$\widetilde{h}$ and $h$ at a collection of points in principle directions
of the right hand side $f\left(\mathbf{x}\right)$. To be precise,
we first solve the eigenvalue problem for all matrices $\Sigma_{f_{i}}$
, $i=1,\dots,100$, 
\[
\Sigma_{f_{i}}=\sum_{j=1}^{d}\lambda_{j}^{\left(i\right)}\mathbf{v}_{j}^{\left(i\right)}\left(\mathbf{v}_{j}^{\left(i\right)}\right)^{T}.
\]
Here the eigenvectors $\mathbf{v}_{j}^{\left(i\right)}$ identify
principle directions for each Gaussian in $f$ so that we can select
an appropriate set of samples along those directions. We note that
the actual range of the eigenvalues of matrices $\Sigma_{f_{i}}$,
$\left\{ \lambda_{j}^{\left(i\right)}\right\} _{{i=1,..,100\atop j=1,\dots,d}}$,
is $\left[\frac{1}{10},40\right]$. Next, for each pair of $\left\{ \lambda_{j}^{\left(i\right)},\mathbf{v}_{j}^{\left(i\right)}\right\} $,
we find an interval $\left[-s_{j}^{\left(i\right)},s_{j}^{\left(i\right)}\right]$
by solving
\[
e^{-\frac{\left(s_{j}^{\left(i\right)}\right)^{2}}{2\lambda_{j}^{\left(i\right)}}}=10^{-10}\Leftrightarrow s_{j}^{\left(i\right)}=\left(-2\lambda_{j}^{\left(i\right)}\log10^{-10}\right)^{1/2}
\]
and generate and equally-spaced grid in $\left[-s_{j}^{\left(i\right)},s_{j}^{\left(i\right)}\right]$
as
\[
s_{jk}^{\left(i\right)}=-s_{j}^{\left(i\right)}+\left(k-1\right)\frac{2s_{j}^{\left(i\right)}}{N_{s}-1},\ \ \ k=1,\dots N_{s}.
\]
Finally, we select sample points 
\[
\mathbf{x}_{jk}^{\left(i\right)}=s_{jk}^{\left(i\right)}\mathbf{v}_{j}^{\left(i\right)}+\boldsymbol{\mu}_{f_{i}}
\]
and evaluate $h_{\epsilon}\left(\mathbf{x}_{jk}^{\left(i\right)}\right)$,
$\widetilde{h}\left(\mathbf{x}_{jk}^{\left(i\right)}\right)$ and
$h\left(\mathbf{x}_{jk}^{\left(i\right)}\right)$ for $i=1,\dots,100$,
$j=1,\dots,d$ and $k=1,\dots,N_{s}$. In our experiment, we set $N_{s}=10$,
and report the resulting errors in Table~\ref{table_poisson-single-rhs}
and \ref{table_poisson-100rhs}. For these two tables, we use the
notation $\left\Vert f\right\Vert _{\infty}=\max_{i,j,k}\left|f\left(\mathbf{x}_{jk}^{\left(i\right)}\right)\right|$. 

\begin{table}[h]
\begin{centering}
\begin{tabular}{|c|c|c|c|c|c|c|}
\hline 
$d$ & $N_{terms}$  & $N_{tot}$ & $\left\Vert h_{\epsilon}\right\Vert _{\infty}/\left\Vert f\right\Vert _{\infty}$ & $\widetilde{N}$ & $\left\Vert \widetilde{h}\right\Vert _{\infty}/\left\Vert f\right\Vert _{\infty}$ & $\left\Vert h\right\Vert _{\infty}/\left\Vert u_{\epsilon}\right\Vert _{\infty}$\tabularnewline
\hline 
\hline 
$3$ & $345$ & $247$ & $1.2e-9$ & $183$ & $1.8e-8$ & $4.0e-8$\tabularnewline
\hline 
$4$ & $397$ & $295$ & $1.3e-9$ & $227$ & $1.2e-7$ & $9.8e-8$\tabularnewline
\hline 
$5$ & $386$ & $282$ & $5.6e-10$ & $207$ & $7.8e-8$ & $1.7e-9$\tabularnewline
\hline 
$6$ & $343$ & $279$ & $1.9e-10$ & $197$ & $1.8e-7$ & $2.6e-9$\tabularnewline
\hline 
$7$ & $354$ & $240$ & $1.7e-10$ & $154$ & $2.6e-7$ & $3.9e-9$\tabularnewline
\hline 
\end{tabular}
\par\end{centering}
~\\
~
\centering{}\caption{\label{table_poisson-single-rhs} Number of terms and relative errors
of solving Poisson's equation in dimensions $d=3,\dots,7$ where the
forcing term is a single randomly generated multivariate Gaussian.
The number of Gaussians to represent the Green's function in (\ref{eq:ApproxViaGaussians-finite})
is $N_{terms}$, the number of terms of $u_{\epsilon}$ in (\ref{eq:u_epsilon})
after truncation of coefficients to $10^{-10}$ is $N_{tot}$ and,
after applying Algorithm~\ref{alg:Reduction-algorithm-using-Gram},
the number of terms of $\tilde{u}$ in (\ref{eq:U_tilde}) is $\tilde{N}$.}
\end{table}
\begin{table}[h]
\begin{centering}
\begin{tabular}{|c|c|c|c|c|c|c|}
\hline 
$d$ & $N_{terms}$  & $N_{tot}$ & $\left\Vert h_{\epsilon}\right\Vert _{\infty}/\left\Vert f\right\Vert _{\infty}$ & $\widetilde{N}$ & $\left\Vert \widetilde{h}\right\Vert _{\infty}/\left\Vert f\right\Vert _{\infty}$ & $\left\Vert h\right\Vert _{\infty}/\left\Vert u_{\epsilon}\right\Vert _{\infty}$\tabularnewline
\hline 
\hline 
$3$ & $345$ & $24694$ & $9.1e-10$ & $2978$ & $2.6e-4$ & $1.4e-6$\tabularnewline
\hline 
$4$ & $397$ & $29564$ & $5.1e-10$ & $3910$ & $4.4e-5$ & $3.7e-7$\tabularnewline
\hline 
$5$ & $386$ & $28103$ & $2.8e-10$ & $4602$ & $5.1e-5$ & $4.3e-7$\tabularnewline
\hline 
$6$ & $343$ & $27813$ & $1.7e-10$ & $5111$ & $9.3e-6$ & $1.3e-7$\tabularnewline
\hline 
$7$ & $354$ & $24153$ & $7.2e-11$ & $5591$ & $3.0e-6$ & $8.3e-8$\tabularnewline
\hline 
\end{tabular}
\par\end{centering}
~\\
~
\centering{}\caption{\label{table_poisson-100rhs}Number of terms and relative errors of
solving Poisson's equation in dimensions $d=3,\dots,7$ where the
forcing term is a linear combination of $100$ randomly generated
multivariate Gaussians. The information displayed in each column is
described in Table~\ref{table_poisson-single-rhs}. The number of
terms $\widetilde{N}$ is significantly larger than that in Table~\ref{table_poisson-single-rhs}
since the principle directions of the matrices $\boldsymbol{\Sigma}_{f_{i}}$
are chosen at random causing the solution to have a larger number
of terms. }
\end{table}

\subsection{Second order elliptic equation with a variable coefficient}

In this example, we consider the second order linear elliptic equation,
\begin{equation}
-\nabla\cdot\left(a\left(\mathbf{x}\right)\nabla u\left(\mathbf{x}\right)\right)+k^{2}u\left(\mathbf{x}\right)=f\left(\boldsymbol{\mathbf{x}}\right),\ \ \ \boldsymbol{\mathbf{x}}\in\mathbb{R}^{d}\label{eq:variable coefficent eq}
\end{equation}
where $d\geq3$ and $k>0.$ We assume that the variable coefficient
$a\left(\boldsymbol{x}\right)$ is of the form 
\begin{equation}
a\left(\mathbf{x}\right)=1+e^{-\frac{1}{2}\left(\mathbf{x}-\boldsymbol{\mu}_{a}\right)^{t}\Sigma_{a}^{-1}\left(\mathbf{x}-\boldsymbol{\mu}_{a}\right)},\label{eq:variable_coeff_a}
\end{equation}
such that $\max_{\mathbf{x}}\left|a\left(\mathbf{x}\right)\right|/\min_{\mathbf{x}}\left|a\left(\mathbf{x}\right)\right|=2$,
and choose the forcing function to be
\begin{equation}
f\left(\mathbf{x}\right)=e^{-\frac{1}{2}\left(\mathbf{x}-\boldsymbol{\mu}_{f}\right)^{t}\Sigma_{f}^{-1}\left(\mathbf{x}-\boldsymbol{\mu}_{f}\right)},\ \ \ \left\Vert f\right\Vert _{L^{\infty}}=1.\label{eq:forcing_function_f}
\end{equation}
The free space Green's function for the problem with a constant coefficient
\[
-\Delta u\left(\mathbf{x}\right)+k^{2}u\left(\mathbf{x}\right)=f\left(\mathbf{x}\right),\ \ \ \mathbf{x}\in\mathbb{R}^{d}
\]
is given by
\begin{equation}
G\left(\mathbf{x}\right)=\left(2\pi\right)^{-\frac{d}{2}}\left(\frac{k}{\left\Vert \mathbf{x}\right\Vert }\right)^{\frac{d}{2}-1}K_{\frac{d}{2}-1}\left(k\left\Vert \mathbf{x}\right\Vert \right),\label{eq:bound state helmholtz green fun}
\end{equation}
where $K_{\frac{d}{2}-1}$ is a modified Bessel function of the second
kind of order $\frac{d}{2}-1$. We approximate the Green's function
(\ref{eq:bound state helmholtz green fun}) by discretizing the integral
\[
G\left(\mathbf{x}\right)=\left(4\pi\right)^{-\frac{d}{2}}\int_{-\infty}^{\infty}e^{-\frac{\left\Vert \boldsymbol{x}\right\Vert ^{2}e^{t}}{4}-k^{2}e^{-t}+\left(\frac{d}{2}-1\right)t}dt
\]
as
\[
G_{\infty}\left(r\right)=\left(4\pi\right)^{-\frac{d}{2}}h\sum_{l\in\mathbb{Z}}e^{-\frac{r^{2}e^{hl}}{4}-k^{2}e^{-hl}+\left(\frac{d}{2}-1\right)hl},
\]
where the step size $h$ is selected to achieve the desired accuracy
$\epsilon$ (see \cite{BEY-MON:2010}). We then truncate the sum in
the same manner as in (\ref{eq:ApproxViaGaussians-finite}) and obtain
\begin{eqnarray*}
G_{M,N}\left(r\right) & = & \left(4\pi\right)^{-\frac{d}{2}}h\sum_{l=M}^{N}e^{-\frac{r^{2}e^{hl}}{4}-k^{2}e^{-hl}+\left(\frac{d}{2}-1\right)hl}\\
 & = & \left(4\pi\right)^{-\frac{d}{2}}h\sum_{l=1}^{N_{terms}}e^{-\frac{r^{2}e^{h\left(l+M-1\right)}}{4}-k^{2}e^{-h\left(l+M-1\right)}+\left(\frac{d}{2}-1\right)h\left(l+M-1\right)},
\end{eqnarray*}
where the number of terms $N_{terms}=M-N+1$ in $G_{M,N}$ weakly
depends on the dimension (see Table\ \ref{table variable coefficient not aligned}).
In our computation, $k=1$ and we select the accuracy range of $G_{M,N}$
to be $\left[10^{-7},10^{2}\right]$ with $\epsilon=10^{-10}$. 

We rewrite (\ref{eq:variable coefficent eq}) as an integral equation,
\begin{equation}
u\left(\mathbf{x}\right)-\int_{\mathbb{R}^{d}}G\left(\mathbf{x}-\mathbf{y}\right)\nabla\cdot\left(\left(a\left(\mathbf{y}\right)-1\right)\nabla u\left(\mathbf{y}\right)\right)d\mathbf{y}=\int_{\mathbb{R}^{d}}G\left(\mathbf{x}-\mathbf{y}\right)f\left(\mathbf{y}\right)d\mathbf{y}.\label{eq:variable coefficient integral eq}
\end{equation}
In order to solve (\ref{eq:variable coefficient integral eq}), we
first observe that, in a multiresolution basis, for a finite accuracy
$\epsilon>0$, the non-standard form (see \cite{BE-CO-RO:1991}) of
the Green\textquoteright s function for (\ref{eq:variable coefficient integral eq})
is banded on all scales. This implies that a set of basis functions
that can represent the solution $u$ is fully determined by the size
of the bands of the multiresolution representation of the Green\textquoteright s
function and of the right hand side $f$. This suggests that due to
the interaction between the essential supports of the functions involved,
we can identify a set of Gaussians atoms by performing one (or a few
more) iterations of the integral equation (\ref{eq:variable coefficient integral eq})
even if the fixed-point iteration does not converge. In this approach
the accuracy is determined \textit{a posteriori} and can be improved
by additional iterations. From the so generated set of atoms, we obtain
a basis of Gaussian atoms $g_{l}\left(\mathbf{x},\boldsymbol{\mu}_{l},\boldsymbol{\Sigma}_{l}\right)$
by applying the reduction algorithm to identify the best linearly
independent subset. Using this basis, we define the ansatz for $u$
as 
\begin{equation}
\tilde{u}\left(\mathbf{x}\right)=\sum_{l=1}^{\widetilde{N}}c_{l}g_{l}\left(\mathbf{x},\boldsymbol{\mu}_{l},\boldsymbol{\Sigma}_{l}\right),\label{eq:ansatz}
\end{equation}
for some (unknown) coefficients $c_{l}$, $l=1,\cdots\widetilde{N}$,
to be determined; substituting $\tilde{u}$ into either (\ref{eq:variable coefficient integral eq})
or the differential equation (\ref{eq:variable coefficent eq}) and
computing appropriate inner products, we solve a system of linear
algebraic equations for the coefficients $c_{l}$. 

Specifically, we rewrite (\ref{eq:variable coefficient integral eq})
as
\[
u\left(\mathbf{x}\right)=\int_{\mathbb{R}^{d}}G\left(\mathbf{x}-\mathbf{y}\right)\nabla\cdot\left(\left(a\left(\mathbf{y}\right)-1\right)\nabla u\left(\mathbf{y}\right)\right)d\mathbf{y}+\int_{\mathbb{R}^{d}}G\left(\mathbf{x}-\mathbf{y}\right)f\left(\mathbf{y}\right)d\mathbf{y}
\]
which leads to the iteration,
\begin{eqnarray}
u_{n+1}\left(\mathbf{x}\right) & = & u_{0}\left(\mathbf{x}\right)+\int_{\mathbb{R}^{d}}G\left(\mathbf{x}-\mathbf{y}\right)\nabla\cdot\left(\left(a\left(\mathbf{y}\right)-1\right)\nabla u_{n}\left(\mathbf{y}\right)\right)d\mathbf{y}\label{eq:PDE iteration}\\
u_{0}\left(\mathbf{x}\right) & = & \int_{\mathbb{R}^{d}}G\left(\mathbf{x}-\mathbf{y}\right)f\left(\mathbf{y}\right)d\mathbf{y}.\nonumber 
\end{eqnarray}
To identify a set of Gaussian atoms, we perform one (or several) iteration(s),
using as the initial $u_{0}$ the collection of atoms in the representation
of $\int_{\mathbb{R}^{d}}G\left(\mathbf{x}-\mathbf{y}\right)f\left(\mathbf{y}\right)d\mathbf{y}$.
Using Algorithm~\ref{alg:Reduction-algorithm-using-Gram} we then
reduce the number of atoms by removing linearly dependent terms (we
may repeat this step if we need to improve accuracy). As a result,
we determine a basis of Gaussian atoms $g_{l}\left(\mathbf{x},\boldsymbol{\mu}_{l},\boldsymbol{\Sigma}_{l}\right)$
to represent the solution of equation (\ref{eq:variable coefficent eq})
as in (\ref{eq:ansatz}). To find the coefficients $c_{l}$, $l=1,\cdots\widetilde{N}$,
we substitute (\ref{eq:ansatz}) into the weak formulation of (\ref{eq:variable coefficent eq})
to obtain the linear system
\begin{equation}
\sum_{l=1}^{\widetilde{N}}c_{l}\left\langle -\nabla\cdot\left(a\left(\mathbf{x}\right)\nabla g_{l}\left(\mathbf{x},\boldsymbol{\mu}_{l},\boldsymbol{\Sigma}_{l}\right)\right),g_{k}\left(\mathbf{x},\boldsymbol{\mu}_{k},\boldsymbol{\Sigma}_{k}\right)\right\rangle =\left\langle f\left(\mathbf{x}\right),g_{k}\left(\mathbf{x},\boldsymbol{\mu}_{k},\boldsymbol{\Sigma}_{k}\right)\right\rangle ,\ k=1,\cdots,\widetilde{N}.\label{eq:Linear system}
\end{equation}
The inner products $\left\langle -\nabla\cdot\left(a\left(\mathbf{x}\right)\nabla g_{l}\left(\mathbf{x},\boldsymbol{\mu}_{l},\boldsymbol{\Sigma}_{l}\right)\right),g_{k}\left(\mathbf{x},\boldsymbol{\mu}_{k},\boldsymbol{\Sigma}_{k}\right)\right\rangle $
are computed explicitly using integration by parts and the fact that
\[
\nabla_{\boldsymbol{x}}e^{-\frac{1}{2}\left(\mathbf{x}-\boldsymbol{\mu}\right)^{t}\Sigma^{-1}\left(\mathbf{x}-\boldsymbol{\mu}\right)}=-\nabla_{\boldsymbol{\mu}}e^{-\frac{1}{2}\left(\mathbf{x}-\boldsymbol{\mu}\right)^{t}\Sigma^{-1}\left(\mathbf{x}-\boldsymbol{\mu}\right)},
\]
leading to integrals involving only Gaussians (the result is then
differentiated with respect to the shift parameter $\boldsymbol{\mu}$). 

Using the SVD, we solve the linear system (\ref{eq:Linear system})
to obtain an approximate solution
\[
\widetilde{u}\left(\mathbf{x}\right)=\sum_{l=1}^{\widetilde{N}}c_{l}g_{l}\left(\mathbf{x},\boldsymbol{\mu}_{l},\boldsymbol{\Sigma}_{l}\right).
\]

\subsubsection{Error estimates and results}

Since the exact solution $u$ is not available, we verify that $\widetilde{u}$
is an approximate solution of (\ref{eq:variable coefficent eq}) by
evaluating the Fourier transform of the error on a particular set
of vectors. Note that

\[
h_{error}\left(\mathbf{x}\right)=-\nabla\cdot\left(a\left(\mathbf{x}\right)\nabla\tilde{u}\left(\mathbf{x}\right)\right)+k^{2}\tilde{u}\left(\boldsymbol{\mathbf{x}}\right)-f\left(\mathbf{x}\right).
\]
is a combination of Gaussians and products of Gaussians with low degree
polynomials and, therefore, we can explicitly compute its Fourier
transform, 
\begin{equation}
\widehat{h}_{error}\left(\boldsymbol{\mathbf{\xi}}\right)=\frac{1}{\left(2\pi\right)^{\frac{d}{2}}}\int_{\mathbb{R}^{d}}\left(-\nabla\cdot\left(a\left(\mathbf{x}\right)\nabla\tilde{u}\left(\mathbf{x}\right)\right)+k^{2}\tilde{u}\left(\mathbf{x}\right)-f\left(\mathbf{x}\right)\right)e^{-i\pi\mathbf{x}\cdot\boldsymbol{\mathbf{\xi}}}d\mathbf{x}.\label{eq:h-hat-error}
\end{equation}
We then evaluate $\widehat{h}_{error}$ for selected vector arguments
$\xi$, which we call frequency vectors. To select these vectors,
we use the principal directions of the matrices $\boldsymbol{\Sigma}_{l}$
of the Gaussian atoms in the representation of $\widetilde{u}$. To
this end, we solve the eigenvalue problem
\[
\boldsymbol{\Sigma}_{l}=\sum_{j=1}^{d}\lambda_{j}^{\left(l\right)}\mathbf{v}_{j}^{\left(l\right)}\left(\mathbf{v}_{j}^{\left(l\right)}\right)^{T}
\]
and select the frequency vectors along the principal directions of
$\Sigma_{l}$. In our experiment, we choose $s_{min}=10^{-5}$ and
$s_{max}=\left(-2\log\left(10^{-10}\right)/\lambda_{min}\right)^{\frac{1}{2}}$,
where $\lambda_{min}=\min_{l=1,\dots,\widetilde{N},j=1,\dots,d}\lambda_{j}^{\left(l\right)}$,
such that 
\[
e^{-\frac{\lambda_{min}s^{2}}{2}}\leq10^{-10}
\]
for $s>s_{max}$. We then sample $s_{k}$, $k=1,\dots,N_{s}$ using
a logarithmic scale on the interval $\left[s_{min},s_{max}\right]$
\[
s_{k}=e^{\log\left(s_{min}+\left(k-1\right)\frac{s_{max}-s_{min}}{N_{s}-1}\right)},
\]
and select the $d\cdot N_{s}\cdot\widetilde{N}$ frequency vectors
$\boldsymbol{\boldsymbol{\xi}}_{jk}^{\left(l\right)}$ to be
\[
\boldsymbol{\boldsymbol{\xi}}_{jk}^{\left(l\right)}=s_{k}\mathbf{v}_{j}^{\left(l\right)}\ \ \ \mbox{for}\ j=1,\dots,d,\ k=1,\dots,N_{s},\ l=1,\dots,\widetilde{N}.
\]
In our experiment, we choose $N_{s}=10$.

We notice that the number of terms in the solution $\widetilde{u}$
grows significantly with the dimension, if the matrices $\boldsymbol{\Sigma}_{a}$
and $\boldsymbol{\Sigma}_{f}$ are not related (they are effectively
random) and/or the range of their eigenvalues is large. In our first
experiment, we select matrices $\boldsymbol{\Sigma}_{a}$ and $\boldsymbol{\Sigma}_{f}$
in (\ref{eq:variable_coeff_a})-(\ref{eq:forcing_function_f}) to
be
\[
\Sigma_{a}=UD_{a}U^{T},\ \ \text{and}\ \ \Sigma_{f}=UD_{f}U^{T},
\]
where $U$ is a $d\times d$ random unitary matrix and $D_{a}$ and
$D_{f}$ are $d\times d$ diagonal matrices. We set the first two
diagonal entries of $D_{a}$ and $D_{f}$ to be $0.1$ and $20$,
and sample the other diagonal entry/entries from a uniform distribution
$\mathcal{U}\left(0.1,20\right)$. A random permutation is applied
after all diagonal entries are generated. In the second experiment,
we construct matrices $\boldsymbol{\Sigma}_{a}$ and $\boldsymbol{\Sigma}_{f}$
as
\[
\Sigma_{a}=U_{a}D_{a}U_{a}^{T},\ \ \ \Sigma_{f}=U_{f}D_{f}U_{f}^{T}
\]
where $U_{a}$ and $U_{f}$ are $d\times d$ random unitary matrices,
$D_{a}$ and $D_{f}$ are $d\times d$ diagonal matrices. We set the
first two diagonal entries of $D_{a}$ and $D_{f}$ to be $0.1$ and
$1$, and sample the other diagonal entry/entries from a uniform distribution
$\mathcal{U}\left(0.1,1\right)$. Again we randomly permute the diagonals
of $D_{a}$ and $D_{f}$. We also notice that if the centers $\mathbf{\mu}_{a}$
and $\mathbf{\mu}_{f}$ are far away (no overlapping essential supports),
then solving (\ref{eq:variable coefficent eq}) is effectively the
same as solving the Poisson's equation. In our tests, we select $\mathbf{\mu}_{f}=\boldsymbol{0}$
and $\mathbf{\mu}_{a}=\left(1,0,\dots,0\right)^{T}$ so that $\left\Vert \mathbf{\mu}_{a}-\mathbf{\mu}_{f}\right\Vert _{2}=1$. 

In both experiments, we iterate (\ref{eq:PDE iteration}) once to
generate a set of Gaussian atoms. The results are displayed in Tables~\ref{tab:variable coefficient aligned}
and \ref{table variable coefficient not aligned} where we show the
dimension of the problem, $d$, the number of terms $N_{terms}$ in
the approximation of the Green's function, the number of terms $N_{tot}$
obtained by performing one iteration in (\ref{eq:PDE iteration}),
the number of terms $\widetilde{N}$ in the solution $\widetilde{u}$
after reduction and the resulting accuracy. For these two tables,
we use the notation $\left\Vert \widehat{g}\right\Vert _{\infty}=\max_{j,k,l}\left|\widehat{g}\left(\boldsymbol{\boldsymbol{\xi}}_{jk}^{\left(l\right)}\right)\right|$. 

\begin{table}
~
\begin{centering}
\begin{tabular}{|c|c|c|c|c|}
\hline 
$d$ & $N_{terms}$ & $N_{tot}$ & $\widetilde{N}$ & $\left\Vert \widehat{h}_{error}\right\Vert _{\infty}/\left\Vert \widehat{f}\right\Vert _{\infty}$\tabularnewline
\hline 
\hline 
$3$ & $104$ & $5985$ & $1184$ & $1.8e-6$\tabularnewline
\hline 
$4$ & $109$ & $11624$ & $2649$ & $3.5e-6$\tabularnewline
\hline 
$5$ & $117$ & $16014$ & $3640$ & $3.8e-6$\tabularnewline
\hline 
$6$ & $121$ & $22466$ & $4377$ & $8.0e-5$\tabularnewline
\hline 
$7$ & $127$ & $31522$ & $5573$ & $2.5e-5$\tabularnewline
\hline 
\end{tabular}
\par\end{centering}
~\\
~
\centering{}\caption{\label{tab:variable coefficient aligned}Results of solving a second
order elliptic equation with a variable coefficient where the principle
directions of the matrices $\Sigma_{a}$ and $\Sigma_{f}$ in (\ref{eq:variable_coeff_a})
and (\ref{eq:forcing_function_f}) are aligned and their eigenvalues
are in the range $\left(0.1,20\right)$.}
\end{table}
\begin{table}
\begin{centering}
\begin{tabular}{|c|c|c|c|c|}
\hline 
$d$ & $N_{terms}$ & $N_{tot}$ & $\widetilde{N}$ & $\left\Vert \widehat{h}_{error}\right\Vert _{\infty}/\left\Vert \widehat{f}\right\Vert _{\infty}$\tabularnewline
\hline 
\hline 
$3$ & $104$ & $15348$ & $2586$ & $8.1e-7$\tabularnewline
\hline 
$4$ & $109$ & $32688$ & $6883$ & $1.9e-6$\tabularnewline
\hline 
$5$ & $117$ & $40110$ & $11793$ & $2.4e-6$\tabularnewline
\hline 
$6$ & $121$ & $64664$ & $18205$ & $6.7e-5$\tabularnewline
\hline 
$7$ & $127$ & $75059$ & $22966$ & \tabularnewline
\hline 
\end{tabular}
\par\end{centering}
~\\
~
\centering{}\caption{\label{table variable coefficient not aligned}Results of solving
a second order elliptic equation with a variable coefficient where
the principle directions of the matrices $\Sigma_{a}$ and $\Sigma_{f}$
in (\ref{eq:variable_coeff_a}) and (\ref{eq:forcing_function_f})
are not aligned and their eigenvalues are in the range $\left(0.1,1\right)$.
We note that the accuracy estimation in dimension $d=7$ is computationally
expensive and we skipped it. }
\end{table}

\section{\label{sec:Kernel-Density-Estimation}Kernel Density Estimation}

We describe a new algorithmic approach to Kernel Density Estimation
(KDE) based on Algorithm~\ref{alg:Reduction-algorithm-using-Gram}.
KDE is a non-parametric method for constructing the PDF of data points
used in cluster analysis, classification, and machine learning. The
standard KDE construction is practical only in low dimensions, $d=1,2,3$
as it requires a Fourier transform of the data points, the cost of
which grows exponentially with dimension (see e.g.\ \cite{O-K-C-C-O:2016,BER-PIG:2011}
for a technique based on the Fourier transform). Our approach avoids
using the Fourier transform and is applicable in high dimensions.
We note that a randomized approach that can be used for KDE estimation
was recently suggested in \cite{MAR-BIR:2017}. In this paper we do
not provide a comparison with other techniques that are applicable
to KDE (e.g.\ reproducing kernel techniques which formulate the problem
as a minimization of an objective function, see for example \cite{CAW-TAL:2002}
and references therein). We plan to develop our approach further and
provide an appropriate comparison with other methods elsewhere. 

The essence of KDE (see e.g.\ \cite{SILVER:1986}) is to associate
a smooth PDF $f$ with data points $\mathbf{x}_{j}\in\mathbb{R}^{d}$,
$j=1,\dots,N$, 
\begin{equation}
f\left(\mathbf{x},h\right)=\frac{1}{N}\sum_{j=1}^{N}K_{h}\left(\mathbf{x}-\mathbf{x}_{j}\right),\label{eq:kernel sum}
\end{equation}
where $K_{h}\left(\mathbf{x}\right)=K\left(\mathbf{x}/h\right)/h$
and $K$ is a nonnegative function with zero mean and $\int_{\mathbb{R}^{d}}K\left(\mathbf{x}\right)d\mathbf{x}=1$.
In what follows, we use a multivariate Gaussian as the kernel $K$.
A naive implementation of (\ref{eq:kernel sum}) would require $N$
evaluations of the kernel $K_{h}$ for each point $\mathbf{x}$ so
that the computational cost of using this approach in a straightforward
manner is prohibitive if $N$ is large. The selection of the parameter
$h$, the so-called bandwidth or scale parameter, is a well recognized
delicate issue and, in our one dimensional example, we use $h$ computed
within Mathematica$\,^{TM}$ implementation of KDE.

In our approach, for a user selected target accuracy $\epsilon$,
we seek a \textit{subset} of linear independent terms in (\ref{eq:kernel sum})
and express the remaining terms as their linear combinations. Thus,
by removing redundant terms in the representation of $f\left(\mathbf{x},h\right)$,
we construct
\begin{equation}
F\left(\mathbf{x},h\right)=\sum_{\ell=1}^{r}a_{\ell}K_{h}\left(\mathbf{x}-\mathbf{x}_{j_{\ell}}\right),\label{eq:kernel-approx}
\end{equation}
where $r\ll N$ and
\[
\left|f\left(\mathbf{x},h\right)-F\left(\mathbf{x},h\right)\right|\le\epsilon.
\]
In other words, with accuracy $\epsilon$, we obtain an approximation
of the function $f$ by a function $F$ with a small number of terms.
This reduction algorithm seeking a subset of linear independent terms
in (\ref{eq:kernel sum}) can be used for any kernel in $\mathbb{R}^{d}$
such that the cost of evaluating the multidimensional inner product
\[
g_{ij}=\int_{\mathbb{R}^{d}}K_{h}\left(\mathbf{x}-\mathbf{x}_{i}\right)K_{h}\left(\mathbf{x}-\mathbf{x}_{j}\right)d\mathbf{x}
\]
is reasonable i.e.\ depends mildly on the dimension $d$). For multivariate
Gaussians, the values $g_{ij}$ are available via an explicit expression,
see Section~\ref{subsec:Inner-product-of}. 

The computational cost of our algorithm is $\mathcal{O}\left(r^{2}N+p\left(d\right)r\thinspace N\right)$.
Here $N$ is the original number of data points and $r$ is the final
number of terms in the chosen linearly independent subset and $p\left(d\right)$
is the cost of computing the inner product between two terms of the
mixture. In typical KDE applications $p\left(d\right)\sim d$ since
usual kernels admit a separated representation. 

\subsection{A comparison in dimension $d=1$}

In low dimensions, the standard approach to KDE relies on using the
Fast Fourier transform to both, assist in estimating the bandwidth
parameter $h$ and in constructing a more efficient representation
of (\ref{eq:kernel sum}) on an equally spaced grid (see e.g.\ \cite[Section 3.5]{SILVER:1986}).
Implementations of this approach can be found in many packages in
dimensions $d=1,2$, e.g.\ Matlab, Mathematica, etc. While this approach
is appropriate in low dimensions, an extension of this algorithm to
high dimensions is prevented by the ``curse of dimensionality''.
Thus, in high dimensions, only values at selected points can be computed
(see \cite{BO-GR-KR:2010} and Matlab implementation of KDE in high
dimensions).

In order to illustrate our approach we provide a simple example with
a bimodal distribution in dimension $d=1$. Although this example
is in one variable, it allows us to emphasize the differences between
the existing KDE methods and our approach. We generate test data by
using two normal distributions with means $\mu_{1}=0$ and $\mu_{2}=4$.
The exact PDF of this data is given by 
\begin{equation}
g\left(x\right)=\frac{1}{2}\left(\frac{1}{\sqrt{2\pi}}e^{-\frac{1}{2}\frac{x^{2}}{\sigma^{2}}}+\frac{1}{\sqrt{2\pi}}e^{-\frac{1}{2}\frac{\left(x-4\right)^{2}}{\sigma^{2}}}\right),\label{eq:true distribution}
\end{equation}
where $\sigma=1$. We then use KDE implemented in Mathematica$\,^{TM}$
with the Gaussian kernel $K$. Using $N=10^{5}$ data samples drawn
from (\ref{eq:true distribution}) (so that the initial sum (\ref{eq:kernel sum})
has $N$ terms), the scaling parameter was set to $h=0.20121412622314902019$.
The true distribution (\ref{eq:true distribution}) and the approximation
error obtained by Mathematica$\,^{TM}$ by reducing (\ref{eq:kernel sum})
from $N=10^{5}$ terms to 241 Gaussian terms centered on an equally
spaced grid are displayed in Figure~\ref{fig:The-distribution-}.
Using our algorithm with the same parameter $h$, we reduce (\ref{eq:kernel sum})
from $N=10^{5}$ terms to 102 terms centered at a selected \textit{subset
of the} \textit{original} \textit{data points}. The error of the resulting
approximation is displayed in Figure~\ref{fig:The-distribution-},
where we also show the difference between (\ref{eq:kernel sum}) and
(\ref{eq:kernel-approx}). The main point here is that while the standard
approach in high dimensions becomes impractical (as it requires a
multidimensional grid), our approach proceeds unchanged since the
cost of the reduction algorithm depend on dimension only mildly. 
\begin{rem}
We selected a much higher accuracy for reduction than the difference
between the original and estimated PDFs in order to illustrate the
fact that the accuracy limit of Algorithm~\ref{alg:Reduction-algorithm-using-Gram}
of about $7$ to $8$ digits is more than sufficient for this application.
Note that to achieve a comparable accuracy for estimation of a PDF
via KDE one needs $\approx10^{16}$ points since the accuracy improves
as $\mathcal{O}\left(1/\sqrt{N}\right)$. 
\end{rem}

~
\begin{rem}
Using KDE in high dimensions requires an additional assumption that
points are located in a vicinity of a low dimensional manifold (see
comments in e.g.\ \cite[Section 1.5.3]{SCOTT:2015} and/or Figure~5
in \cite{MAR-BIR:2017}).
\end{rem}

\begin{figure}
\begin{centering}
\includegraphics[scale=0.3]{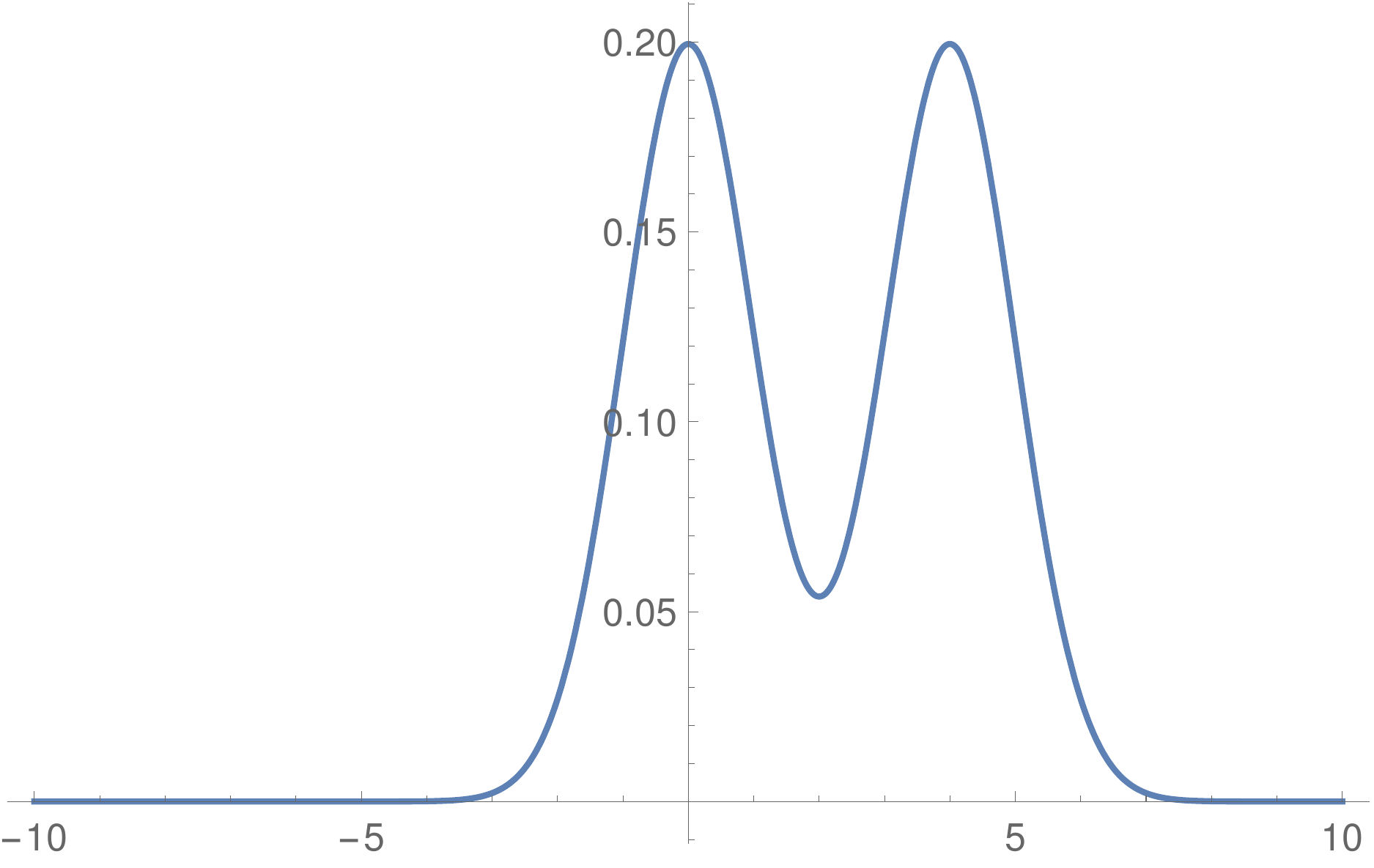}\includegraphics[scale=0.3]{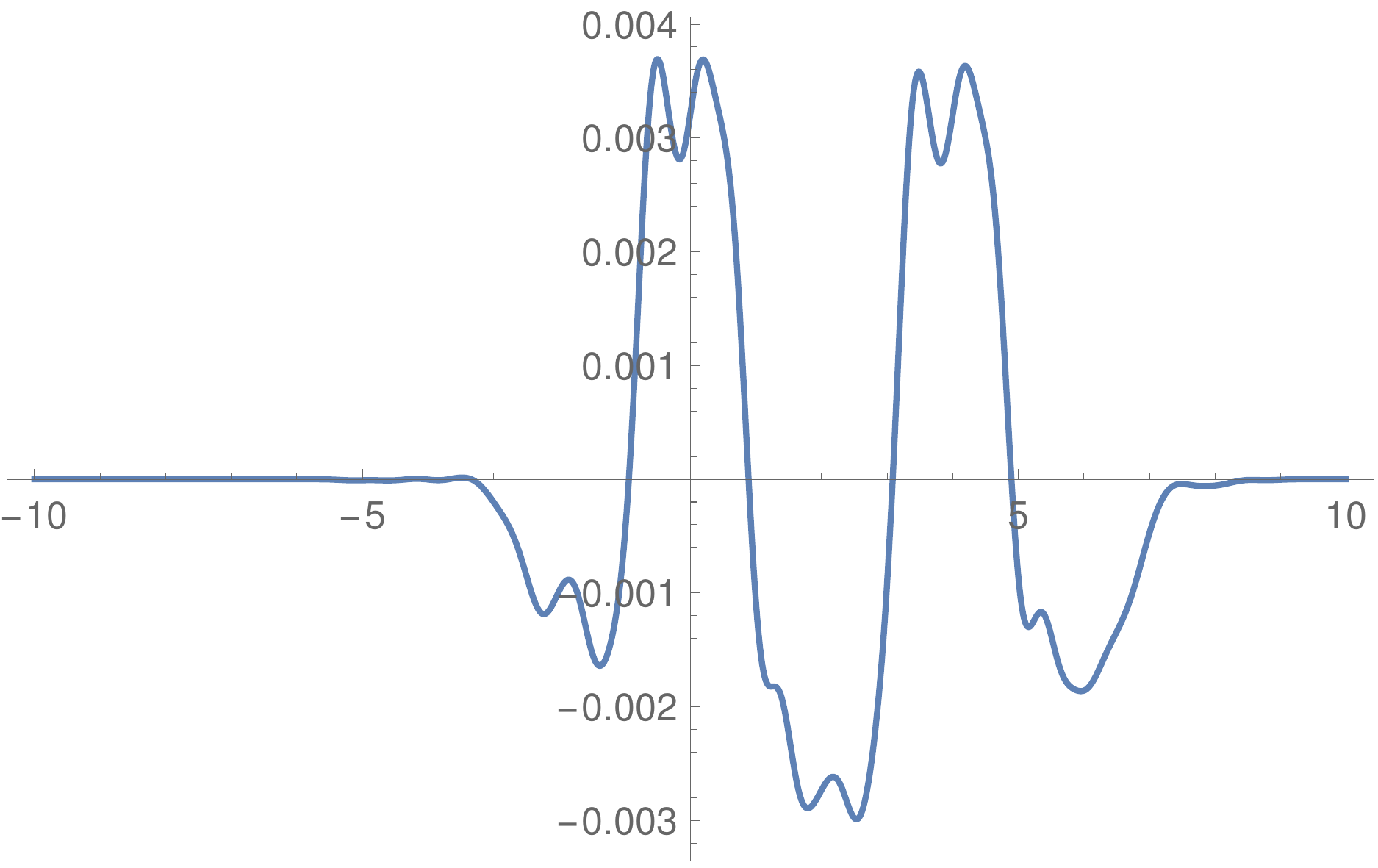}
\par\end{centering}
\begin{centering}
\includegraphics[scale=0.3]{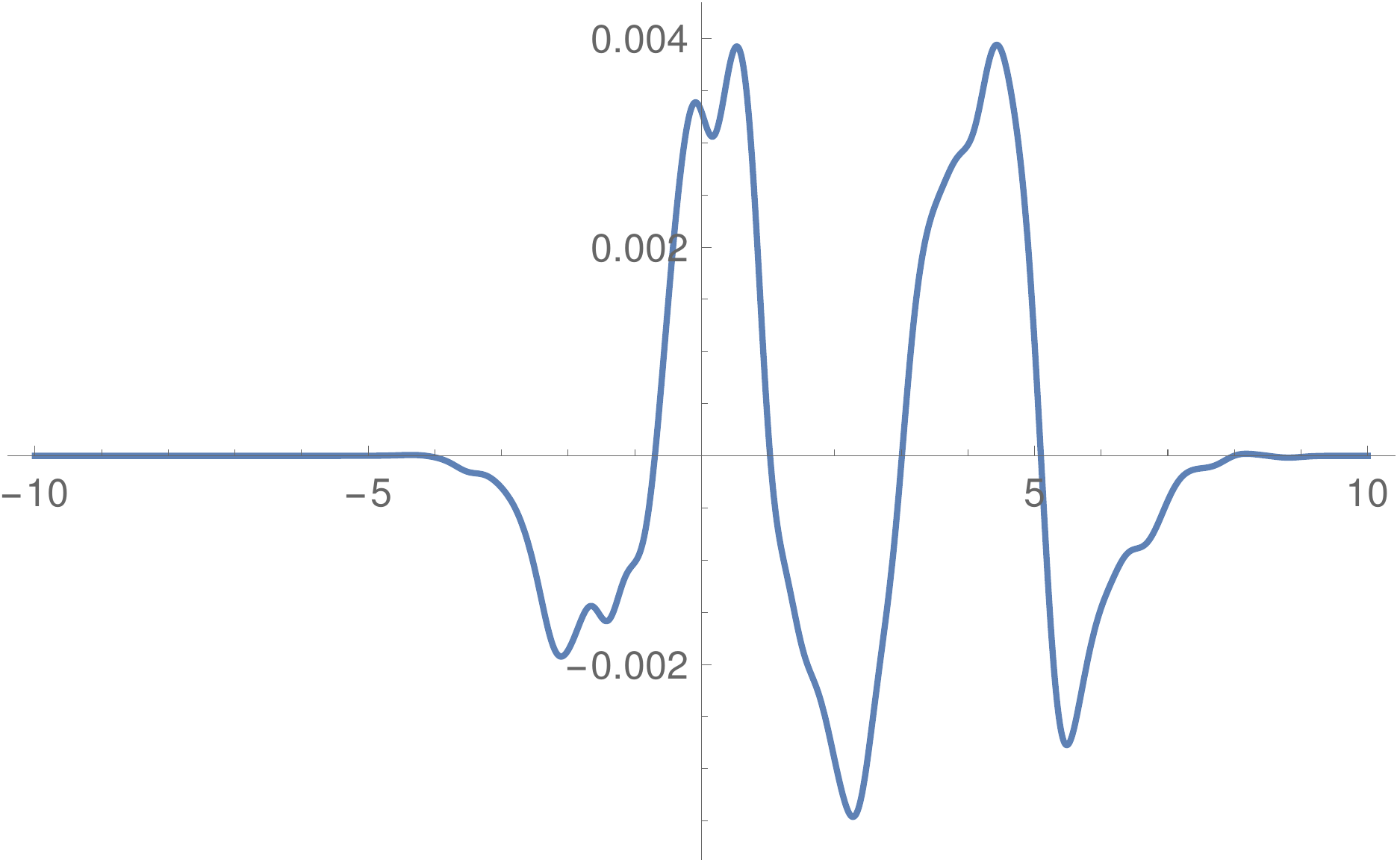}\includegraphics[scale=0.3]{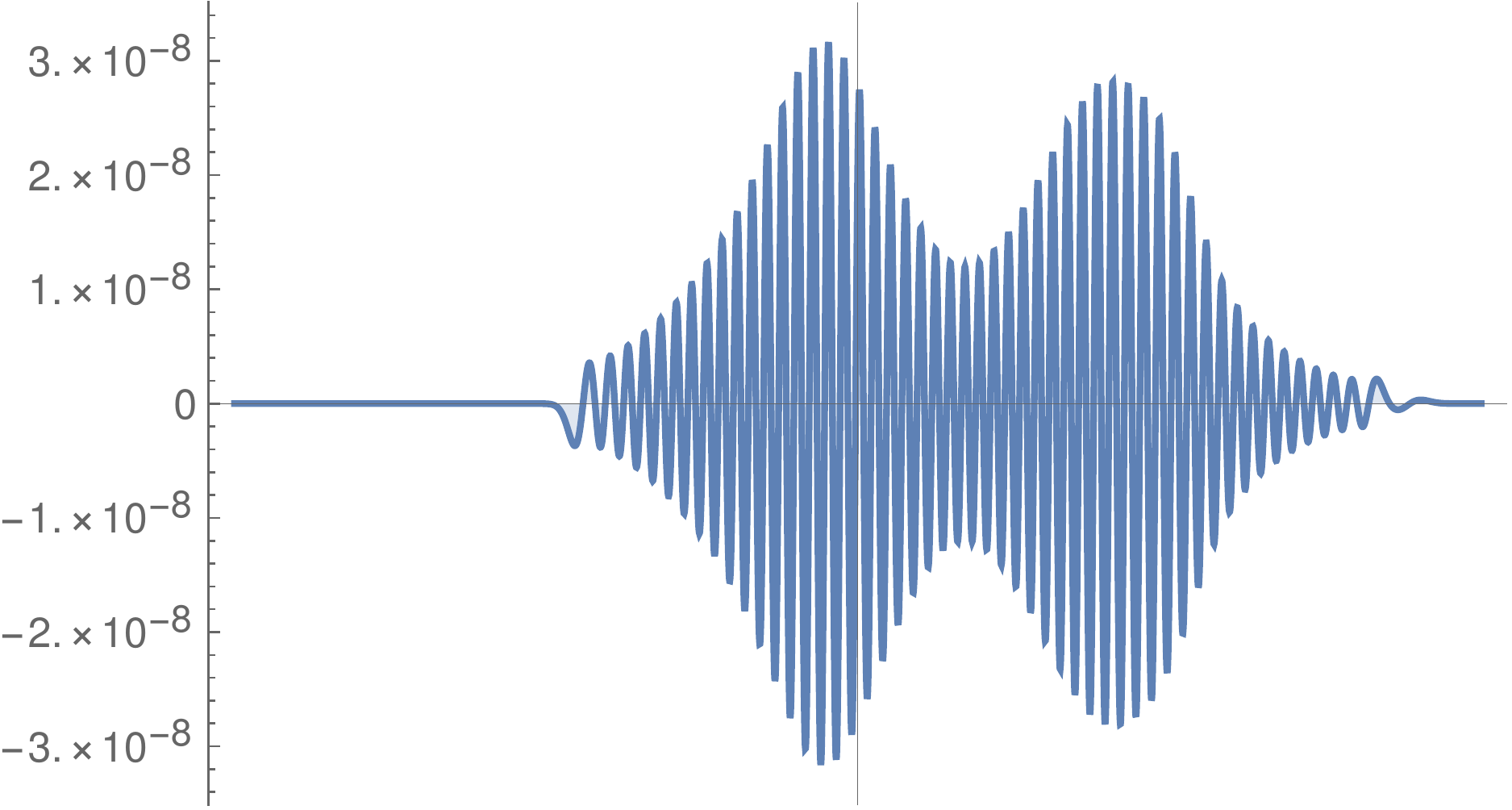}
\par\end{centering}
\caption{\label{fig:The-distribution-}The distribution (\ref{eq:true distribution})
(top left) and the error of its estimation with $241$ term centered
at an equally spaced grid obtain using KDE in Mathematica$\,^{TM}$
(top right). The error of estimating (\ref{eq:true distribution})
using Algorithm~\ref{alg:Reduction-algorithm-using-Gram} with $102$
terms centered at a selected subset of the \textit{original} data
points (bottom left). We also show the difference between the definition
of the PDF in (\ref{eq:kernel sum}) and its reduced version in (\ref{eq:kernel-approx})
obtained using Algorithm~\ref{alg:Reduction-algorithm-using-Gram}
where accuracy was set to $10^{-7}$ (bottom right).}
\end{figure}

\subsection{\label{subsec:An-example-in-high-dimensions}An example in high dimensions}

We generate $N=10^{5}$ samples from a two dimensional Gaussian distribution
with the PDF
\[
g\left(\mathbf{y}\right)=\frac{1}{2}\left(\frac{1}{2\pi}e^{-\frac{1}{2}\left(\mathbf{y}-\boldsymbol{\mu}_{1}\right)^{T}\boldsymbol{\Sigma}_{1}^{-1}\left(\mathbf{y}-\boldsymbol{\mu}_{1}\right)}+\frac{1}{2\pi}e^{-\frac{1}{2}\left(\mathbf{y}-\boldsymbol{\mu}_{2}\right)^{T}\boldsymbol{\Sigma}_{2}^{-1}\left(\mathbf{y}-\boldsymbol{\mu}_{2}\right)}\right)
\]
where $\mathbf{y}=\left(y_{1},y_{2}\right)$, $\boldsymbol{\mu}_{1}=\left(0,0\right)$,
$\boldsymbol{\mu}_{2}=\left(3,3\right)$ and
\[
\boldsymbol{\Sigma}_{1}=\begin{pmatrix}2 & 0\\
0 & 0.5
\end{pmatrix},\ \ \ \boldsymbol{\Sigma}_{2}=\begin{pmatrix}1 & 0\\
0 & 1
\end{pmatrix}.
\]
We pad these samples with zeros so that they belong to a $d$-dimensional
space and denote them by $\left\{ \mathbf{y}_{i}\right\} _{i=1}^{N}$.
We then apply a random rotation matrix $U$ to obtain the test data
$\left\{ \mathbf{x}_{i}\right\} _{i=1}^{N}$. As the bandwidth parameter,
we set
\[
h=\left(\frac{4}{2d+1}\right)^{\frac{1}{d+4}}N^{-\frac{1}{d+4}},
\]
a value that minimizes the mean integrated square error for an underlying
standard normal distribution (see e.g.\ \cite{SILVER:1986}). In
our case, since the intrinsic dimension of the test data is $2$,
we set $h=0.14142135623730950488$. The initial kernel density estimator
using all test data is
\begin{equation}
f\left(\mathbf{x}\right)=\frac{1}{Nh^{d}}\sum_{i=1}^{N}\frac{1}{\left(2\pi\right)^{\frac{d}{2}}}e^{-\frac{1}{2}\left(\mathbf{x}-\mathbf{x}_{i}\right)^{T}\left(h^{2}\boldsymbol{I}_{d}\right)^{-1}\left(\mathbf{x}-\mathbf{x}_{i}\right)},\label{eq:initial_KDE}
\end{equation}
where $\boldsymbol{I}_{d}$ is the $d$-by-$d$ identity matrix. We
then reduce the number of terms in $\eqref{eq:initial_KDE}$ using
Algorithm~\ref{alg:Reduction-algorithm-using-Gram} and obtain a
sum of Gaussians with fewer terms,
\[
\widetilde{f}\left(\mathbf{x}\right)=\sum_{j=1}^{\widetilde{N}}c_{j}\frac{1}{\left(2\pi\right)^{\frac{d}{2}}}e^{-\frac{1}{2}\left(\mathbf{x}-\mathbf{x}_{j}\right)^{T}\left(h^{2}\boldsymbol{I}_{d}\right)^{-1}\left(\mathbf{x}-\mathbf{x}_{j}\right)}.
\]
In our experiments, we choose dimensions $d=2,\dots16$ and error
threshold $\epsilon=10^{-1}$ in Algorithm~\ref{alg:Reduction-algorithm-using-Gram}.
The number of terms after reduction is about $2000$ for all dimensions
and the approximation error is about $1.2\times10^{-3}$ (see Figure~\ref{fig:kde_16d}).

In order to compare our approximation with the true PDF $g\left(\mathbf{x}\right)$,
we consider a point $\mathbf{y}=\left(y_{1},y_{2},0,\dots,0\right)\in\mathbb{R}^{d}$
such that, under the rotation $U$, we have $\mathbf{x}=U\mathbf{y}$.
Evaluating (\ref{eq:initial_KDE}) at $\mathbf{x}$, we obtain
\begin{eqnarray*}
f\left(\mathbf{x}\right) & = & f\left(U\mathbf{y}\right)\\
 & = & \frac{1}{Nh^{d}}\sum_{i=1}^{N}\frac{1}{\left(2\pi\right)^{\frac{d}{2}}}e^{-\frac{1}{2}\left(U\left(\mathbf{y}-\mathbf{y}_{i}\right)\right)^{T}\left(h^{2}\boldsymbol{I}_{d}\right)^{-1}\left(U\left(\mathbf{y}-\mathbf{y}_{i}\right)\right)}\\
 & = & \frac{1}{Nh^{d}}\sum_{i=1}^{N}\frac{1}{\left(2\pi\right)^{\frac{d}{2}}}e^{-\frac{1}{2}\left(\mathbf{y}-\mathbf{y}_{i}\right)^{T}\left(h^{2}\boldsymbol{I}_{d}\right)^{-1}\left(\mathbf{y}-\mathbf{y}_{i}\right)}\\
 & = & c_{d}\frac{1}{Nh^{2}}\sum_{i=1}^{N}\frac{1}{2\pi}e^{-\frac{\left(y_{1}-y_{1}^{\left(i\right)}\right)^{2}+\left(y_{2}-y_{2}^{\left(i\right)}\right)^{2}}{2h^{2}}},
\end{eqnarray*}
where $c_{d}=h^{2-d}\left(2\pi\right)^{\frac{2-d}{2}}$ and $y_{j}^{\left(i\right)}$
denote the $j$-th component of $\mathbf{y}_{i}$. Notice that the
kernel density estimator for the PDF of the original distribution
in two dimensions is 
\[
g\left(y_{1},y_{2}\right)\approx\frac{1}{Nh^{2}}\sum_{i=1}^{N}\frac{1}{2\pi}e^{-\frac{\left(y_{1}-y_{1}^{\left(i\right)}\right)^{2}+\left(y_{2}-y_{2}^{\left(i\right)}\right)^{2}}{2h^{2}}}=\frac{1}{c_{d}}f\left(\mathbf{x}\right)
\]
Therefore, to estimate the size of the approximation error, we generate
a set of $64\times64$ equispaced grid points $\left(\bar{y}_{1}^{\left(i\right)},\bar{y}_{2}^{\left(j\right)}\right)$,
$i,j=1,\dots,64$ in $\left[-8,8\right]\times\left[-8,8\right]$.
We pad each $\left(\bar{y}_{1}^{\left(i\right)},\bar{y}_{2}^{\left(j\right)}\right)$
with $d-2$ zeros to embed it into $\mathbb{R}^{d}$, which we denote
as $\bar{\mathbf{y}}_{ij}=\left(\bar{y}_{1}^{\left(i\right)},\bar{y}_{2}^{\left(j\right)},0,\dots,0\right)\in\mathbb{R}^{d}$.
We then apply the same rotation matrix $U$ to obtain a set of points
$\mathbf{\bar{x}}_{ij}=U\bar{\mathbf{y}}_{ij}\in\mathbb{R}^{d}$.
In Figure~\ref{fig:kde_16d} we illustrate the result for dimension
$d=16$; we show the PDF $g$, the errors between $g$ and the KDE
estimate $f$ before reduction and the estimate $\widetilde{f}$ after
reduction, as well as the difference between $f$ and $\widetilde{f}$
.

\begin{figure}[h]
\begin{centering}
\includegraphics[scale=0.3]{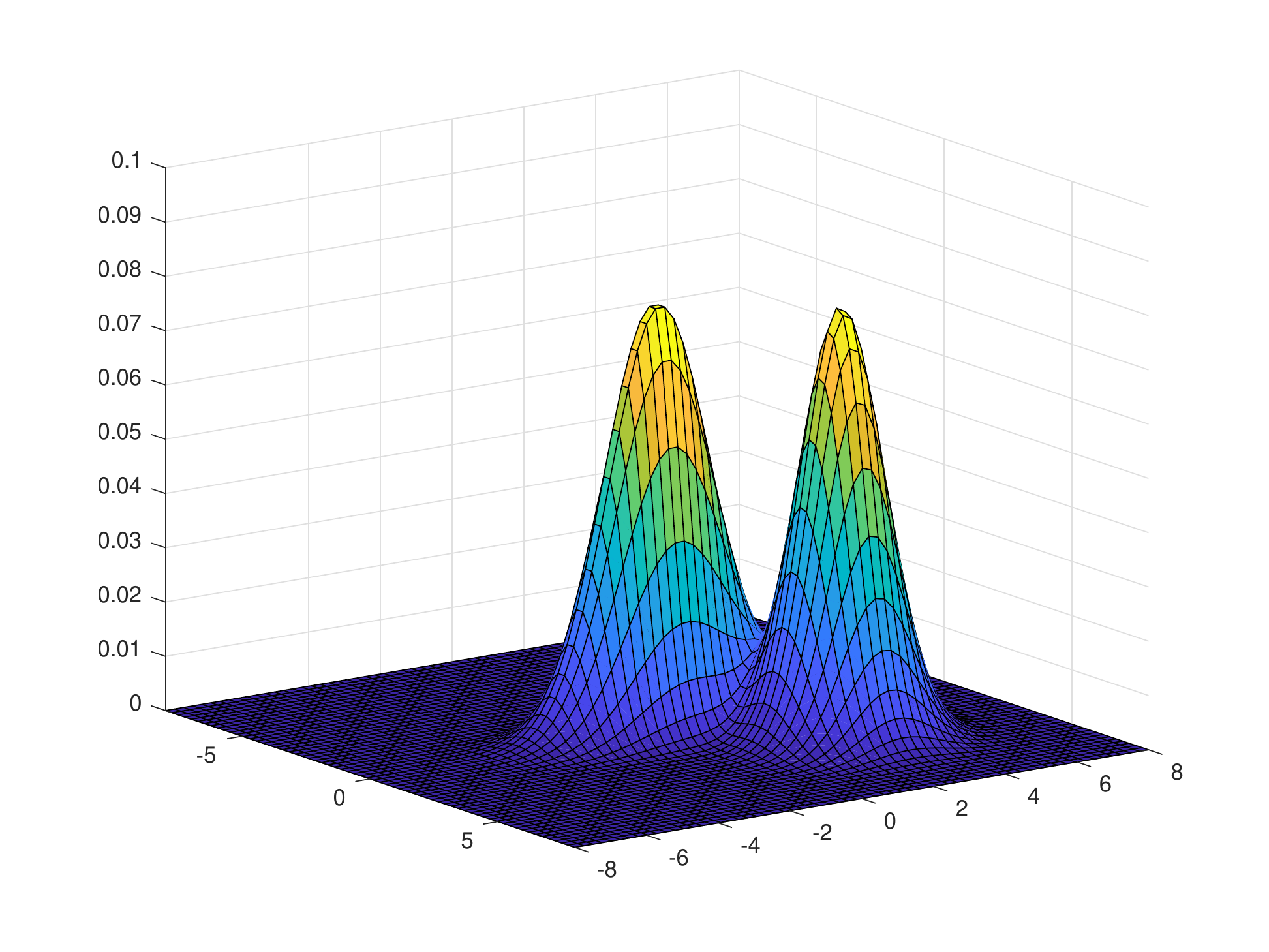}\includegraphics[scale=0.3]{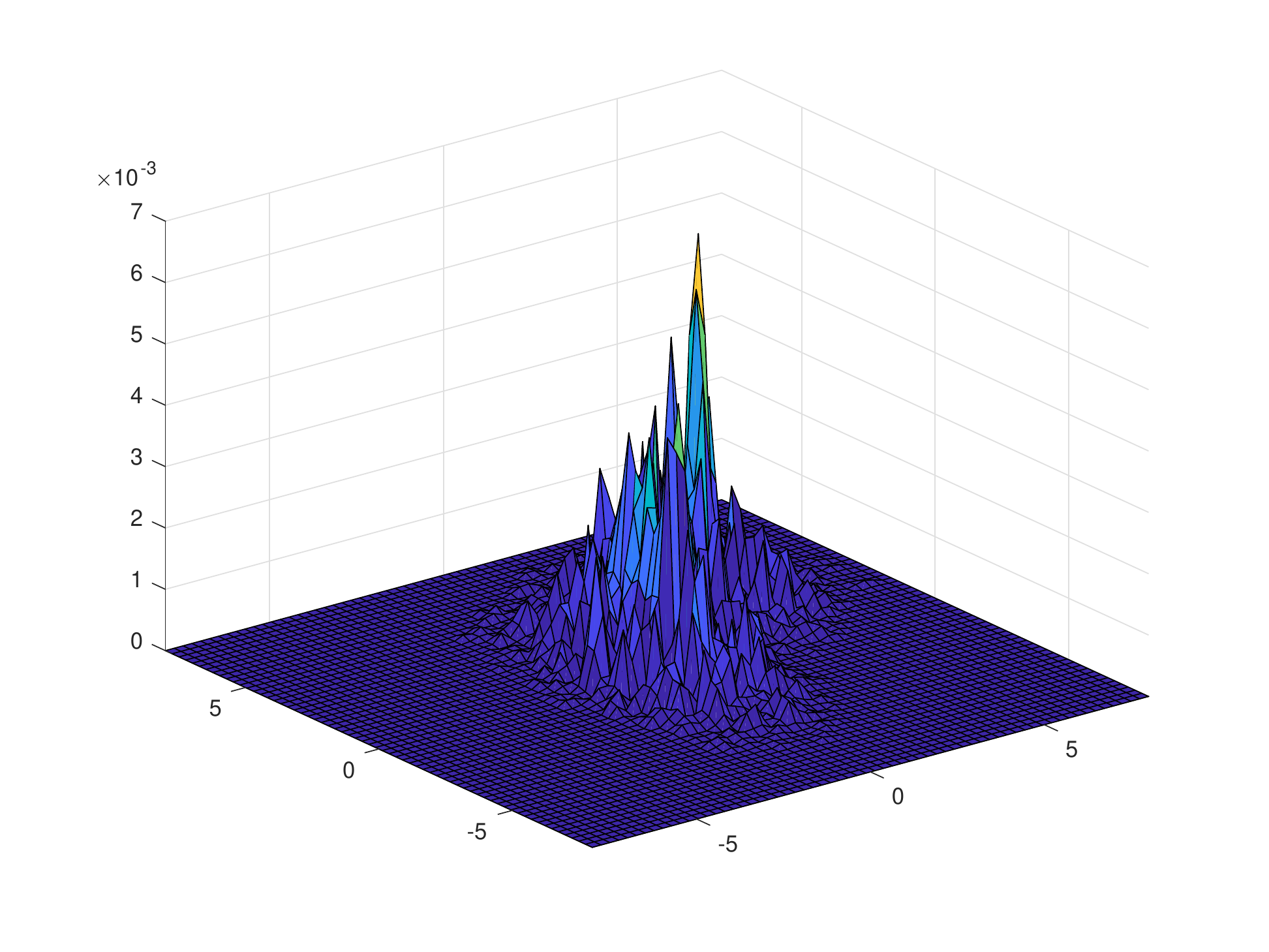}
\par\end{centering}
\begin{centering}
\includegraphics[scale=0.3]{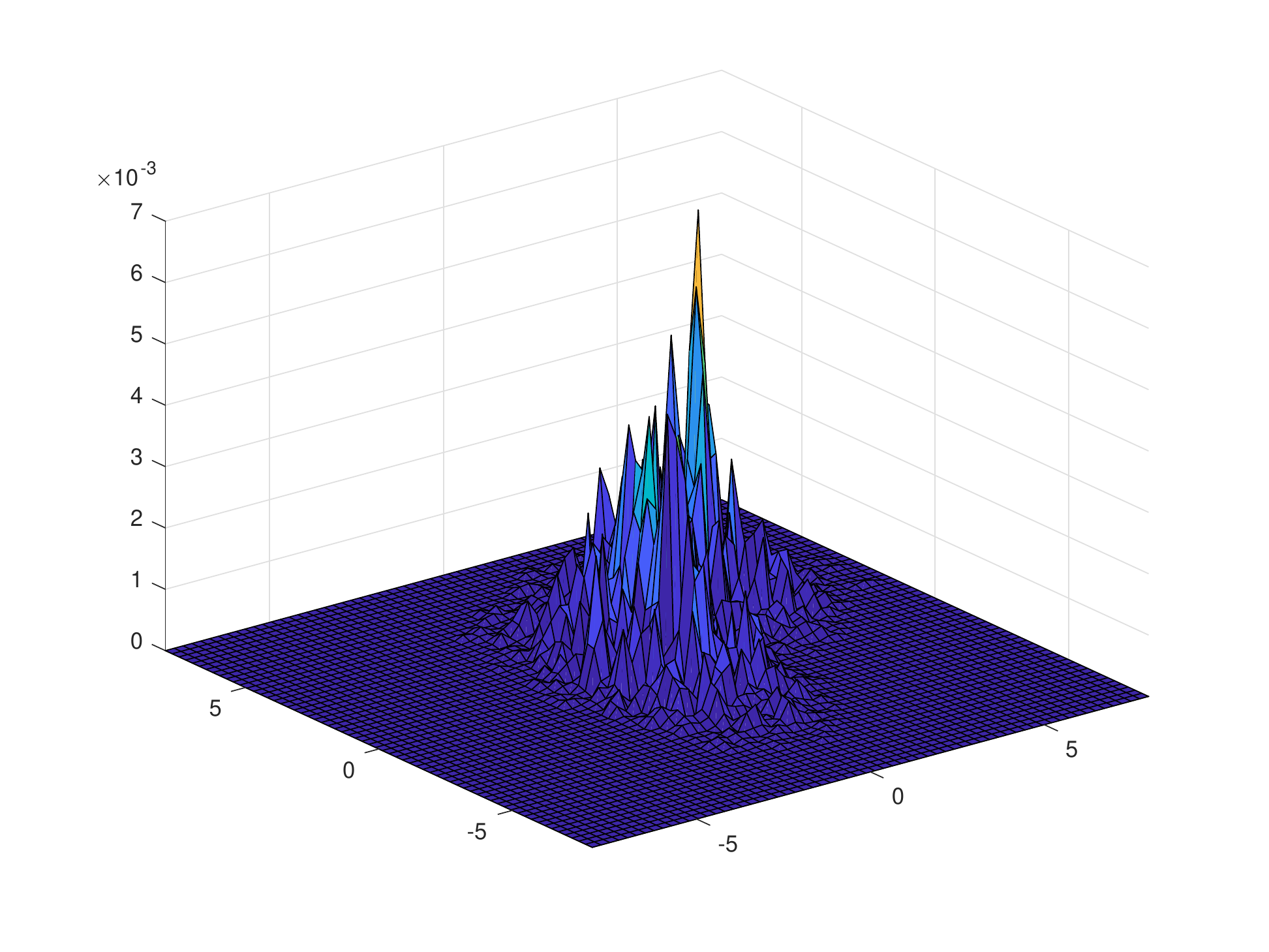}\includegraphics[scale=0.3]{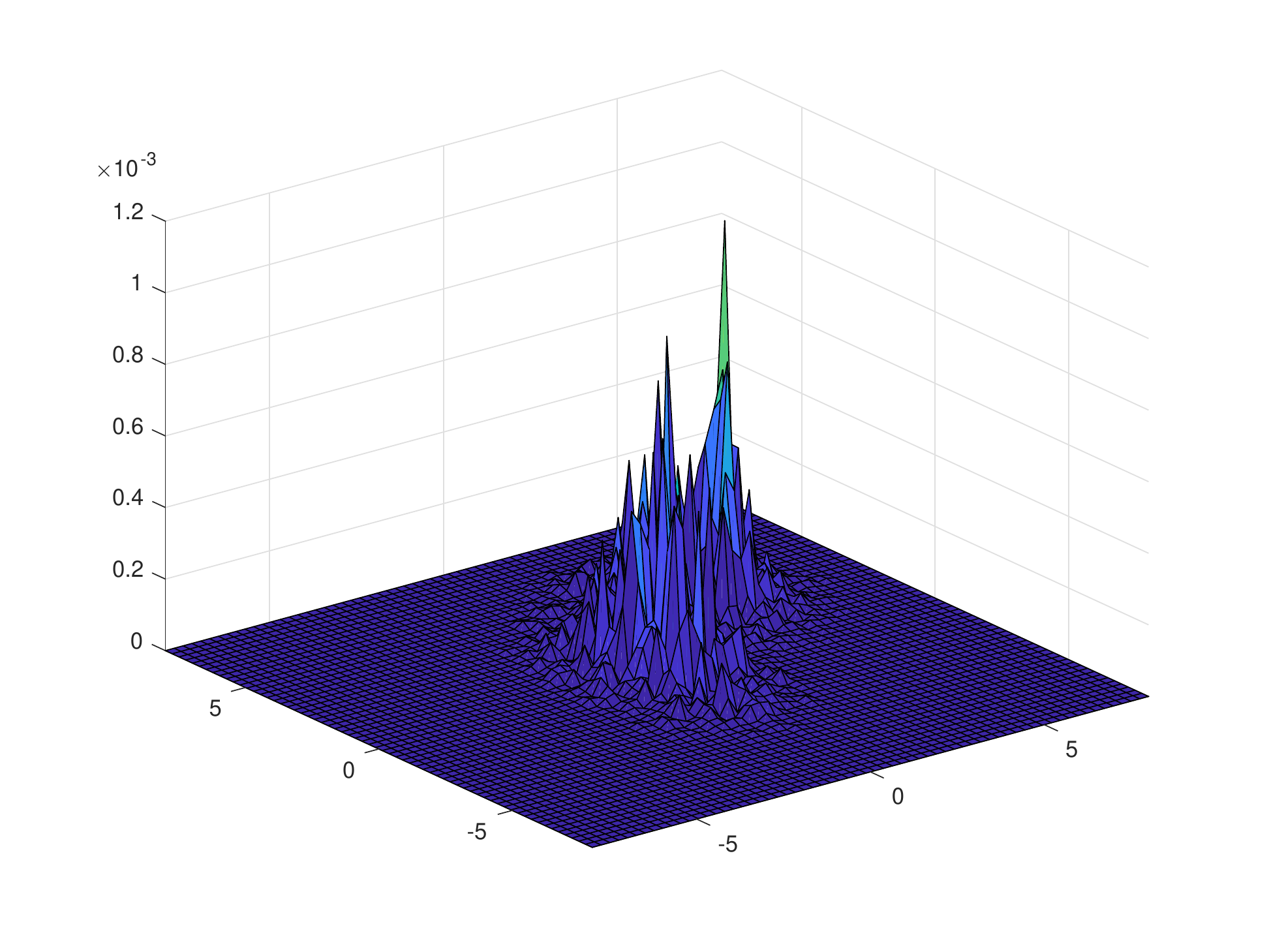}
\par\end{centering}
\caption{\label{fig:kde_16d}Using the grid $\left(\bar{y}_{1}^{\left(i\right)},\bar{y}_{2}^{\left(j\right)}\right)$,
$i,j=1,\dots64$, and the rescaling constant $c_{d}=h^{2-d}\left(2\pi\right)^{\frac{2-d}{2}}$,
$d=16$, we display the PDF $g\left(\bar{y}_{1}^{\left(i\right)},\bar{y}_{2}^{\left(j\right)}\right)$
(top left) and the difference $\left|g\left(y_{1}^{\left(i\right)},y_{2}^{\left(j\right)}\right)-\frac{1}{c_{d}}f\left(U\bar{\mathbf{y}}_{ij}\right)\right|$
between $g$ and constructed PDF $f$ (top right). We also show the
difference $\left|g\left(\bar{y}_{1}^{\left(i\right)},\bar{y}_{2}^{\left(j\right)}\right)-\frac{1}{c_{d}}\widetilde{f}\left(U\bar{\mathbf{y}}_{ij}\right)\right|$
between $g$ and $\widetilde{f}$ (bottom left) and the difference
$\frac{1}{c_{d}}\left|f\left(U\bar{\mathbf{y}}_{ij}\right)-\widetilde{f}\left(U\bar{\mathbf{y}}_{ij}\right)\right|$
between $f$ and $\widetilde{f}$ (bottom right) . }
\end{figure}

\section{\label{sec:Far-field-summation-in}Far-field summation in high dimensions}

Given a large set of points in dimension $d\gg3$ with pairwise interaction
via a non-oscillatory kernel, our approach provides a deterministic
algorithm for fast summation in the far-field setup (i.e.\ where
two groups of points are separated). Recently a randomized algebraic
approach in a similar setup was suggested in \cite{MAR-BIR:2017}.
Instead, we use Algorithm~\ref{alg:Reduction-algorithm-using-Gram}
as a tool to rapidly evaluate

\[
g_{m}=\sum_{n=1}^{N}f_{n}K\left(\mathbf{x}_{m},\mathbf{y}_{n}\right),\,\,\,m=1,\dots,M,
\]
where $M$ and $N$ are large and $K\left(\mathbf{x},\mathbf{y}\right)$
is a non-oscillatory kernel with a possible singularity at $\mathbf{x}=\mathbf{y}$
(recall that kernels of mathematical physics typically have a singularity
for coincident arguments \textbf{$\mathbf{x}$} and $\mathbf{y}$).

Let us consider sources $\left\{ \left(\mathbf{y}_{n},f_{n}\right)\right\} _{n=1}^{N}$
and targets $\left\{ \left(\mathbf{x}_{m},g_{m}\right)\right\} _{m=1}^{M}$
occupying two distinct $d$-dimensional balls. Specifically, we assume
that $\left\Vert \mathbf{y}_{n}-\mathbf{y}_{c}\right\Vert \leq r_{s}$
and $\left\Vert \mathbf{x}_{m}-\mathbf{x}_{c}\right\Vert \leq r_{t}$,
where $\boldsymbol{y}_{c}$, $\boldsymbol{x}_{c}$ and $r_{s}$, $r_{t}$
are the centers and the radii of the balls. We also assume that sources
and targets are separated, i.e., $0<r\le\left\Vert \mathbf{x}_{m}-\mathbf{y}_{n}\right\Vert \le R$.
The separation of sources and targets implies that in the evaluation
of the kernel we are never close to a possible singularity of the
kernel at $\mathbf{x}=\mathbf{y}$. Separation of sources and targets
allows us to use a non-singular approximation of the kernel and reduce
the problem to finding the best linearly independent subsets of sources
(or targets) as defined by such approximate kernel. We also assume
that the sources are located on a low-dimensional manifold embedded
in high-dimensional space (see Remark~\ref{rem:If-sources-are-chosen}
below).

In order to use Algorithm~\ref{alg:Reduction-algorithm-using-Gram},
we need to define an inner product that takes into account the assumption
of separation of sources and targets. We illustrate this using the
example of the Poisson kernel (\ref{eq:poissongreenfun}), $K\left(\mathbf{x},\mathbf{y}\right)=\left\Vert \mathbf{x}-\mathbf{y}\right\Vert ^{-d+2}$
(without standard normalization). Approximating $K\left(\mathbf{x},\mathbf{y}\right)$
as in Section~\ref{subsec:Poisson-equation-in} (c.f. (\ref{eq:ApproxViaGaussians})),
we have
\begin{equation}
\left|K\left(\mathbf{x},\mathbf{y}\right)-\sum_{l\in\mathbb{Z}}w_{l}e^{-\tau_{l}\left\Vert \mathbf{x}-\mathbf{y}\right\Vert ^{2}}\right|\le\epsilon K\left(\mathbf{x},\mathbf{y}\right).\label{eq:kernel approximation}
\end{equation}
Since sources and targets are separated,we drop terms in (\ref{eq:kernel approximation})
with sufficiently large exponents (they produce a negligible contribution
in the interval $\left[r,R\right]$) as well as replace terms with
small exponents using an algorithm described in \cite{BEY-MON:2010}.
As a result, we obtain
\[
\widetilde{K}\left(\mathbf{x},\mathbf{y}\right)=\sum_{l=L_{0}}^{L_{1}}w_{l}e^{-\tau_{l}\left\Vert \mathbf{x}-\mathbf{y}\right\Vert ^{2}},
\]
such that
\[
\left|K\left(\mathbf{x},\mathbf{y}\right)-\widetilde{K}\left(\mathbf{x},\mathbf{y}\right)\right|<\widetilde{\epsilon},\,\,\,\mbox{for}\,\,\,r\leq\left\Vert \mathbf{x}-\mathbf{y}\right\Vert \leq R,
\]
where $\widetilde{\epsilon}$ is slightly larger than $\epsilon$.
It follows that
\[
\left|g_{m}-\tilde{g}_{m}\right|\le\widetilde{\epsilon},
\]
where
\[
\widetilde{g}_{m}=\sum_{n=1}^{N}f_{n}\widetilde{K}\left(\mathbf{x}_{m},\mathbf{y}_{n}\right),\,\,\,m=1,\dots,M.
\]
Since \textbf{$\left\Vert \mathbf{x}-\mathbf{x}_{c}\right\Vert \leq r_{t}$
}implies\textbf{ $r\leq\left\Vert \mathbf{x}-\mathbf{y}_{n}\right\Vert \leq R$,}
we define the inner product as an integral over the ball $\left\Vert \mathbf{x}-\mathbf{x}_{c}\right\Vert \leq r_{t}$,
\begin{eqnarray}
\langle\widetilde{K}\left(\mathbf{\cdot},\mathbf{y}_{n}\right),\widetilde{K}\left(\mathbf{\cdot},\mathbf{y}_{n'}\right)\rangle_{d} & = & \int_{\left\Vert \boldsymbol{x}-\boldsymbol{x}_{c}\right\Vert \leq r_{t}}\widetilde{K}\left(\mathbf{x},\mathbf{y}_{n}\right)\widetilde{K}\left(\mathbf{x},\mathbf{y}_{n'}\right)d\mathbf{x}.\label{eq:inner product on disk}
\end{eqnarray}
The inner product (\ref{eq:inner product on disk}) can be reduced
to a one dimensional integral, as we show next.
\begin{eqnarray*}
 &  & \langle\widetilde{K}\left(\mathbf{\cdot},\mathbf{y}_{n}\right),\widetilde{K}\left(\mathbf{\cdot},\mathbf{y}_{n'}\right)\rangle_{d}\\
 & = & \sum_{l,l'=L_{0}}^{L_{1}}w_{l}w_{l'}\int_{\left\Vert \mathbf{x}-\mathbf{x}_{c}\right\Vert \leq r_{t}}e^{-\tau_{l}\left\Vert \mathbf{x}-\mathbf{y}_{n}\right\Vert ^{2}}e^{-\tau_{l'}\left\Vert \mathbf{x}-\mathbf{y}_{n'}\right\Vert ^{2}}d\mathbf{x}\\
 & = & \sum_{l,l'=L_{0}}^{L_{1}}w_{l}w_{l'}e^{-\frac{\tau_{l}\tau_{l'}}{\tau_{l}+\tau_{l'}}\left\Vert \mathbf{y}_{n}-\mathbf{y}_{n'}\right\Vert ^{2}}\int_{\left\Vert \mathbf{x}-\mathbf{x}_{c}\right\Vert \leq r_{t}}e^{-\left(\tau_{l}+\tau_{l'}\right)\left\Vert \mathbf{x}-\frac{\tau_{l}\mathbf{y}_{n}+\tau_{l'}\mathbf{y}_{n'}}{\tau_{l}+\tau_{l'}}\right\Vert ^{2}}d\mathbf{x}\\
 & = & \sum_{l,l'=L_{0}}^{L_{1}}\widetilde{w}_{ll'}^{nn'}I_{ll'}^{nn'},
\end{eqnarray*}
where 
\[
\widetilde{w}_{ll'}^{nn'}=w_{l}w_{l'}e^{-\frac{\tau_{l}\tau_{l'}}{\tau_{l}+\tau_{l'}}\left\Vert \mathbf{y}_{n}-\mathbf{y}_{n'}\right\Vert ^{2}},\ \ \ \widetilde{\tau}_{ll'}=\tau_{l}+\tau_{l'},\ \ \ \widetilde{\mathbf{y}}_{ll'}^{nn'}=\frac{\tau_{l}\mathbf{y}_{n}+\tau_{l'}\mathbf{y}_{n'}}{\tau_{l}+\tau_{l'}},
\]
and 
\[
I_{ll'}^{nn'}=\int_{\left\Vert \mathbf{x}-\mathbf{x}_{c}\right\Vert \leq r_{t}}e^{-\widetilde{\tau}_{ll'}\left\Vert \mathbf{x}-\widetilde{\mathbf{y}}_{ll'}^{nn'}\right\Vert ^{2}}d\mathbf{x}.
\]
Using spherical coordinates in dimensions $d\ge3$, we obtain
\begin{eqnarray}
I_{ll'}^{nn'} & = & \int_{\left\Vert \mathbf{x}-\mathbf{x}_{c}\right\Vert \leq r_{t}}e^{-\widetilde{\tau}_{ll'}\left\Vert \mathbf{x}-\widetilde{\mathbf{y}}_{ll'}^{nn'}\right\Vert ^{2}}d\mathbf{x}\nonumber \\
 & = & e^{-\widetilde{\tau}_{ll'}\left\Vert \mathbf{x}_{c}-\widetilde{\mathbf{y}}_{ll'}^{nn'}\right\Vert ^{2}}\int_{\left\Vert \mathbf{z}\right\Vert \leq r_{t}}e^{-\widetilde{\tau}_{ll'}\left\Vert \mathbf{z}\right\Vert ^{2}}e^{-2\widetilde{\tau}_{ll'}\left\langle \mathbf{z},\mathbf{x}_{c}-\widetilde{\mathbf{y}}_{ll'}^{nn'}\right\rangle }d\mathbf{z}\nonumber \\
 & = & e^{-\widetilde{\tau}_{ll'}\left\Vert \mathbf{x}_{c}-\widetilde{\mathbf{y}}_{ll'}^{nn'}\right\Vert ^{2}}\Omega_{d-2}\int_{0}^{r_{t}}\left(\int_{0}^{\pi}e^{-2\widetilde{\tau}_{ll'}r\left\Vert \mathbf{x}_{c}-\widetilde{\mathbf{y}}_{ll'}^{nn'}\right\Vert \cos\theta}\sin^{d-2}\theta d\theta\right)e^{-\widetilde{\tau}_{ll'}r^{2}}r^{d-1}dr\nonumber \\
 & = & e^{-\widetilde{\tau}_{ll'}\left\Vert \mathbf{x}_{c}-\widetilde{\mathbf{y}}_{ll'}^{nn'}\right\Vert ^{2}}\frac{2\pi^{\frac{d}{2}}}{\left(\widetilde{\tau}_{ll'}\left\Vert \mathbf{x}_{c}-\widetilde{\mathbf{y}}_{ll'}^{nn'}\right\Vert \right)^{\frac{d-2}{2}}}\int_{0}^{r_{t}}I_{\frac{d-2}{2}}\left(2\widetilde{\tau}_{ll'}\left\Vert \mathbf{x}_{c}-\widetilde{\mathbf{y}}_{ll'}^{nn'}\right\Vert r\right)e^{-\widetilde{\tau}_{ll'}r^{2}}r^{\frac{d}{2}}dr,\label{eq:integral to compute}
\end{eqnarray}
where $\Omega_{d}$ is the surface area of the $d$-dimensional sphere
embedded in $\left(d+1\right)$-dimensional space, i.e., $\Omega_{d}=\frac{2\pi^{\frac{d+1}{2}}}{\Gamma\left(\frac{d+1}{2}\right)}$
and $I_{d}$ is the $d$-th order modified Bessel function of the
first kind (see \cite[Eq. 9.6.18]{ABR-STE:1964}). While in odd dimensions
$d$ the integrand in (\ref{eq:integral to compute}) can be extended
from $\left[0,r_{t}\right]$ to $\left[-r_{t},r_{t}\right]$ as a
smooth function allowing us to use the trapezoidal rule, such extension
is not available in even dimensions. Because of this, we choose to
use quadratures on $\left[0,r_{t}\right]$ developed in \cite{BEY-MON:2002}
(alternatively, one can use quadratures from \cite{ALPERT:1999}). 
\begin{rem}
It is an important observation that the selection of the inner product
for finding a linearly independent subset of functions is not limited
to the standard one defined in (\ref{eq:inner product on disk}).
Observing that (\ref{eq:inner product on disk}) approaches zero as
$d$ increases, in all of our experiments in dimensions $d=3,\dots128$,
we use (\ref{eq:inner product on disk}) where we set $d=3$. Thus,
the inner product $\langle\cdot,\cdot\rangle_{3}$ no longer corresponds
to the integral between the functions $\widetilde{K}\left(\mathbf{x},\mathbf{y}_{n}\right)$
and $\widetilde{K}\left(\mathbf{x},\mathbf{y}_{n'}\right)$. However,
since we use inner products only to identify the best linearly independent
subset of sources (skeleton sources) and compute the coefficients
to replace the remaining terms as linear combinations of these skeleton
sources, there are many choices of inner products that will produce
similar results. 
\end{rem}

Associating with each $\mathbf{y}_{n}$ the source function $K\left(\mathbf{x},\mathbf{y}_{n}\right)$,
we use Algorithm~\ref{alg:Reduction-algorithm-using-Gram} to find
the skeleton terms (i.e.\ the skeleton sources) with indices $\widehat{I}=\left\{ n_{k}\right\} _{k=1}^{r_{s}}$
allowing us to express the remaining source functions (for $n\notin\widehat{I}$)
as
\[
\left|\widetilde{K}\left(\mathbf{x},\mathbf{y}_{n}\right)-\sum_{k=1}^{r_{s}}\widetilde{f}_{n_{k}}\widetilde{K}\left(\mathbf{x},\mathbf{y}_{n_{k}}\right)\right|\le\epsilon_{1},\,\,\,n\notin\widehat{I}.
\]

\begin{rem}
Using Algorithm~\ref{alg:Reduction-algorithm-using-Gram} to find
the skeleton sources requires $\mathcal{O}\left(r_{s}^{2}\,N+p\left(d\right)r_{s}N\right)$
operations and computing interactions between skeleton sources and
targets requires additional $\mathcal{O}\left(r_{s}\,M\right)$ operations.
Clearly, instead of working with sources, we can work with targets.
If targets are located on a low-dimensional manifold, we can associate
the functions $\left\{ K\left(\mathbf{x}_{m},\mathbf{y}\right)\right\} _{m=1}^{M}$
with targets and use Algorithm~\ref{alg:Reduction-algorithm-using-Gram}
to find the skeleton targets. In such case, the computational cost
becomes $\mathcal{O}\left(r_{t}^{2}M+p\left(d\right)r_{t}M+r_{t}N\right)$,
where $r_{t}$ is the number of skeleton targets. 
\end{rem}

~
\begin{rem}
\label{rem:If-sources-are-chosen}If sources are chosen from a random
distribution in $\mathbb{R}^{d}$ rather than located in a small neighborhood
of a low-dimensional manifold, the expected distance between two sources
$\left\Vert \mathbf{y}_{n}-\mathbf{y}_{n'}\right\Vert $ becomes increasingly
large as the dimension $d$ increases (see e.g.\ comments in \cite[Section 1.5.3]{SCOTT:2015}
and examples in \cite{MAR-BIR:2017}). As a result, the functions
of variable $\mathbf{x}$, $\widetilde{K}\left(\mathbf{x},\mathbf{y}_{n}\right)$
and $\widetilde{K}\left(\mathbf{x},\mathbf{y}_{n'}\right)$, are effectively
linearly independent as $d$ becomes large so that in order to have
compressibility, the sources must have a low intrinsic dimension.
Therefore, the assumption that sources are located in a small neighborhood
of a low-dimensional manifold is not specific to our approach. 
\end{rem}

\subsection{Skeleton sources\label{subsec:kernel sum example}}

We illustrate our approach using sources located on a two-dimensional
manifold embedded in a high-dimensional space. For our example, we
generate points $\left\{ \mathbf{y}_{n}\right\} _{n=1}^{N},\,\,\,\mathbf{y}_{n}\in\mathbb{R}^{d}$
so that the first two coordinates are random variables drawn from
the two-dimensional standard normal distribution and the remaining
coordinates are set to zero. Next we apply a random rotation and rescale
the points so that $\left\Vert \mathbf{y}_{n}\right\Vert \leq1$ for
all $n=1,\dots,N$. For targets, we draw points $\left\{ \mathbf{x}_{m}\right\} _{m=1}^{M}$
from the $d$-dimensional standard normal distribution and rescale
them so that $\left\Vert \mathbf{x}_{m}\right\Vert \leq1$. We then
shift the first component of $\left\{ \mathbf{y}_{n}\right\} _{n=1}^{N}$
by $2$ and that of $\left\{ \mathbf{x}_{m}\right\} _{m=1}^{M}$ by
$-2$ so that sources and targets are well separated. Finally, we
select the coefficients of sources, $\left\{ f_{n}\right\} _{n=1}^{N}$,
from the uniform distribution $\mathcal{U}\left(0,1\right)$. In all
tests we set $N=10^{4}$ and $M=10^{3}$. In Table~\ref{tab:Skeleton-source-functions-standard inner prod}
we report the actual minimal and maximum distances between sources
and targets ($\mbox{dist}_{near}$ and $\mbox{dist}_{far}$), the
number of skeleton sources $r_{s}$, and the relative error of the
approximation, 
\begin{equation}
error=\frac{\max_{m=1,\dots,M}\left|g_{m}-\widetilde{g}_{m}\right|}{\max_{m=1,\dots,M}\left|g_{m}\right|},\label{eq:error for kernel summation}
\end{equation}
for selected dimensions $3\le d\le128$. For dimensions $d=3,4,5$,
we use the Poisson kernel $\left\Vert \mathbf{x}-\mathbf{y}\right\Vert ^{-d+2}$
while for $d\geq8$, we use the kernel $\left\Vert \mathbf{x}-\mathbf{y}\right\Vert ^{-1}$
(the fast decay of $\left\Vert \mathbf{x}-\mathbf{y}\right\Vert ^{-d+2}$
results in a negligible interaction between sources and targets in
our setup).

\begin{table}
\begin{centering}
\begin{tabular}{|c|c|c|c|c|}
\hline 
$d$ & $\mbox{dist}_{near}$ & $\mbox{dist}_{far}$ & $r_{s}$ & $error$\tabularnewline
\hline 
\hline 
$3$ & $2.5310$ & $5.6317$ & $22$ & $2.6573e-07$\tabularnewline
\hline 
$4$ & $2.8922$ & $5.3544$ & $27$ & $1.5657e-07$\tabularnewline
\hline 
$5$ & $3.2019$ & $5.0864$ & $29$ & $3.0440e-08$\tabularnewline
\hline 
$8$ & $3.1261$ & $5.1186$ & $23$ & $3.2212e-08$\tabularnewline
\hline 
$16$ & $3.1881$ & $5.1049$ & $25$ & $2.5774e-08$\tabularnewline
\hline 
$32$ & $3.5261$ & $5.0242$ & $25$ & $1.1001e-08$\tabularnewline
\hline 
$64$ & $3.6577$ & $4.7520$ & $24$ & $1.4346e-09$\tabularnewline
\hline 
$128$ & $3.8410$ & $4.4825$ & $25$ & $1.7492e-09$\tabularnewline
\hline 
\multicolumn{1}{c}{} & \multicolumn{1}{c}{} & \multicolumn{1}{c}{} & \multicolumn{1}{c}{} & \multicolumn{1}{c}{}\tabularnewline
\end{tabular}
\par\end{centering}
\centering{}\caption{\label{tab:Skeleton-source-functions-standard inner prod}Skeleton
sources selected using the inner product in (\ref{eq:inner product on disk})
in dimension $d$. We report the actual minimal and maximum distances
between sources and targets ($\mbox{dist}_{near}$ and $\mbox{dist}_{far}$),
the number of skeleton sources $r_{s}$, and the relative error in
(\ref{eq:error for kernel summation}).}
\end{table}

\subsection{Equivalent sources}

In this example, we consider a similar setting as in Section~\ref{subsec:kernel sum example}
for $d=2,3.$ We want to replace true sources $\left\{ \mathbf{y}_{n}\right\} _{n=1}^{N}$
located inside a ball by equivalent sources on its boundary as to
reproduce their interaction with the targets within a selected accuracy
$\epsilon$. We expect the number of equivalent sources on the boundary
to be significantly smaller than the number of original true sources
so that pairwise interactions with targets can be computed rapidly.
We note that such strategy is used in many numerical algorithms (see
e.g.\ \cite{YI-BI-ZO:2004}) and here we demonstrate that our reduction
algorithm can solve this problem.

We combine an initial set of candidate equivalent sources (note that
their number will be reduced by the procedure) with the true sources
and compute the Cholesky decomposition of their Gram matrix. We use
Algorithm~\ref{alg:Reduction-algorithm-using-Gram} with the inner
product defined in (\ref{eq:inner product on disk}) and modify the
pivoting strategy to first pivot only among the candidate equivalent
sources until we run out of significant pivots (i.e., pivots above
the accuracy $\epsilon$); only then we switch to pivot among the
true sources. Finally, we compute new coefficients in the usual way
(see Algorithm~\ref{alg:Reduction-algorithm-using-Gram}) noting
that, initially, the candidate equivalent sources had zero coefficients.
This approach allows us to (i) obtain the minimal number of equivalent
sources and (ii) remove as many of the true sources as possible (we
do not preclude the possibility of some of the true sources to remain).

To examine the performance of our approach, we draw source and target
points from the $d$-dimensional standard normal distribution (where
$d=2,3$), rescale and translate these points so that targets are
located in a ball of radius $1$ centered at $\mathbf{x}_{c}$ and
sources are located in a ball of radius $0.9$ centered at $\mathbf{y}_{c}$.
We choose $\mathbf{x}_{c}=\left(-2,0\right)$, $\mathbf{y}_{c}=\left(2,0\right)$
for $d=2$ and $\mathbf{x}_{c}=\left(-2,0,0\right),\mathbf{y}_{c}=\left(2,0,0\right)$
for $d=3$ to make sure sources and targets are well separated. Next
we pick locations for the candidate equivalent sources on the surface
of the ball of radius $1$ centered at $\mathbf{y}_{c}$. In dimension
$d=2$, we pick
\[
\mathbf{z}_{k}=\mathbf{y}_{c}+\left(\cos\theta_{k},\sin\theta_{k}\right),\ \ \ k=1,\dots,K,
\]
where the angles $\theta_{k}$ are equally spaced on $\left[0,2\pi\right]$
with step size $\frac{2\pi}{K}$. In dimension $d=3$ we pick
\[
\mathbf{z}_{kl}=\mathbf{y}_{c}+\left(\cos\theta_{k}\sin\phi_{l},\sin\theta_{k}\sin\phi_{l},\cos\phi_{l}\right),\ \ \ k=1,\dots,K,\ l=1,\dots,L,
\]
where the angles $\theta_{k}$ are equally spaced on $\left[0,2\pi\right]$
with step size $\frac{2\pi}{K}$ and the angles $\phi_{l}$ are the
Gauss-Legendre nodes on $\left[0,\pi\right]$. In our experiments
we choose a relatively small number of true sources and targets ($N,M=1000$)
so that the result can be clearly visualized (see Figure~\ref{fig:equivalent sources d=00003D2}
and \ref{fig:equivalent sources d=00003D3}). Note that the number
of sources can be significantly higher since the algorithm is linear
in this parameter. We demonstrate the results in Figure~\ref{fig:equivalent sources d=00003D2}
and \ref{fig:equivalent sources d=00003D3}, where we display the
original sources and targets and indicate both, candidate equivalent
sources and selected equivalent sources obtained by Algorithm~\ref{alg:Reduction-algorithm-using-Gram}.

\begin{figure}[h]
\noindent\begin{minipage}[t]{1\columnwidth}%
\begin{center}
\includegraphics[scale=0.5]{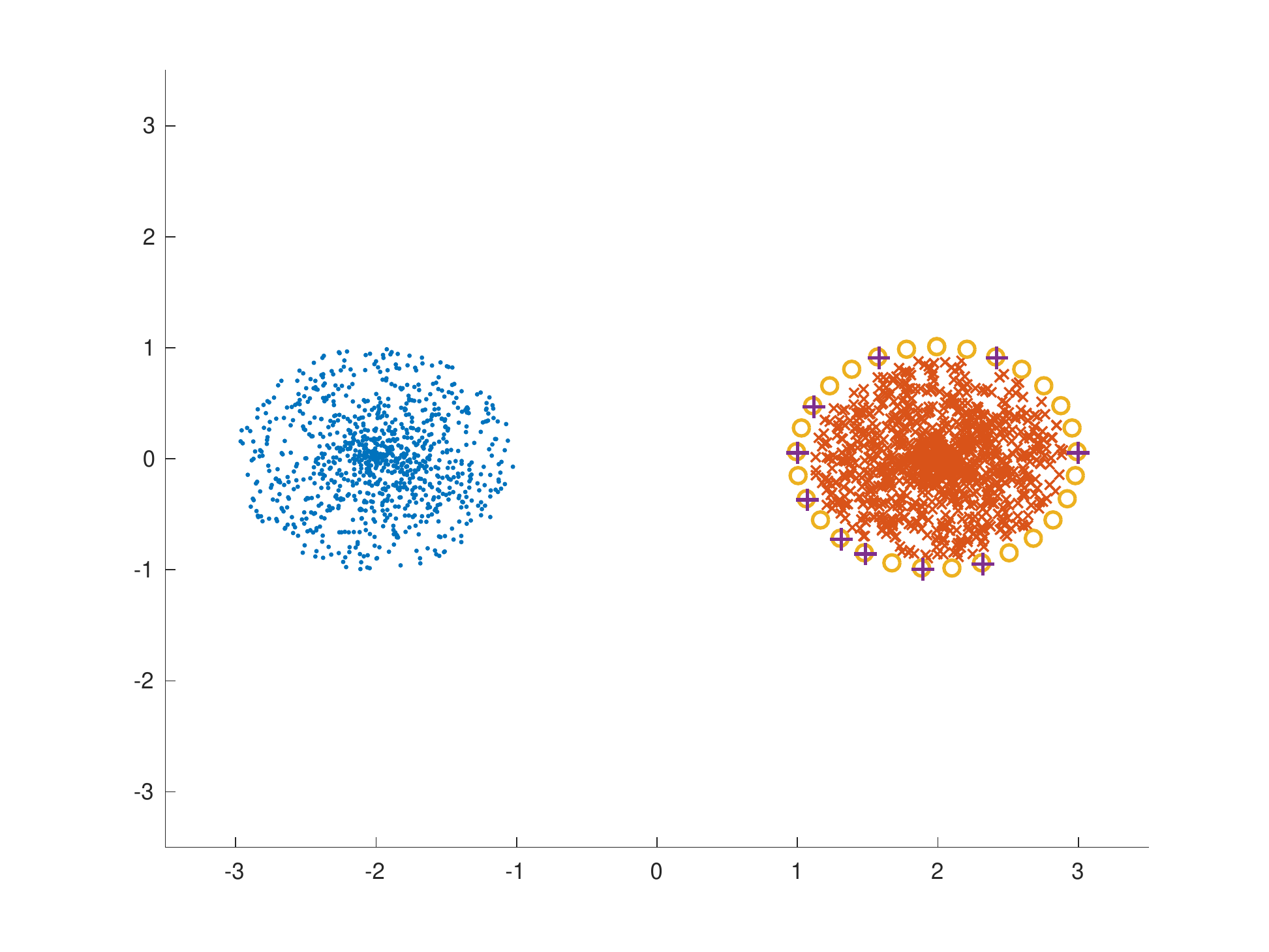}\caption{\label{fig:equivalent sources d=00003D2}Example in dimension $d=2$.
We display $M=1000$ targets (marked with a dot on the left), $N=1000$
sources (marked with an ``x'' on the right), and $K=30$ candidate
sources (marked with a circle on the right). Algorithm~\ref{alg:Reduction-algorithm-using-Gram}
selects $10$ equivalent sources from the $30$ candidate sources
(marked with a $+$). The relative approximation error in (\ref{eq:error for kernel summation})
is $1.3e-07$.}
\par\end{center}%
\end{minipage}\vfill{}
\noindent\begin{minipage}[t]{1\columnwidth}%
\begin{center}
\includegraphics[scale=0.5]{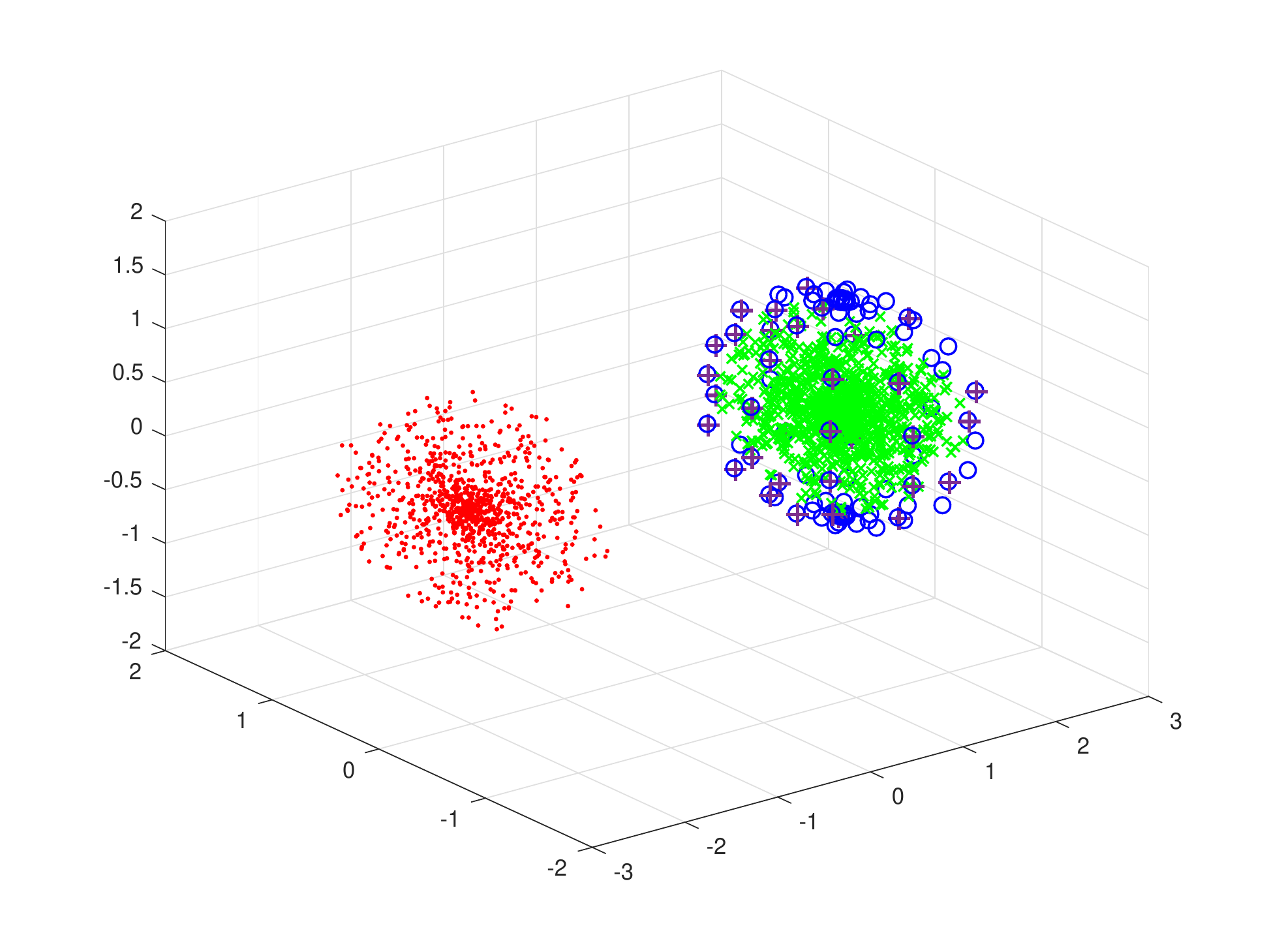}\caption{\label{fig:equivalent sources d=00003D3}Example in dimension $d=3$.
We display $M=1000$ targets (marked with a dot on the left), $N=1000$
sources (marked with an ``x'' on the right), and $K\times L=10\times10$
candidate sources (marked with a circle on the right). Algorithm~\ref{alg:Reduction-algorithm-using-Gram}
selects $35$ equivalent sources from the $100$ candidate sources
(marked with a $+$). The relative approximation error in (\ref{eq:error for kernel summation})
is $8.7e-08$.}
\par\end{center}%
\end{minipage}
\end{figure}

\subsection{Partitioning of points into groups}

Reduction Algorithm~\ref{alg:Reduction-algorithm-using-Gram} can
be used to subdivide scattered points into groups. Indeed, if a set
of points (seeds) are specified beforehand then, like in Voronoi decomposition,
all points can be split into groups by their proximity to the seeds,
i.e.\ a point belongs to a group associated with a given seed if
it is the closest to it among all seeds. There are several algorithms,
e.g.\ Lloyd's algorithm \cite{LLOYD:1982}, that use such seeds in
an iterative procedure to optimize some properties of the sought subdivision.
The question then becomes how to choose such seeds. In order to avoid
poor clusterings the so-called k-means++ algorithm is often used \cite{ART-VAS:2007}.
We would like to point out that Algorithm~\ref{alg:Reduction-algorithm-using-Gram}
can be used to generate initial seeds using linear dependence (which
is a proxy for distances between points). We only illustrate its potential
use for selecting seeds and provide no comparison with k-means++ or
spectral clustering algorithms (see e.g.\ \cite{NG-JO-WE:2002} and
references therein). We also do not provide a comparison with model-based
clustering (see e.g.\ \cite{LI-RA-LI:2007,MCNICH:2016}). Using Algorithm~\ref{alg:Reduction-algorithm-using-Gram}
to subdivide scattered points into groups requires further analysis
and we plan to address it elsewhere.

For our experiment, we use the same set of points as in Section~\ref{subsec:An-example-in-high-dimensions}
and associate with each point a Gaussian centered at that point. The
scale parameter of the Gaussians can be selected sufficiently large
(so that the Gaussian is sufficiently flat) to cover the whole set
of points. We can then use Algorithm~\ref{alg:Reduction-algorithm-using-Gram}
to select the seeds. In order to force Algorithm~\ref{alg:Reduction-algorithm-using-Gram}
to select a specific first point, we introduce an additional point
as the mean
\[
\overline{\mathbf{x}}=\frac{1}{N}\sum_{i=1}^{N}\mathbf{x}_{i},
\]
and associate it with an additional Gaussian which we place at the
beginning of the list of Gaussian atoms.

The seeds are the first significant pivots produced by the algorithm
and our choice of their number depends on the goals of the subdivision.
By its nature, Algorithm~\ref{alg:Reduction-algorithm-using-Gram}
tends to push these seeds far away from each other. We observe that
groups with a small number of points appear to contain outliers (see
Figure~\ref{fig:Subdivision-of-points}), so that the resulting subdivision
can be helpful in identifying them. Since the computational cost of
Algorithm~\ref{alg:Reduction-algorithm-using-Gram} is $\mathcal{O}\left(r^{2}N+p\left(d\right)r\,N\right)$,
where $N$ is the original number of points and $r$ is the number
of seeds, as long as the number of groups we are seeking is small,
this algorithm is essentially linear. We note that we can subdivide
the resulting groups further and, in a hierarchical fashion, build
a tree structure. In this paper we simply illustrate the use of Algorithm~\ref{alg:Reduction-algorithm-using-Gram}
for subdivision of points into groups and plan to develop applications
of this approach elsewhere. 

For the example in Figure~\ref{fig:Subdivision-of-points} we use
the two dimensional distribution of points described in Section~\ref{subsec:An-example-in-high-dimensions}.
We choose the bandwidth parameter $h=200$ when selecting 4 seeds
and $h=16$ when selecting 10 seeds in order to obtain the corresponding
subdivisions of the set. Observe that outliers tend to be associated
with linearly independent terms and, thus, form a subset with a small
number of points. 

\begin{figure}[h]
\begin{centering}
\includegraphics[scale=0.7]{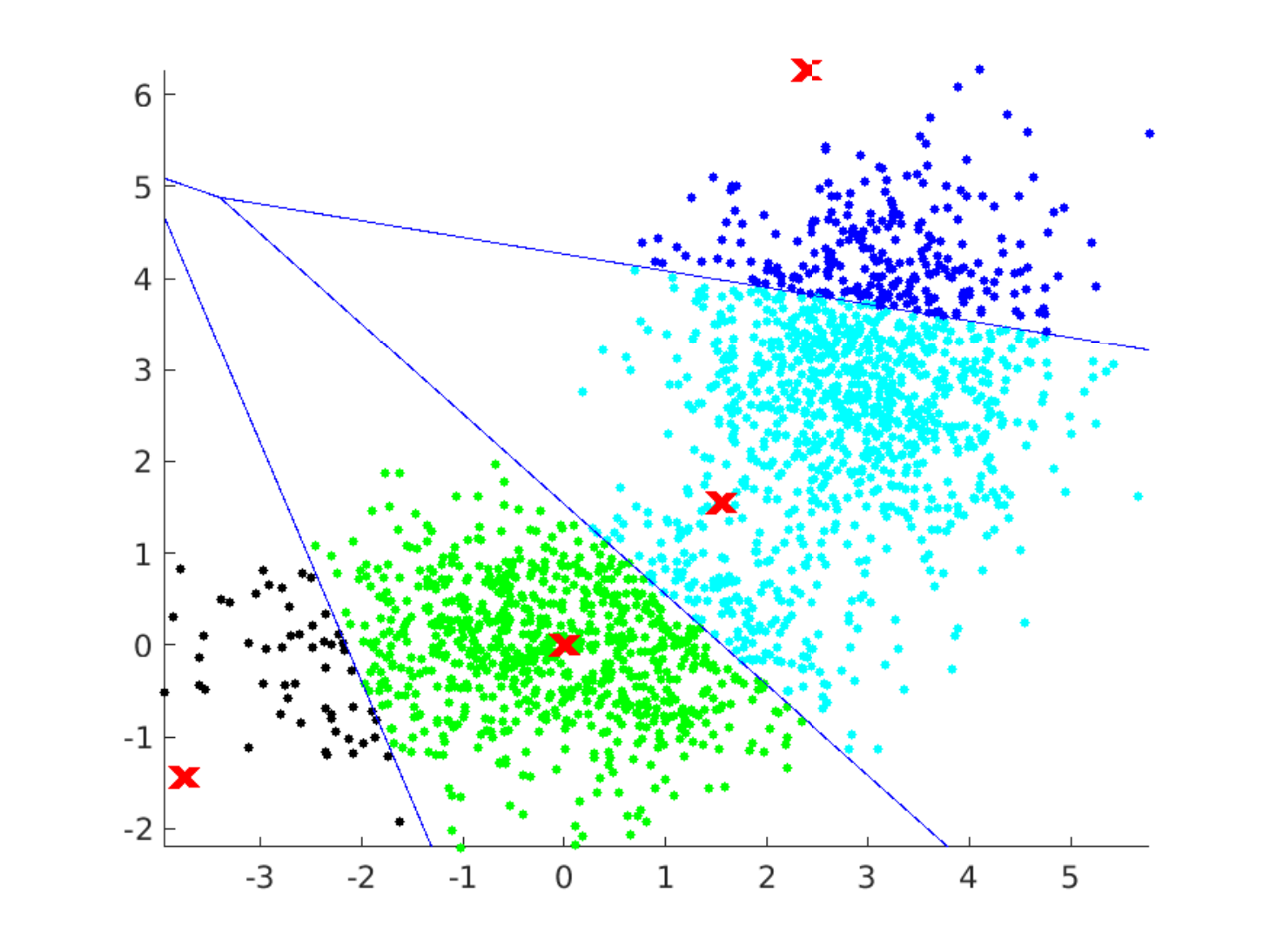}
\par\end{centering}
\begin{centering}
\includegraphics[scale=0.7]{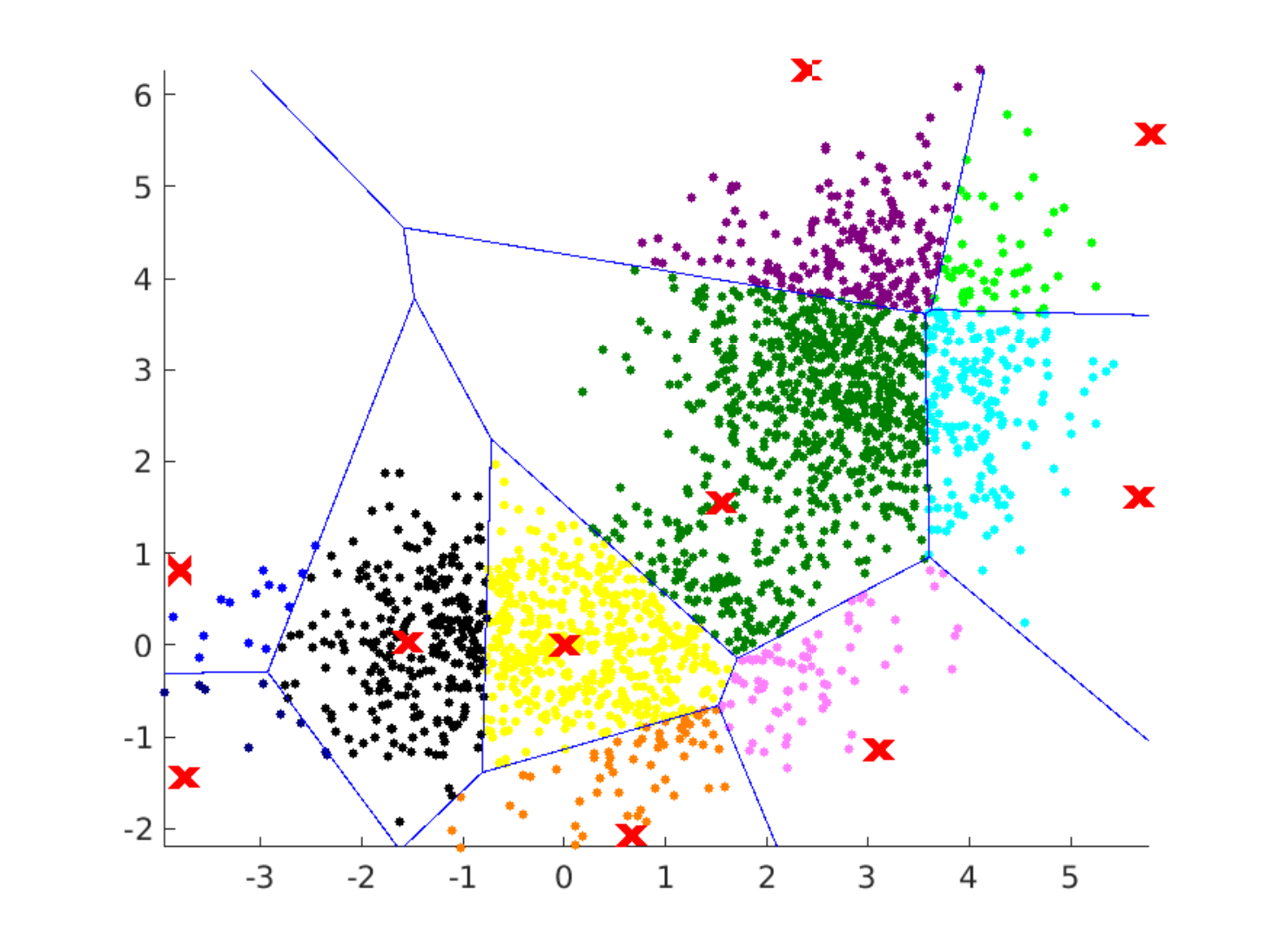}
\par\end{centering}
\caption{\label{fig:Subdivision-of-points}Subdivision of $2,000$ points into
groups using seeds (marked by ``x'') produced by Algorithm~\ref{alg:Reduction-algorithm-using-Gram}.
Illustrated are subdivisions into four groups (top) and into ten groups
(bottom). Note that groups with a small number of points are likely
to contain outliers.}
\end{figure}

\section{\label{sec:Conclusions-and-further}Conclusions and further work}

We presented fast algorithms for reducing the number of terms in non-separated
multivariate mixtures, analyzed them, and demonstrated their performance
on a number of examples. These reduction algorithms allow us to work
with non-separated multivariate mixtures which are a far reaching
generalization of multivariate separated representations \cite{BEY-MOH:2002,BEY-MOH:2005,BE-GA-MO:2009}
and can be used as a tool for solving a variety of multidimensional
problems. Further work is required to develop new numerical methods
that use non-separated multivariate mixtures in applications, for
example in quantum chemistry. We plan to pursue several multivariate
problems with the techniques illustrated in this paper. Specifically,
we plan to develop further our approach to KDE and compare it with
existing techniques. We also plan to pursue the problem of hierarchical
subdivision of point clouds into groups and its applications to clustering
and detection of outliers.

\section*{Acknowledgements}

We would like to thank Zydrunas Gimbutas (NIST) and the anonymous
reviewers for their many useful comments and suggestions.

\section{\label{sec:Appendix-A}Appendix A}

As mentioned in the paper, computing with multivariate Gaussian mixtures
is particularly convenient since all common operations result in explicit
integrals. We present below the key identities for multivariate Gaussians
using the standard $L^{1}$ normalization,

\[
N\left(\mathbf{x},\boldsymbol{\mu},\boldsymbol{\Sigma}\right)=\frac{1}{\det\left(2\pi\boldsymbol{\Sigma}\right)^{1/2}}\exp\left(-\frac{1}{2}\left(\mathbf{x}-\boldsymbol{\mu}\right)^{T}\boldsymbol{\Sigma}^{-1}\left(\mathbf{x}-\boldsymbol{\mu}\right)\right).
\]
However, when computing integrals with Gaussians atoms, it is convenient
to normalize them to have unit $L^{2}$-norm. 

\subsubsection{Convolution of two normal distributions}

\begin{equation}
\int_{\mathbb{R}^{d}}N\left(\mathbf{x}-\mathbf{y},\boldsymbol{\mu}_{1},\boldsymbol{\Sigma}_{1}\right)N\left(\mathbf{y},\boldsymbol{\mu}_{2},\boldsymbol{\Sigma}_{2}\right)d\mathbf{y}=N\left(\mathbf{x},\boldsymbol{\mu}_{1}+\boldsymbol{\mu}_{2},\boldsymbol{\Sigma}_{1}+\boldsymbol{\Sigma}_{2}\right).\label{eq:Convolution of two normal distributions}
\end{equation}

\subsubsection{Sum of two quadratic forms }

Consider vectors $\mathbf{x}$, $\mathbf{a}$, and $\mathbf{b}$ and
two symmetric positive definite matrices $\mathbf{A}$ and $\mathbf{B}$.
We have 
\[
\left(\mathbf{x}-\mathbf{a}\right)^{T}\mathbf{A}\left(\mathbf{x}-\mathbf{a}\right)+\left(\mathbf{x}-\mathbf{b}\right)^{T}\mathbf{B}\left(\mathbf{x}-\mathbf{b}\right)=\left(\mathbf{x}-\mathbf{c}\right)^{T}\left(\mathbf{A}+\mathbf{B}\right)\left(\mathbf{x}-\mathbf{c}\right)+\left(\mathbf{a}-\mathbf{b}\right)^{T}\mathbf{C}\left(\mathbf{a}-\mathbf{b}\right),
\]
where
\[
\mathbf{c}=\left(\mathbf{A}+\mathbf{B}\right)^{-1}\left(\mathbf{A}\mathbf{a}+\mathbf{B}\mathbf{b}\right)
\]
and
\[
\mathbf{C}=\mathbf{A}\left(\mathbf{A}+\mathbf{B}\right)^{-1}\mathbf{\mathbf{B}}=\left(\mathbf{A}^{-1}+\mathbf{\mathbf{B}}^{-1}\right)^{-1}.
\]

\subsubsection{Product of two normal distributions}

We have

\begin{equation}
N\left(\mathbf{x},\boldsymbol{\mu}_{1},\boldsymbol{\Sigma}_{1}\right)N\left(\mathbf{x},\boldsymbol{\mu}_{2},\boldsymbol{\Sigma}_{2}\right)=N\left(\boldsymbol{\mu}_{1},\boldsymbol{\mu}_{2},\boldsymbol{\Sigma}_{1}+\boldsymbol{\Sigma}_{2}\right)\cdot N\left(\mathbf{x},\boldsymbol{\mu}_{c},\left(\boldsymbol{\Sigma}_{1}^{-1}+\boldsymbol{\Sigma}_{2}^{-1}\right)^{-1}\right)\label{eq:product}
\end{equation}
where
\[
\boldsymbol{\mu}_{c}=\left(\boldsymbol{\Sigma}_{1}^{-1}+\boldsymbol{\Sigma}_{2}^{-1}\right)^{-1}\left(\boldsymbol{\Sigma}_{1}^{-1}\boldsymbol{\mu}_{1}+\boldsymbol{\Sigma}_{2}^{-1}\boldsymbol{\mu}_{2}\right).
\]

\subsubsection{\label{subsec:Inner-product-of}Inner product of two normal distributions}

It follows that

\begin{equation}
\int_{\mathbb{R}^{d}}N\left(\mathbf{x},\boldsymbol{\mu}_{1},\boldsymbol{\Sigma}_{1}\right)N\left(\mathbf{x},\boldsymbol{\mu}_{2},\boldsymbol{\Sigma}_{2}\right)d\mathbf{x}=N\left(\boldsymbol{\mu}_{1},\boldsymbol{\mu}_{2},\boldsymbol{\Sigma}_{1}+\boldsymbol{\Sigma}_{2}\right).\label{eq:inner product of Gaussians}
\end{equation}
Indeed, from (\ref{eq:Convolution of two normal distributions}) we
have 
\begin{eqnarray*}
\int_{\mathbb{R}^{d}}N\left(\mathbf{x},\boldsymbol{\mu}_{1},\boldsymbol{\Sigma}_{1}\right)N\left(\mathbf{x},\boldsymbol{\mu}_{2},\boldsymbol{\Sigma}_{2}\right)d\mathbf{x} & = & \int_{\mathbb{R}^{d}}N\left(2\boldsymbol{\mu}_{1}-\mathbf{x},\boldsymbol{\mu}_{1},\boldsymbol{\Sigma}_{1}\right)N\left(\mathbf{x},\boldsymbol{\mu}_{2},\boldsymbol{\Sigma}_{2}\right)d\mathbf{x}\\
 & = & N\left(2\boldsymbol{\mu}_{1},\boldsymbol{\mu}_{1}+\boldsymbol{\mu}_{2},\boldsymbol{\Sigma}_{1}+\boldsymbol{\Sigma}_{2}\right)\\
 & = & N\left(\boldsymbol{\mu}_{1},\boldsymbol{\mu}_{2},\boldsymbol{\Sigma}_{1}+\boldsymbol{\Sigma}_{2}\right).
\end{eqnarray*}
Alternatively, from (\ref{eq:product}) we have 
\begin{eqnarray*}
\int_{\mathbb{R}^{d}}N\left(\mathbf{x},\boldsymbol{\mu}_{1},\boldsymbol{\Sigma}_{1}\right)N\left(\mathbf{x},\boldsymbol{\mu}_{2},\boldsymbol{\Sigma}_{2}\right)d\mathbf{x} & = & N\left(\boldsymbol{\mu}_{1},\boldsymbol{\mu}_{2},\boldsymbol{\Sigma}_{1}+\boldsymbol{\Sigma}_{2}\right)\int_{\mathbb{R}^{d}}N\left(\mathbf{x},\boldsymbol{\mu}_{c},\left(\boldsymbol{\Sigma}_{1}^{-1}+\boldsymbol{\Sigma}_{2}^{-1}\right)^{-1}\right)d\mathbf{x}\\
 & = & N\left(\boldsymbol{\mu}_{1},\boldsymbol{\mu}_{2},\boldsymbol{\Sigma}_{1}+\boldsymbol{\Sigma}_{2}\right).
\end{eqnarray*}

\bibliographystyle{plain}


\end{document}